\documentclass[aap]{imsart}

\RequirePackage{amsthm,amsmath,amsfonts,amssymb}
\RequirePackage[numbers]{natbib}
\RequirePackage[colorlinks,citecolor=blue,urlcolor=blue]{hyperref}
\RequirePackage{graphicx}

\usepackage{tikz}
\usetikzlibrary{backgrounds}
\usetikzlibrary{patterns,fadings}
\usetikzlibrary{arrows,decorations.pathmorphing}
\usetikzlibrary{decorations}
\usetikzlibrary{calc}
\usetikzlibrary{shapes.misc}

\definecolor{light-gray}{gray}{0.95}
\usepackage{float,bbm,mathrsfs}

\def\centerarc[#1](#2)(#3:#4:#5){\draw[#1] ($(#2)+({#5*cos(#3)},{#5*sin(#3)})$) arc (#3:#4:#5);}

\startlocaldefs
\numberwithin{equation}{section}

\theoremstyle{plain}
\newtheorem{theorem}{Theorem}
\newtheorem{lemma}{Lemma}
\newtheorem{proposition}{Proposition}

\theoremstyle{remark}
\newtheorem{remark}{Remark}
\newtheorem{definition}{Definition}

\newtheorem{assumption}{Assumption}

\newcommand{\mc}[1]{{\mathcal #1}}

\newcommand{\bb}[1]{{\mathbb #1}}

\newcommand{\<}{\langle}
\renewcommand{\>}{\rangle}

\renewcommand{\epsilon}{\varepsilon}

\newcommand{\R}{\mathbb R}

\newcommand{\Z}{\mathbb Z}
\newcommand{\N}{\mathbb N}
\renewcommand{\P}{\mathbb P}
\newcommand{\T}{\mathbb T}
\newcommand{\E}{\mathbb E}

\renewcommand{\bar}{\overline}
\renewcommand{\tilde}{\widetilde}
\renewcommand{\hat}{\widehat}

\newcommand{\supp}{{\rm{supp}}}

\newcommand{\norm}[1]{\left\vert\left\vert #1 \right\vert\right\vert}
\newcommand{\normm}[1]{{\left\vert\kern-0.1ex\left\vert\kern-0.1ex\left\vert\; #1 \; \right\vert\kern-0.1ex\right\vert\kern-0.1ex\right\vert}}    
\newcommand{\cro}[1]{\left[#1\right]}
\newcommand{\pa}[1]{\left(#1\right)}

\newcommand{\genex}{\mathscr{L}_N^{\mathrm{FEX}}}
\newcommand{\genzr}{\mathscr{L}_M^{\mathrm{FZR}}}
\newcommand{\tgenzr}{\widetilde{\mathscr{L}}_M^{\mathrm{FZR}}}

\renewcommand{\leq}{\leqslant}
\renewcommand{\geq}{\geqslant}

\newcommand{\ang}[1]{\big\langle #1 \big\rangle}
\newcommand{\bL}{\mathbb{L}}

\newcommand{\mm}[1]{{#1}}
\newcommand{\ccl}[1]{{#1}}											
\newcommand{\lj}[1]{{#1}}

\endlocaldefs

\begin{document}

\begin{frontmatter}
\title{Mapping hydrodynamics for the facilitated exclusion and zero-range processes\thanksref{T1}}
\runtitle{Mapping hydrodynamics for FEP and FZR processes}
\thankstext{T1}{\textsc{Acknowledgments:} \mm{We warmly thank the anonymous referees for their  careful reading and remarks which contributed to a better version of this paper.} This work has been done while M. Simon was appointed at Gran Sasso Science Institute (GSSI), L'Aquila, Italy. This project is partially supported by the ANR grant MICMOV (ANR-19-CE40-0012)
of the French National Research Agency (ANR), and by the European Union with the program FEDER ``Fonds européen de développement régional'' with the Région Hauts-de-France. It has also received funding from the European
Research Council (ERC) under the European Union’s Horizon 2020 research and innovative program
(grant agreement n° 715734), from
Labex CEMPI (ANR-11-LABX-0007-01), and from the Fundamental Research Funds for the Central Universities in China.
\newline
\textsc{Data availability statement}: data sharing not applicable to this article as no datasets were generated or analyzed during the current study.}

\begin{aug}
\author[A]{\fnms{Erignoux}~\snm{Cl\'ement}\ead[label=e1]{clement.erignoux@inria.fr}},
\author[B]{\fnms{Simon}~\snm{Marielle}\ead[label=e2]{msimon@math.univ-lyon1.fr}}
\and
\author[C]{\fnms{Zhao}~\snm{Linjie}\ead[label=e3]{linjie$\_$zhao@hust.edu.cn}}

\address[A]{Inria, Univ. Lille, CNRS, UMR 8524 - Laboratoire Paul Painlev\'e, F-59000 Lille\printead[presep={,\ }]{e1}}
\address[B]{Univ Lyon, CNRS, Université Claude Bernard Lyon 1, UMR 5208, Institut Camille Jordan, F-69622 Villeurbanne, France\printead[presep={,\ }]{e2}}
\address[C]{School of Mathematics and Statistics, Huazhong University of Science and Technology, Wuhan, 430074, China\printead[presep={,\ }]{e3}}

\end{aug}

\begin{abstract}
 We derive the hydrodynamic limit for two degenerate lattice gases, the \emph{facilitated exclusion process} (FEP) and the \emph{facilitated zero-range process} (FZRP), both in the symmetric and the asymmetric case. For both processes, the hydrodynamic limit in the symmetric case takes the form of a diffusive Stefan problem, whereas the asymmetric case is characterized by a hyperbolic Stefan problem. Although the FZRP is attractive, a property that we extensively use to derive its hydrodynamic limits in both cases, the FEP is not. To derive the hydrodynamic limit for the latter, we exploit that of the zero-range process, together with a classical mapping between exclusion and zero-range processes, both at the microscopic and macroscopic level. Due to the degeneracy of both processes, the asymmetric case is a new result, but our work  also provides a simpler proof than the one that was previously proposed for the FEP in the symmetric case in \cite{blondel2021stefan}.
\end{abstract}



\end{frontmatter}


\section{Introduction}

Recently, the mathematical understanding of \emph{free boundary problems} has generated ongoing interest in the scientific community. The so-called \textit{Stefan problem} introduced by J.~Stefan in \cite{Stefan} typically describes the temperature distribution in a homogeneous medium which is subject to a {phase change}. Assuming that heat linearly diffuses, one can write  the evolution equation of both separate phases, with the separation interface allowed to evolve in time. Let us define for instance the liquid region as the domain where $\rho>0$ and the solid region  where $\rho=0$. In one dimension, the mathematical formulation of the Stefan problem is the following: find a curve $x=\Gamma(t)$ and a function $\rho(x,t) \geqslant 0$ such that
\begin{equation}
\label{eq:Stef}
\begin{cases}
& \partial_t \rho = D \partial_{xx}^2\rho, \qquad \text{if } 0 < x < \Gamma(t), \\
& \rho(x,t) = 0, \qquad \;  \; \; \; \text{if } x \geqslant \Gamma(t),
\end{cases} \qquad \text{and} \qquad  \frac{d\Gamma}{dt} = -\partial_{x} \rho(\Gamma(t),t),
\end{equation} plus initial and boundary conditions. Physically speaking, the  second condition on the {free boundary} $\Gamma(t)$  translates the presence of latent heat at the phase transition. The mathematical solution to \eqref{eq:Stef} is usually obtained \textit{via} the weak formulation and the regularity of $\Gamma(t)$ can be analyzed (see \textit{e.g.}~\cite{Meir} for a review).  

Such a macroscopic behavior has been lately derived in \cite{blondel2021stefan} from a one-dimensional underlying microscopic system of interacting particles, whose dynamics is generated by a Markov jump process with degenerate jump rates.  The approach followed by the authors is based on a mathematical procedure  called \emph{hydrodynamic limit}, \textit{i.e.}~the macroscopic behavior is obtained via a long-time and large-space scaling limit (see for instance \cite{klscaling} for a review). The underlying microscopic model is called \emph{facilitated exclusion process} (FEP), and belongs to the class of exclusion processes, with kinetic constraints, which have attracted a lot of interest in recent years due to their rich and complex behavior. In the FEP, particles are spread on the one-dimensional discrete lattice, satisfying the exclusion rule which authorizes at most one particle per site. Each particle performs random jumps to one of its two neighboring sites, providing that: (i) the neighboring site to which the jump is directed is not occupied by another particle, (ii) the other neighbor \emph{is} occupied by a particle. If particles choose one of the two neighboring sites with equal probability, the FEP is called \emph{symmetric}, otherwise \emph{asymmetric}. As a result of the second hard constraint (ii), the FEP exhibits two phases, one \lj{called} \emph{frozen}, where particles quickly stop moving, and the other one \emph{active}, where particles adopt a diffusive behavior. Both phases are separated at a critical value for the density of particles given by $\rho_c=\frac12$. 
The facilitated exclusion process has been widely explored in the recent years after being introduced by \cite{RossiPastorVespignani00} in 2000. On one hand, in the physics literature, the critical behaviors, such as the critical density and the critical exponents, in different dimensions have been investigated in \cite{BasuMohanty09,oliveira,lubeck}. In \cite{GabelKrapivskyRedner10} the authors found the phenomenon of jump continuity at the leading edge of rarefaction waves, which is quite different from asymmetric simple exclusion. They also formally derive the macroscopic behavior of the particle density when the initial profile is a \emph{step} function (with two constant densities $\rho_{\pm}$ to the right and to the left of the origin), and they predict the \emph{shock waves} phenomenon, which is a consequence of the hydrodynamic limit proved in this paper.   On the other hand, in the mathematics literature, the stationary states of the facilitated exclusion, either in the continuous or discrete time setting, have been studied in \cite{Ayyer2020StationarySO,Chen2019,2019JSMTE,Goldstein2020,Goldstein2022}. Limit theorems have also been proved in  \cite{BaikBarraquandCorwinToufic16} for the position of the rightmost particle starting from step initial condition.

The  \emph{symmetric FEP} has recently  been under significant scrutiny. For instance, the diffusive hydrodynamics of the supercritical phase ($\rho>\frac12$) had   been  investigated in \cite{blondel2020hydrodynamic}. Then, in \cite{blondel2021stefan}, it has been proved that, 
starting from an initial density profile $\rho^{\rm ini}$ with both supercritical and
subcritical regions, the hydrodynamics for the symmetric FEP is given by a Stefan problem:
the diffusive supercritical phase progressively invades the subcritical phase via flat interfaces, until either
one of the phases disappears, and the macroscopic density $\rho$ (obtained as the limit of the empirical density in the diffusive space-time scaling) is the unique weak solution to 
\begin{equation}
\label{eq:stefanintro}
\partial_t \rho = \partial_{xx}\big( \tfrac{2\rho-1}\rho \mathbbm{1}_{\{\rho>\frac12\}}\big), \qquad \rho|_{t=0}=\rho^{\rm ini}.
\end{equation} 
One of the important tools used in \cite{blondel2020hydrodynamic, blondel2021stefan} to derive the macroscopic limit of the symmetric FEP is a classical mapping between exclusion processes and zero-range processes (which can be implemented in both symmetric and asymmetric cases). In the latter, particles are spread on the one-dimensional discrete lattice but there is no constraint on the number of particles per site. More precisely, one can map the exclusion configuration of particles $\eta\in\{0,1\}^{\Z}$  to a zero-range configuration $\omega\in\N^\Z$ as follows: in $\eta$, look for the first empty site to the right of or at the origin, label it $1$. Moving to the right, for any $i>0$,  define $\omega(i)$ as the number of particles between the $i$--th and $(i+1)$--th empty sites in $\eta$, and do the same moving to the left in order to define $\omega(i)$ for $i\leqslant 0$. See Figure \ref{fig:mapping} for an illustration. 
Note that this mapping is not one-to-one: 
shifting $\eta$ one site to the left may not change $\omega$. In particular, 
the reverse mapping is only defined up to the position of the empty site with label $1$. Moreover,  through this mapping the correspondence between invariant measures  is not trivial. For only a few dynamics are the properties of the invariant measures   seen from a tagged particle known; \lj{for} the simple exclusion process, for example, they are product Bernoulli, however the dynamics of the FEP is much more complex. In the litterature, this mapping between exclusion and zero-range processes has been exploited in other contexts. In \cite{Kipnis86} for instance, the author used it in order to prove a central limit theorem for a tagged particle in the asymmetric simple exclusion process (ASEP); while in \cite{FunakiSasada}, the weakly asymmetric zero-range process with a
stochastic reservoir at the boundary (associated with the dynamics of two-dimensional Young diagrams) is mapped to the weakly asymmetric
simple exclusion process on the full line without any boundary condition, for which the hydrodynamic limit is known.

 In \cite{blondel2021stefan}, the above mapping has been used for the symmetric FEP in order to prove an ergodic decomposition of any infinite volume stationary measure \textit{à la} De Finetti, but has not been used directly to derive the hydrodynamic limit, which has been proved using Funaki's strategy based on Young measures \cite{funaki1999free}.  In fact, using the transformation described above, the exclusion process $\{\eta_t\}_{t\geqslant 0}$ can be coupled with a \emph{facilitated zero-range process} (FZRP) $\{\omega_t\}_{t\geqslant 0}$ as follows: whenever a particle jumps in the process $\eta_t$, a particle in the corresponding pile in $\omega_t$ jumps in the same direction.  In particular, a jump of a particle to an empty site is allowed in the FEP if and only if the corresponding pile in the FZRP has at least two particles. Then, one can easily check that $\{\omega_t\}_{t\geq 0}$ is a  zero-range process (symmetric or asymmetric) with the jump rate function: $
g(\omega):= \mathbbm{1}_{\{\omega (0) \geqslant 2\}}.$ 
 Working with the FZRP has two main advantages: first, the invariant measures for this process are simpler, they are geometric product measures; second, it has the so-called \emph{attractiveness property}, which permits to use several powerful tools, as noted in \cite{rezakhanlou91} for instance.
As a result, its hydrodynamic limit can be derived, both in the symmetric and asymmetric case, following two steps: (i) first, obtain the one-block estimate in the context of Stefan problems, thanks to the decomposition of translation invariant stationary states, as implemented by Funaki in \cite{funaki1999free}; (ii) second, derive the two-blocks estimate using the attractiveness property, as given in Rezakhanlou \cite{rezakhanlou91}. Working a little, one can then obtain that: \begin{itemize} \item For the symmetric FZRP, in the \emph{diffusive} space-time scaling, the empirical density profile converges towards the unique weak solution to a Stefan problem:
\begin{equation}
\label{eq:StefanZR}
\partial_t \alpha = \partial_{xx}\big(\tfrac{\alpha-1}{\alpha} \mathbbm{1}_{\alpha >1}\big),
\end{equation}
starting from some suitable initial profile $\alpha^{\rm ini}$.
\item For the asymmetric FZRP with asymmetry bias $p \in (\frac12,1]$, in the \emph{hyperbolic} space-time scaling, the empirical density profile converges towards the unique entropy solution to the hyperbolic Stefan problem 
\begin{equation}
\label{eq:StefanZRasym}
\lj{\partial_t\alpha + (2p-1) \partial_x\big( \tfrac{\alpha-1}{\alpha} \mathbbm{1}_{\alpha >1}\big) = 0,}
\end{equation}
starting from $\alpha^{\rm ini}$.
\end{itemize}
Note however that even using attractiveness, because of the two phased nature of the process, deriving its hydrodynamic is not completely straightforward.

\medskip

A natural strategy to derive the hydrodynamic limits for the FEP is then to  deduce it from the one of the FZRP, for instance to deduce \eqref{eq:stefanintro}  from \eqref{eq:StefanZR} in the symmetric case. However,  new difficulties appear. First, one needs to carefully write the corresponding macroscopic mapping so as to \ccl{show that the image through the mapping of the solution $\alpha$ of \eqref{eq:StefanZR} is the solution $\rho$ of \eqref{eq:stefanintro}}, which is far from trivial because of the different space scales, and because the solutions to \eqref{eq:stefanintro} and \eqref{eq:StefanZR} are not regular -- in particular they have discontinuities \ccl{at the interfaces between active and frozen phases}. Second, it is not straightforward to deduce a hydrodynamic limit  through \lj{the mapping between FEP and FZRP}, because the initial \emph{local equilibrium} measure for the FEP, which is assumed to be product and fits the initial profile $\rho^{\rm ini}$, is \emph{not} mapped onto a \emph{product} local equilibrium measure for the FZRP. This is problematic because the strategy followed by Rezakhanlou \lj{has} been \lj{developed} only for product initial measures.
  
  The objective of this paper is twofold: \begin{enumerate}
  \item We give a new, simpler proof of the hydrodynamic limit for the symmetric FEP previously derived in \cite{blondel2021stefan}, driven by the Stefan problem \eqref{eq:stefanintro}, by implementing rigorously the above strategy based on the mapping mentioned previously. 
  \item We obtain for the first time the hydrodynamic limit of the asymmetric FEP, by adapting the strategy above to the asymmetric case. If $p\in(\frac12,1]$ denotes the asymmetry bias to the right, we then show that the macroscopic density is the unique entropy solution of the hyperbolic Stefan problem
  \begin{equation}\label{eq:stefanintro-asym} \partial_t \rho + (2p-1) \partial_x \big( \tfrac{(1-\rho)(2\rho-1)}{\rho}\mathbbm{1}_{\rho>\frac12}\big)\ccl{=0}, \qquad \rho|_{t=0}=\rho^{\rm ini}.\end{equation}
  \end{enumerate}
 These two results are contained in Theorem \ref{thm:ex} below.

The technical novelties which are needed with respect to what already exists in the literature are the following: \begin{itemize}
  \item One needs to derive the hydrodynamic limit for the (attractive) FZRP, both in symmetric and asymmetric cases, starting from a local equilibrium measure fitting $\alpha^{\rm ini}$ which is \emph{not product}. Because of the degeneracy of the jump rates, usual entropy tools\footnote{namely entropy method or relative entropy method, as developed in \cite{klscaling}.} cannot be used. \lj{In fact, using only the tools of \cite{rezakhanlou91}, the initial probability distribution being product is fundamental to derive the hydrodynamic limits, because it is used to dominate the process by an equilibrium state and to obtain the two-blocks estimates.  Therefore, the rigourous derivation of \eqref{eq:StefanZR} from a non-product initial measure is already interesting on its own, see Theorem \ref{thm:zr}.}
   \item  One needs to carefully map the solutions to equations \eqref{eq:stefanintro} and \eqref{eq:stefanintro-asym} onto the solutions to \eqref{eq:StefanZR} and \eqref{eq:StefanZRasym}, respectively, taking into account the fact the none of these solutions are regular. This is done by smooth approximations, and relates to the PDE theory of \ccl{Stefan} problems.
  \end{itemize}

\subsection{Outline of the paper} In Section \ref{sec:results} we start by introducing the (symmetric and asymmetric) FEP, as well as the (symmetric and asymmetric) FZRP, and we state the main results about their hydrodynamic limits, namely Theorem \ref{thm:ex} for the macroscopic behavior of the FEP, and Theorem \ref{thm:zr} for the macroscopic behavior of the FZRP. We end this section by an explanation of the strategy of the proof, see Section \ref{ssec:strategy}. In Section \ref{sec:sym} we give the complete proof of the hydrodynamic limit \eqref{eq:stefanintro} for the symmetric FEP, after having rigorously defined both mappings at the microscopic and macroscopic level, and assuming that the hydrodynamic limit for the symmetric FZRP holds (the latter will be proved independently in the last section). In Section \ref{sec:asym} we implement the same strategy in the asymmetric case, in order to derive \eqref{eq:stefanintro-asym}. Since the hydrodynamic equations are pretty different (and so are the time scales), one cannot straightforwardly use the same arguments as in Section \ref{sec:sym}. Finally, \lj{ Section \ref{sec:hydroZR} is} devoted to the proof of Theorem \ref{thm:zr}, namely the hydrodynamic limit for the FZRP both in the symmetric and asymmetric cases, and starting from an initial probability measure which is not necessarily product.

\subsection{General notations}
In this article, we consider  systems of particles which are either symmetric and evolve  on the one-dimensional finite ring, or asymmetric 	and evolve on the infinite line. If $N\in\N$ is the scaling parameter we denote by $\bb{L}_N$ the corresponding discrete space which is given by
\begin{itemize}
	\item $\bb{L}_N=\T_N =\Z/N\Z$, namely the discrete ring of size $N$, \lj{for the symmetric case}; 
	\item $\bb{L}_N = \bb{Z}$, namely the infinite discrete line, \lj{for the asymmetric case}. Note that in this case the lattice does not depend on the scaling parameter, but \lj{since the rescaled process does} (see below), we keep $N$ in the notation in order to indicate that it corresponds to the discrete setting.
\end{itemize}
We denote by $\bL$ the continuous limit of the discrete lattice $\bL_N$, namely 
\[ \bL:= \begin{cases}  \T:=[0,1]=\R/\Z  & \text{ if } \bL_N=\T_N\\ \R  & \text{ if }\bL_N=\Z. \end{cases}\] Given two functions $f, g$ in $L^2(\bL)$, we denote by 
\[\big\langle  f, g \big\rangle:=\int_\bL f(u) g(u) du,\]
the standard scalar product of $f$ and $g$. 

Throughout, we will alternate between two particle systems: one being a \emph{zero-range} process, the other one being an \emph{exclusion} process. Therefore we find it convenient to introduce distinct notations for these two processes, which we summarize here (all precise definitions will be given  in the next sections):
 \begin{center}
\begin{tabular}{l|c|cc}
& \textsc{exclusion} & \textsc{zero-range} &\\
Scaling parameter & $N$& $M$ &\\
Discrete configuration & $\eta_x, \; x \in\bb{L}_N$ & $\omega_y, \; y \in \bb{L}_M$ & \\
Law of the process & $\bb{P}_\nu$ & $\bf P_\mu$ & ($\nu,\mu$ being initial prob. measures)\\
Macroscopic density & $\rho(u), \; u \in\bb{L}$ & $\alpha(v), \; v \in\bb{L}$  & 
\end{tabular}
\end{center}
\medskip

Finally, we will work with the following spaces of functions. Let $U,V$ be two open subsets of $\R^d$. Then, $C^{k,\ell}(U\times V)$ is the set of functions $f:U\times V \to R$ which are $C^k$ (resp.~$C^\ell$) regular in the first (resp.~second) variable; $C_c(U)$ (resp.~$C^\infty(U)$) is the space of compactly supported (resp.~smooth) real-valued functions defined on $U$; $L_{\rm loc}^1(U)$ is the space of locally integrable functions defined on $U$. The usual $L^p(U)$--spaces are endowed with their norm denoted by $\|\cdot\|_{L^p(U)}$, $p\in[1,+\infty]$. Finally,  a function $f:\R \rightarrow \R$ is of bounded variation if the total variation of $f$ defined as below is finite:
	\[TV(f) := \sup \Big\{ \int_\R f(u) g^\prime (u) du: g \in C_c^\infty (\R), ||g||_{L^\infty (\R)} \leq 1 \Big\}\]  and we denote by $BV(\R)$ the set of functions $f:\R\to\R$ which are of bounded variation.

\section{Main results} \label{sec:results}

\subsection{Hydrodynamic limit for the facilitated exclusion process}

Let us start by defining the microscopic particle system. We denote by $N\in \N$ the scaling parameter for the exclusion process. Its \emph{particle configurations} $\eta$ are sequences of $0$'s and $1$'s indexed by $\bb{L}_N$, namely  $\eta_x=1$ if and only if  site $x \in \bL_N$ is occupied by a particle. The \emph{facilitated exclusion process on $\bL_N$} is a Markov process on the set of configurations $ \Sigma_N:=\{0,1\}^{\bb{L}_N}$. 
The infinitesimal generator ruling the evolution in time of this Markov process is given by $\genex$, which acts on  functions $f:\Sigma_N \to \R$ as 
\begin{equation}
\label{eq:DefLN}
\genex f(\eta):=\sum_{x\in\bb{L}_N} c_{x,x+1}(\eta)\big(f(\eta^{x,x+1})-f(\eta)\big),
\end{equation}
where $\eta^{x,x'}$ denotes the configuration obtained from $\eta$ by swapping the values at sites $x$ and $x'$,
\[
\eta_z^{x,x'}=
\begin{cases}
\eta_{x'} & \mbox{ if } z=x,\\
\eta_x & \mbox{ if } z=x',\\
\eta_z & \mbox{ otherwise.}
\end{cases}
\]
The jump rates $c_{x,x'}(\eta)$ encode two dynamical constraints: 
\begin{enumerate}
\item [i.] the \emph{exclusion rule},  which imposes no more than one particle at each site,
\item [ii.] the \emph{facilitated rule}, which asks for a neighboring occupied site  in order for a particle to jump to the other neighboring empty site.
\end{enumerate} Moreover, we consider here nearest-neighbor dynamics, and there are two parameters $p,p^\prime \in [0,1]$  which regulate the choice of one of the two possible jump directions (to the left or to the right). The jump rate to swap the values $\eta_x$ and $\eta_{x+1}$ is thus given by
\begin{equation} 
\label{eq:rate} 
c_{x,x+1}(\eta)=p \eta_{x-1}\eta_x(1-\eta_{x+1})+p^\prime (1-\eta_{x})\eta_{x+1}\eta_{x+2}.
\end{equation}
Note that $p=p'$  corresponds to the \emph{symmetric} case, while $p\neq p'$ to the \emph{asymmetric} one. 

\begin{remark} In the asymmetric case $p\neq p'$, recall that we will take $\bL_N=\Z$, and therefore the generator $\genex$ does not depend on $N$, nevertheless we keep it as an index to keep in mind that $N$ is the scaling parameter for the facilitated exclusion process.\end{remark}

Let us now recall some results from \cite{blondel2020hydrodynamic,blondel2021stefan}: the facilitated exclusion process displays a \emph{phase transition} at the critical density $\rho_c=\frac12$. Indeed,
\ccl{because of the facilitated constraint (rule ii. above)}, pairs of neighboring empty sites cannot be created by the
dynamics. Therefore, if at initial time the density of particles $\rho$ is bigger than $\rho_c$ (at least half of the sites are occupied), then particles will perform random jumps
in the microscopic system until there are no longer two neighboring empty sites. Similarly, if
initially $\rho<\rho_c$  (at least half of the sites are empty), particles will perform random jumps until all particles can no longer move. The particle configurations can therefore be divided into several categories: 
\begin{itemize}
\item the \emph{ergodic} configurations, where all empty sites are isolated, namely:

 $\eta$ is \emph{ergodic} if, for any $x\in\bL_N$, $\eta_x+\eta_{x+1}\geqslant 1$; 
\item the \emph{frozen} configurations, where all particles are isolated, namely: 

$\eta$ is \emph{frozen} if, for any $x\in\bL_N$, $\eta_x+\eta_{x+1}\leqslant 1$;

\noindent \ccl{Note in particular that  the alternated configurations, where each particle is surrounded by empty sites and vice-versa, are critical, since their density is exactly $\frac12$, and it is convenient for them to be considered both frozen \emph{and} ergodic:  they are indeed frozen (no particle is allowed jump in them), and because they have probability non-zero under the grand canonical distribution of the process, whose support we refer to as the \emph{ergodic component}, it is natural to see them as ergodic as well.}

\item the \emph{transient} configurations, which are the remaining ones, those which are neither ergodic, nor frozen. They are called transient in \cite{blondel2020hydrodynamic,blondel2021stefan} because, assuming that $\bL_N=\T_N$ is of size $N$ and starting from a transient configuration, the microscopic process will evolve towards either the ergodic or frozen component, after a number of jumps which is finite a.s (and depend on the size $N$ of the lattice and the initial distribution of particles). See Figure \ref{fig:setEF} below.
\end{itemize} 

\begin{center}
\begin{figure}[h]
\centering
\begin{tikzpicture}
\node at (1,0.5) {\color{white}x};
\draw (0,0) -- (8,0);
\foreach \i in {0,...,7}
{
\draw (\i+0.5,-0.1) -- (\i+0.5,0.1);
}
\node[circle,fill=black,inner sep=1mm] at (0.5,0.2) {};
\node[circle,fill=black,inner sep=1mm] at (3.5,0.2) {};
\node[circle,fill=black,inner sep=1mm] at (5.5,0.2) {};
\node[circle,fill=black,inner sep=1mm] at (1.5,0.2) {};
\node[circle,fill=black,inner sep=1mm] at (6.5,0.2) {};

\node[anchor=west] at (8.2,0.1) {$\eta$ is ergodic};

\draw (0,-1) -- (8,-1);
\foreach \i in {0,...,7}
{
\draw (\i+0.5,-1.1) -- (\i+0.5,-0.9);
}
\node[circle,fill=black,inner sep=1mm] at (4.5,-0.8) {};
\node[circle,fill=black,inner sep=1mm] at (7.5,-0.8) {};
\node[circle,fill=black,inner sep=1mm] at (2.5,-0.8) {};
\node[anchor=west] at (8.2,-0.9) {$\eta$ is frozen};

\draw (0,-2) -- (8,-2);
\foreach \i in {0,...,7}
{
\draw (\i+0.5,-2.1) -- (\i+0.5,-1.9);
}
\node[circle,fill=black,inner sep=1mm] at (3.5,-1.8) {};

\node[circle,fill=black,inner sep=1mm] at (4.5,-1.8) {};
\node[circle,fill=black,inner sep=1mm] at (2.5,-1.8) {};
\node[anchor=west] at (8.2,-1.9) {$\eta$ is transient and will become frozen};

\draw (0,-3) -- (8,-3);
\foreach \i in {0,...,7}
{
\draw (\i+0.5,-3.1) -- (\i+0.5,-2.9);
}
\node[circle,fill=black,inner sep=1mm] at (3.5,-2.8) {};
\node[circle,fill=black,inner sep=1mm] at (1.5,-2.8) {};
\node[circle,fill=black,inner sep=1mm] at (7.5,-2.8) {};

\node[circle,fill=black,inner sep=1mm] at (4.5,-2.8) {};
\node[circle,fill=black,inner sep=1mm] at (2.5,-2.8) {};
\node[anchor=west] at (8.2,-2.9) {$\eta$ is transient and will become ergodic};

\end{tikzpicture}
\caption{Example of configurations belonging to the ergodic, frozen and transient sets, with $N=8$ \ccl{sites in a periodic setting (lattice for the symmetric case)}.}
\label{fig:setEF}
\end{figure}

\end{center}

As a consequence,
the invariant measures of the facilitated process are not independent products of homogeneous Bernoulli
measures (as in the standard Simple Exclusion Process for instance). More precisely: 
 if  $\rho>\rho_c$, then
there is a unique invariant canonical measure $\pi_\rho$ on $\{0,1\}^{\bb Z}$ (which can be described explicitly, see \cite[Section 6]{blondel2020hydrodynamic}), while all the invariant measures are superpositions of atoms (concentrated on frozen configurations) if the
density is less than $\rho_c$. We refer the reader to \cite[Lemma 3.6]{blondel2021stefan} for the full characterization of all stationary measures, which has been proved in the symmetric case $p=p'$, but the exact same argument remains valid for the asymmetric case where $p\neq p'$.

\medskip

In the present work we investigate the rescaled  process $\{\eta(t)\}_{t\geqslant 0}$ with generator $N^{\kappa}\genex$, where \begin{equation}\kappa= \begin{cases} 2 & \text{ in the symmetric case } p=p'=1 \quad \text{ (diffusive scaling)}, \\ 1 & \text{ in the asymmetric case } p'=1-p\in[0,\tfrac12)\quad \text{ (hyperbolic scaling).} \end{cases}\label{eq:kappa}\end{equation}
We denote by $u\simeq x/N$ the continuous space variable.  For an initial density profile $\rho^{\mathrm{ini}}:\bL\to[0,1]$, let us define  the   initial  product  distribution $\nu_N$  on $\Sigma_N$ by its marginals
\begin{equation}\label{eq:nuN}\eta_x(0)=
\begin{cases}
	1 & \mbox{ with probability }\rho^{\mathrm{ini}}(\frac xN)\\
	0& \mbox{ with probability }1-\rho^{\mathrm{ini}}(\frac xN)
\end{cases} \quad \text{for any } x \in \bb{L}_N.
\end{equation}
In other words, $\nu_N$ is a non-homogeneous product of Bernoulli measures which fits $\rho^{\rm ini}$, since it satisfies the following convergence: 
\[\lim_{N\to\infty} \frac1N \sum_{x\in\bL_N} \eta_x \varphi(\tfrac x N) = \int_\bL \rho^{\rm ini}(u)\varphi(u)du, \qquad \text{in } \nu_N\text{--probability,}\] for any continuous test function $\varphi$. Note that, under $\nu_N$, there can be two neighboring empty sites with positive probability, therefore the initial configuration is not supposed to be ergodic.

Fix a time horizon $T > 0$.  For an initial probability measure $\nu$ on $\Sigma_N$, denote by $\P_{\nu}$ the measure on the path space of c\`adl\`ag trajectories $D([0,T],\Sigma_N)$ associated to the rescaled process $\eta(t)$ with generator $N^ \kappa \genex$ and initial distribution $\nu$. In each section, we will clearly state if we are looking at the asymmetric or symmetric case, therefore we do not burden our notations by making the constant $\kappa$ appear explicitly. 

\medskip

Now, let us introduce the notations which are necessary in order to understand the macroscopic evolution of this system. First, let us define the following functions on $[0,1]$:
\begin{equation}
\label{eq:defH} 
\mathcal{H}(r)=\frac{2r-1}{r}\mathbbm{1}_{\{r> \tfrac12\}}\qquad  \mbox{and} \qquad \mathfrak{H} (r)  = \frac{(1-r) (2r -1)}{r} \mathbbm{1}_{\{r > \tfrac12\}}=(1-r)\mathcal{H}(r).
\end{equation}
We are now ready to define the notion of solution to two PDEs corresponding to the two possible macroscopic limits of the facilitated exclusion process: (i) \lj{a parabolic} equation (in the symmetric case) (ii) a hyperbolic equation (in the asymmetric case). 

\begin{definition}[\lj{Parabolic} \ccl{Stefan problem} for exclusion]
\label{Def:PDEweak}
Fix a measurable initial profile $\rho^{\rm ini}:\T\to[0,1]$. 
We say that a  measurable function $\rho:(t,u)\in\R_+\times \T \mapsto\rho_t(u) \in [0,1]$ is a \emph{weak solution to the \ccl{Stefan problem}}
\begin{equation}
\label{eq:PDEstrong}
\partial_t\rho=\partial_u^2\mathcal{H}(\rho)
\end{equation}
with initial condition $\rho_0=\rho^{\mathrm{ini}}$, if for any test function  $\varphi_t(u)\in C^{1,2}(\R_+\times \T)$, any $t>0$,
\begin{equation}
\label{eq:weakF1}
\big\langle \rho_t,\varphi_t \big\rangle = \big\langle \rho^{\rm ini}, \varphi_0 \big\rangle+\int_0^t\big\langle \rho_s,\partial_s\varphi_s \big\rangle ds +\int_0^t \big\langle \mathcal{H}(\rho_s),\;  \partial_u^2\varphi_s \big\rangle ds.
\end{equation}
\end{definition}

\begin{definition}[Hyperbolic \ccl{Stefan problem} for exclusion]\label{def:entropy}
Fix a measurable initial profile $\rho^{\mathrm{ini}}: \bb{R} \rightarrow [0,1]$.  We say that a measurable function $\rho:(t,u)\in\R_+\times \R\mapsto\rho_t(u) \in [0,1]$ is  an \emph{entropy solution to the hyperbolic equation}
	\begin{equation}\label{ep:Hydro}
			\partial_t \rho + (2p-1) \partial_u \mathfrak{H} (\rho) = 0 \end{equation} with the initial condition $\rho_0 = \rho^{\mathrm{ini}}$
	if 
	\begin{enumerate}
		\item \emph{(entropy inequality)} for any non-negative test function  $\varphi \in C^{1,1}( \bb{R}_+ \times \bb{R})$ with compact support in $(0,\infty) \times \bb{R}$,  for any $0 \leq c \leq 1$,
		\[
			\int_0^\infty \langle|\rho_t - c|, \partial_t \varphi_t\rangle + (2p-1)\langle \mathfrak{q} (\rho_t;c,\mathfrak{H}), \partial_u \varphi_t\rangle dt \geq 0,
		\]
		where $\mathfrak{q} (\rho;c,\mathfrak{H}) = {\rm sign }(\rho - c) (\mathfrak{H}(\rho) - \mathfrak{H}(c))$ ; 
		\medskip

		\item \emph{(initial condition)} for any $A > 0$,
		\[
			\lim_{t \rightarrow 0}\, \int_{-A}^A\, |\rho_t (u)- \rho^{\mathrm{ini}} (u)|\,du = 0.
		\] 
	\end{enumerate}
\end{definition}

\begin{remark}[Uniqueness of solutions]\label{remark:uniqueness}
We refer the readers to \cite{Uchiyama} for the uniqueness of weak solutions to \eqref{eq:PDEstrong}, which is based on the fact that the function $\mathcal{H}$ is non-decreasing, and to \cite[Theorem 2.5.1]{malek1996} for the uniqueness of entropy solutions to \eqref{ep:Hydro}.
\end{remark}

We now state the main result about the macroscopic limit of the facilitated exclusion process, both in the symmetric and asymmetric cases.
\medskip

\begin{theorem}[Hydrodynamic limit for the  facilitated exclusion process]
\label{thm:ex}

Let us assume that the initial profile  $\rho^{\mathrm{ini}}: \bb{L} \rightarrow [0,1]$ is Riemann integrable on $\mathbb{L}$ and  \mm{bounded} away from 1, namely: $\rho^{\mathrm{ini}}(u)\leqslant \rho_\star < 1$ for any $u \in \bb{T}$.
\begin{enumerate}
\item[\emph{(I)}] \emph{In the symmetric case, $p=p'=1$:} 

 choose $\kappa=2$, and denote by $\rho_t(u)$ the unique weak solution of the \lj{parabolic} equation \eqref{eq:PDEstrong};
 \medskip
 
\item[\emph{(II)}] \emph{in the asymmetric case, $p'=1-p\in [0,\frac12)$:} 

assume moreover that $\rho^{\rm ini}$ is of bounded variation  on  $\R$, namely $\rho^{\rm ini}\in BV(\R)$,

choose $\kappa=1$, and denote by $\rho_t(u)$ the unique entropy solution of the hyperbolic equation \eqref{ep:Hydro};
\end{enumerate}
both with initial condition $\rho_0=\rho^{\mathrm{ini}}$. Then, for any $\varepsilon>0$, any test function $\varphi \in C^2(\bL)$ with compact support, and any $t>0$,
\[\limsup_{N\to\infty}\P_{\nu_N} \pa{\bigg|\frac{1}{N}\sum_{x\in \bL_N}\eta_x(t)\varphi (\tfrac xN)-\int_{\bL}\rho_t(u) \varphi (u)du\bigg|>\varepsilon}=0.\]
\end{theorem}

This theorem is one of the main results of this paper, its proof will be concluded  in Section \ref{sec:conclsym} (symmetric case) and Section \ref{sec:conclasym} (asymmetric case).

\begin{remark}
Note that the symmetric case \emph{(I)} has already been proved in \cite{blondel2021stefan}, under slightly more general assumption on the initial profile $\rho^{\rm ini}$ which only needs to be Riemann  integrable. We propose here an alternative proof, simpler, but with an additional restriction on the initial data.
\end{remark}

\subsection{Hydrodynamic limit for the facilitated zero-range process}
\label{sec:ZR}

In this section we denote by $M\in\N$ the scaling parameter, in order to distinguish notations with respect to the exclusion process defined in the previous section.
We consider here the \emph{facilitated zero-range process} which  is a Markov process on the set of  configurations $ \omega \in \Gamma_M:=\N^{\bL_M}$, where  we recall that $\bL_M$ can be either $\T_M$ (in the symmetric case) or $\Z$ (in the asymmetric case). As before, we denote by $\bL$ its continuous limit, which is either the continuous torus $\T=\R/\Z$ or the infinite line $\R$. Since we will extensively be mapping, both at the microscopic and macroscopic level, zero-range process and exclusion process, in order to avoid confusion we use different notations: in particular  we use $y$ as the microscopic space variable for the zero-range process,  and therefore $\omega_y(t)\in \N$ is the number of particles present at site $y$ at time  $t$. Moreover, its  macroscopic space variable is denoted by $v\simeq y/M$.  

\medskip

The infinitesimal generator ruling the evolution in time of the \emph{facilitated zero-range process} is given by $\genzr$, which acts on  functions $f:\Gamma_M \to \R$ as 
\begin{equation}\label{eq:DefLM}
	\genzr f (\omega) := \sum_{y \in\bb{L}_M} \Big\{ p \mathbbm{1}_{\{\omega_y \geq 2\}} \big(f(\omega^{y,y+1}) - f(\omega)\big) 
	+ p^\prime  \mathbbm{1}_{\{\omega_{y+1} \geq 2\}} \big(f(\omega^{y+1,y}) - f(\omega)\big)\Big\},
\end{equation}
where $\omega^{y,y'}$ denotes the configuration obtained from $\omega$ by adding a particle at site $y'$ and removing one at site $y$, namely:
\[
\omega_z^{y,y'}=
\begin{cases}
\omega_y-1 & \mbox{ if } z=y,\\
\omega_{y'}+1 & \mbox{ if } z=y',\\
\omega_z & \mbox{ otherwise}.
\end{cases}
\]
In other words, if there are at least two particles at site $y \in \bL_M$, one of them jumps to the right (resp.~to the left) at rate $p$ (resp.~$p'$). As before, $p,p^\prime \in [0,1]$ are two parameters such that $p=p'$  corresponds to the symmetric case, while $p\neq p'$ to the asymmetric one. Note that, as will be extensively exploited later on, this zero-range process is \emph{attractive}, because the function $g(k)=\mathbbm{1}_{\{k \geq 2\}}$ is non-decreasing (see \cite[Chapter 2, Section 5]{klscaling}).

\medskip

As for the exclusion process, the zero-range process admits a critical density $\alpha_c=1$ which induces a phase separation. Here, the dynamical constraint entails that a particle that is alone on its site cannot move. Therefore, if the initial density of particles $\alpha$ is bigger than $1$, then particles will perform random jumps until there is at least one particle per site. If initially $\alpha <1$, then particles will perform random jumps until all sites are occupied by at most one particle. For the  facilitated zero-range process, a configuration $\omega$ is  \emph{ergodic} if, for any $y \in\bL_M$, $\omega_y \geqslant 1$, and is \emph{frozen} if, for any $y\in\bL_M$, $\omega_y \leqslant 1$. See Figure \ref{fig:setEFomega}.

\begin{center}
\begin{figure}[h]
\centering
\begin{tikzpicture}
\node at (1,0.5) {\color{white}x};
\draw (0,0) -- (4,0);
\foreach \i in {0,...,3}
{
\draw (\i+0.5,-0.1) -- (\i+0.5,0.1);
}
\node[circle,fill=black,inner sep=1mm] at (0.5,0.2) {};
\node[circle,fill=black,inner sep=1mm] at (0.5,0.5) {};

\node[circle,fill=black,inner sep=1mm] at (3.5,0.2) {};
\node[circle,fill=black,inner sep=1mm] at (2.5,0.8) {};
\node[circle,fill=black,inner sep=1mm] at (2.5,0.5) {};
\node[circle,fill=black,inner sep=1mm] at (1.5,0.2) {};
\node[circle,fill=black,inner sep=1mm] at (2.5,0.2) {};

\node[anchor=west] at (4.2,0.1) {$\omega$ is ergodic};

\draw (0,-1) -- (4,-1);
\foreach \i in {0,...,3}
{
\draw (\i+0.5,-1.1) -- (\i+0.5,-0.9);
}
\node[circle,fill=black,inner sep=1mm] at (1.5,-0.8) {};
\node[circle,fill=black,inner sep=1mm] at (2.5,-0.8) {};
\node[anchor=west] at (4.2,-0.9) {$\omega$ is frozen};

\draw (0,-2) -- (4,-2);
\foreach \i in {0,...,3}
{
\draw (\i+0.5,-2.1) -- (\i+0.5,-1.9);
}
\node[circle,fill=black,inner sep=1mm] at (3.5,-1.8) {};

\node[circle,fill=black,inner sep=1mm] at (2.5,-1.5) {};
\node[circle,fill=black,inner sep=1mm] at (2.5,-1.8) {};
\node[anchor=west] at (4.2,-1.9) {$\omega$ is transient and will become frozen};

\draw (0,-3) -- (4,-3);
\foreach \i in {0,...,3}
{
\draw (\i+0.5,-3.1) -- (\i+0.5,-2.9);
}
\node[circle,fill=black,inner sep=1mm] at (3.5,-2.8) {};
\node[circle,fill=black,inner sep=1mm] at (1.5,-2.8) {};
\node[circle,fill=black,inner sep=1mm] at (1.5,-2.5) {};
\node[circle,fill=black,inner sep=1mm] at (2.5,-2.5) {};

\node[circle,fill=black,inner sep=1mm] at (2.5,-2.8) {};
\node[anchor=west] at (4.2,-2.9) {$\omega$ is transient and will become ergodic};

\end{tikzpicture}
\caption{Example of configurations belonging to the ergodic, frozen and transient sets, with $M=4$  \ccl{sites in a periodic setting (lattice for the symmetric case)}.}
\label{fig:setEFomega}
\end{figure}

\end{center}

On the ergodic component, there is a unique family $\{\mu^\star_\alpha\}$ of invariant measures, parametrized by the density $\alpha >1$, which can be easily described as follows \lj{(see \cite{Andjel82})}: $\mu^\star_\alpha$ is a product of geometric probability measures whose marginal at each site takes values in $\N^*:=\N\setminus\{0\}$, and moreover
\begin{equation}
\label{eq:mualphastar}
\mu^\star_\alpha(\omega_y=k)=\mathbbm{1}_{\{k\in \N^*\}}\frac{1}{\alpha}\pa{1-\frac{1}{\alpha}}^{k-1}\quad \text{ for any } y \in \bL_M.
\end{equation}
Given an initial density profile $\alpha^{\rm ini}:\bL\to\R_+$,  we consider for the zero-range process an initial probability measure $\mu_M$  on $\Gamma_M$ fitting $\alpha^{\rm ini}$, \textit{i.e}.~such that for any $\varepsilon>0$ and any test function $\varphi \in C^2(\bL)$ with compact support,
\begin{equation}
\label{eq:initmeasure}
\lim_{M\to\infty} \mu_M \pa{\bigg|\frac{1}{M}\sum_{y\in \bL_M}\omega_y \varphi (\tfrac yM)-\int_{\bL}\alpha^{\rm ini} (v) \varphi (v)dv\bigg|>\varepsilon}=0.
\end{equation}
One possible example is to take $\mu_M$ as the product measure on $\Gamma_M$ with marginals given by
\begin{equation}\label{eq:muM}
\mu_M(\omega_y=k)=\mathbbm{1}_{\{k\in \N\}}\frac{1}{1+\alpha^{\mathrm{ini}}(\frac{y}{M})}\pa{1-\frac{1}{1+\alpha^{\mathrm{ini}}(\frac{y}{M})}}^k, \quad \text{ for any } y \in \bL_M.
\end{equation}
Note that $\mu_M(\omega_y=0)$ can be positive, in particular  the initial distribution $\mu_M$ is not a state of local \lj{equilibrium;} it allows for empty sites. In particular, the initial configuration is not assumed to be ergodic at the initial time.

Fix a time horizon $T>0$. Given a  probability measure $\mu$ on $\Gamma_M$ we denote by $ \{\omega(t)\}_{t\geqslant 0}$ the facilitated zero-range process started from the initial distribution $\mu$, and with generator $M^\kappa\genzr$, where  $\kappa \in \{1,2\}$ is the time scaling parameter defined as in \eqref{eq:kappa}. We denote by $\mathbf{P}_{\mu}$ the corresponding probability measure on the space  of trajectories $D([0,T],\Gamma_M)$.

For any $r\geq 0$, let us define the function
\begin{equation}
\label{eq:defG} 
\mathcal{G}(r)=\frac{r-1}{r}{\mathbbm{1} }_{\{r>1\}}.
\end{equation}
Given the function $\mathcal{G}$ the definition of weak and entropy solutions for the zero-range process are strictly analogous to Definitions \ref{Def:PDEweak} and \ref{def:entropy}:  a measurable function $\alpha_t(v)$ is a weak solution to 
\begin{equation}
\label{eq:PDEstrongZR}
\partial_t\alpha=\partial_v^2\mathcal{G}(\alpha),
\end{equation}
with initial condition $\alpha_0=\alpha^{\mathrm{ini}}$,  if for any test function  $\varphi \in C^{1,2}(\R_+\times \T)$ ,  any $t>0$, 
\begin{equation}
\label{eq:weakF1ZR}
\big\langle \alpha_t,\varphi_t \big\rangle = \big\langle \alpha^{\rm ini}, \varphi_0 \big\rangle+\int_0^t\big\langle \alpha_s,\partial_s\varphi_s \big\rangle ds +\int_0^t \big\langle \mathcal{G}(\alpha_s),\;  \partial_v^2\varphi_s \big\rangle ds.
\end{equation}

On the other hand, given a \emph{bounded} initial profile $\alpha^{\mathrm{ini}}$, an $L^\infty(\R_+\times\R)$--function  $\alpha_t(v)$ is an entropy solution to 
\begin{equation}
\label{zrp:Hydro}	
\partial_t\alpha+(2p-1)\partial_v\mathcal{G}(\alpha)=0 
\end{equation}
 if for any non-negative test function  $\varphi \in C^{1,1}(\R_+\times \R)$ with compact support in $(0,\infty)\times \R$, any $c \geq 0$, any $A>0$
		\[
			\int_0^\infty  \langle |\alpha_t - c |,\partial_t \varphi_t\rangle + (2p-1)\langle \mathfrak{q}\big(\alpha_t;c, \mathcal{G}\big),\partial_v \varphi_t\rangle dt \geq 0,
		\]
		and
		\begin{equation}
		\label{eq:initZR}
			\lim_{t \rightarrow 0}\, \int_{-A}^A\, |\alpha_t(v) - \alpha^{\mathrm{ini}} (v)|\,dv = 0,
		\end{equation} 
see  Definitions  \ref{Def:PDEweak} and \ref{def:entropy} above for the precise statements.

\begin{remark}
Note the following relation between functions $\mathcal G$, $\mathcal H$ and $\mathfrak H$ which will be useful in what  follows: 
\begin{equation}
\label{eq:relation}
\mathcal G(r)= \mathcal{H}\big(\tfrac{r}{1+r}\big)=(1+r)\mathfrak{H}\big(\tfrac{r}{1+r}\big).
\end{equation}
\end{remark}

We are now  ready to state the hydrodynamic limit result for the facilitated zero-range process, both in the symmetric and asymmetric cases.

\begin{theorem}[Hydrodynamic limit for the  facilitated zero-range process] 
\label{thm:zr}

Let us assume that the initial profile $\alpha^{\mathrm{ini}}:\bb{L}\to\R_+$ is bounded and Riemann integrable on $\bb{L}$, \lj{and that the initial measure $\mu_M$ satisfies \eqref{eq:initmeasure}}.
\begin{enumerate}
\item[\emph{(I)}] \emph{In the symmetric case, $p=p'=1$:} 

choose $\kappa=2$, and denote by $\alpha_t(v)$ the weak solution of the \lj{parabolic} equation \eqref{eq:PDEstrongZR};
\medskip

\item[\emph{(II)}] \emph{in the asymmetric case, $p'=1-p\in [0,\frac12)$:} 

choose $\kappa=1$, and denote by $\alpha_t(v)$ the entropy solution of the hyperbolic equation \eqref{zrp:Hydro};
\end{enumerate}
both with initial condition $\alpha_0=\alpha^{\mathrm{ini}}$. Then, for any $\varepsilon>0$, any test function $\varphi \in C^2(\bL)$ with compact support, and any $t>0$,
\begin{equation}\label{lln}
\lim_{M\to\infty}\mathbf{P}_{\mu_M} \pa{\bigg|\frac{1}{M}\sum_{y\in \bL_M}\omega_y(t)\varphi (\tfrac yM)-\int_{\bL}\alpha_t(v) \varphi (v)dv\bigg|>\varepsilon}=0.
\end{equation}
\end{theorem}

This theorem is the second main result of this paper, and its proof will be achieved in Section \ref{sec:hydroZR}. 

\begin{remark}
Note that the function $\mathcal{G}$ appearing in the hydrodynamic limit for the zero-range process is the same in both the symmetric and asymmetric cases, which is not the case for the exclusion process, for which the two functions $\mathcal{H}$ and $\mathfrak{H}$ (\emph{cf.}~\eqref{eq:defH}) are different. This is a specific feature of zero-range processes. 

To illustrate this point, fix a lattice gas, for which we denote $j_{x,x+1}=p c_{x,x+1}-p'c_{x+1,x}$ the microscopic current of particles along the edge $(x,x+1)$, $c_{x,y}$ denoting the rate function at which a particle jumps from $x$ to $y$. In the symmetric case $p=\frac12$, assuming the model is \emph{gradient}, there exists a function $h_x$ (depending on a finite number of coordinates) such that $j_{x,x+1}=h_x-h_{x+1}$.
The symmetric hydrodynamic limit ($p=\frac12$) is then formally given by 
\[\partial_t \rho =\partial_u^2\; \E_{\rho}[h_0],\]
where $\E_{\rho}$ stands for the expectation w.r.t.~the grand-canonical equilibrium distribution with density $\rho$. 

In the asymmetric case $p'=1-p\in [0,\frac12)$, the symmetric part of the current vanishes in the macroscopic limit, and what remains of the current at the macroscopic hyperbolic scale is $(2p-1)c_{x,x+1}$, therefore the asymmetric hydrodynamic limit is formally given by
\[\partial_t \rho =-(2p-1)\partial_u \E_{\rho}[c_{0,1}].\]
By definition, for zero-range processes, $h_x=c_{x,x+1}=g(\omega_x)$ (where $g:\N\to\R_+$ is the general rate function) so that the macroscopic quantities appearing in the symmetric and asymmetric hydrodynamic limits are the same.
\end{remark}

\subsection{Strategy of the proof}
\label{ssec:strategy}
Our strategy is the following: 
\begin{enumerate}
\item We construct a mapping at the microscopic level between the facilitated exclusion process $\{\eta(t)\}$ and the facilitated zero-range process $\{\omega(t)\}$ and we prove that it has good properties, namely that: 
\begin{itemize} 
\item it keeps all the information needed to deduce the hydrodynamic limit of $\{\eta(t)\}$ from the one of $\{\omega(t)\}$,

\item and moreover it corresponds to a macroscopic mapping between weak solutions of the hydrodynamic equations given in Theorem \ref{thm:ex} and \ref{thm:zr}. \end{itemize}
This is the purpose of Section \ref{sec:sym} (in the symmetric case) and Section \ref{sec:asym} (in the asymmetric case). 
\item We prove independently Theorem \ref{thm:zr}, namely the hydrodynamic limits for the facilitated zero-range process in both symmetric \emph{(I)} and asymmetric \emph{(II)} cases, adapting various tools in the literature and in particular making use of the attractiveness property (which the zero-range process \emph{has}, but the exclusion process \emph{does not have}). This is done in Section \ref{sec:hydroZR}.

\item Finally we can use the mapping construction to deduce Theorem \ref{thm:ex} from Theorem \ref{thm:zr}: we propose, first, an alternative proof of the hydrodynamic limit for exclusion in the symmetric case \emph{(I)}  (which has already been obtained in \cite{blondel2021stefan}) -- this is the purpose of Section \ref{sec:conclsym};  and second, a rigorous proof of the asymmetric case \emph{(II)}, which is new and was out of reach without this mapping construction -- this is the purpose of Section \ref{sec:conclasym}.
\end{enumerate}


\section{The symmetric case}  \label{sec:sym} 

Let us start with the symmetric case $p=p'=1$. In this section we assume that Theorem \ref{thm:zr} holds. It will be proved independently in Section \ref{sec:hydroZR}.  We start, both at the microscopic and macroscopic level, from the exclusion process, from which we will build the zero-range process, as well as the zero-range macroscopic profile.

\subsection{Mapping EX  $\mapsto$ ZR} For an exclusion configuration $\eta=(\eta_x)_{x\in \T_N}$, denote by \[M=M^\eta:=N-\sum_{x\in \T_N}\eta_x\] the number of empty sites in the configuration $\eta$.  We then denote by $X_1<X_2<\dots< X_M$ their positions, $X_1$ being the first empty site to the right of site $0$ (if it exists) in the configuration. We can then define, for any $y\in \{1,\dots, M\}$ 
\begin{equation}\label{eq:defmap}\widehat{\omega}^{\eta, X_1}_y = \text{ the number of particles between the $y$-th and $ (y+1)$-th empty sites in $\eta$},\end{equation}
where $M+1$ is identified with $1$. In other words: 
\begin{equation}\label{eq:defmapbis}\widehat{\omega}^{\eta, X_1}_y = X_{y+1}-X_y -1 \quad \mathrm{ mod } \; N \ \in \{1,\dots, N\}.\end{equation}
If there are no empty sites in $\eta$, we arbitrarily set $M=1$, $X_1=1$ and $\widehat{\omega}^{\eta, X_1}$ to be the degenerate configuration with $N$ particles on its unique site. See Figure \ref{fig:mapping} for an illustration of this mapping. 

\begin{center}
\begin{figure}[h]
\centering
\begin{tikzpicture}
\node at (1,0.5) {\color{white}x};
\draw (0,0) -- (8,0);
\foreach \i in {0,...,7}
{
\draw (\i+0.5,-0.1) -- (\i+0.5,0.1);
}
\node[circle,fill=black,inner sep=1mm] at (0.5,0.2) {};
\node at (0.5,-0.3) {$0$};
\node at (1.5,-0.3) {$X_1$};
\node at (4.5,-0.3) {$X_2$};
\node at (5.5,-0.3) {$X_3$};

\node[circle,fill=blue,inner sep=1mm] at (3.5,0.2) {};
\node[circle,fill=black,inner sep=1mm] at (6.5,0.2) {};
\node[circle,fill=blue,inner sep=1mm] at (2.5,0.2) {};
\node[circle,fill=black,inner sep=1mm] at (7.5,0.2) {};

\node[anchor=west] at (8.5,0) {$\eta$ };

\draw (2,-2) -- (5,-2);
\foreach \i in {2,...,4}
{
\draw (\i+0.5,-2.1) -- (\i+0.5,-1.9);
}
\node at (2.5,-2.3) {$1$};

\node at (3.5,-2.3) {$2$};
\node at (4.5,-2.3) {$3$};
\node[circle,fill=black,inner sep=1mm] at (4.5,-1.8) {};
\node[circle,fill=black,inner sep=1mm] at (4.5,-1.5) {};
\node[circle,fill=black,inner sep=1mm] at (4.5,-1.2) {};

\node[circle,fill=blue,inner sep=1mm] at (2.5,-1.5) {};
\node[circle,fill=blue,inner sep=1mm] at (2.5,-1.8) {};
\node[anchor=west] at (8.5,-2) {$\hat\omega^\eta$};

\end{tikzpicture}
\caption{An example of $\eta$ and its corresponding $\hat \omega^\eta$ in \lj{a periodic setting with $N=8$}.}
\label{fig:mapping}
\end{figure}
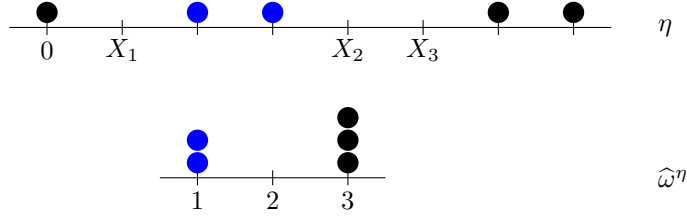

\end{center}

\begin{lemma}[Mapping] \label{lem:mapping} Assume that $\eta(t)$ is a trajectory of the Markov process with generator $N^2 \genex$ defined in \eqref{eq:DefLN},  with initial configuration $\eta(0) \in \Sigma_N$. Denote by $M=M^{\eta(0)}\leq N$ the initial number of empty sites. Due to the conservation law of the total number of particles, $M=M^{\eta(t)}$ does not depend on time $t$. Let us tag the first empty site to the right of site $0$ (if it exists) in $\eta(0)$, with initial position $X_1(0)$, and follow its evolution throughout the dynamics, denoting by $X_1(t)$ its position at time $t$. Then, the process 
\[\omega(t):=\widehat{\omega}^{\eta(t), X_1(t)}\] defined through \eqref{eq:defmap}, 
is a Markov process on the state space $\Gamma_{M}$, with generator $N^2 \genzr$ defined in \eqref{eq:DefLM}, and initial configuration $\omega(0):=\widehat{\omega}^{\eta(0), X_1(0)}$. If there are no empty sites in the initial configuration, the corresponding dynamics are trivial, and with the convention above $M=X_1(t)=1$, the result still holds.
\end{lemma}


This result is straightforward, hence we do not detail the proof here. In order for the mapping to be well defined at both scales, as stated in Theorem \ref{thm:ex}, we need the following assumption on the initial density profile.

\begin{assumption}
\label{ass:L1}
The initial density profile $\rho^{\rm ini}:\T\to[0,1]$ for the exclusion process is bounded away from $1$ on $\T$
\[\rho^{\mathrm{ini}}(u)\leq \rho_\star<1   \quad \text{ for any } u\in \T   .\]
\end{assumption}

Assume now that $\eta(0)$ is distributed according to $\nu_N$ given in \eqref{eq:nuN}, with $\rho^{\rm ini}$ satisfying Assumption \ref{ass:L1}. Therefore, the process $\omega(t) =\widehat{\omega}^{\eta(t), X_1(t)}$ is well defined. Moreover, the distribution of $\omega(0)=\widehat{\omega}^{\eta(0), X_1(0)}$ satisfies the following local equilibrium property:  define 
\[ \theta := \int_{\T} (1-\rho^{\rm ini}(u))du \qquad \text{and}\qquad v(u):= \theta^{-1} \int_0^u (1-\rho^{\rm ini}(u^\prime))du^\prime, \quad \text{for any $u\in\T$,}\]
and let $\alpha^{\rm ini}:\T\to\R_+$ be such that, for any $v=v(u) \in \T$,
\begin{equation} \alpha^{\rm ini}(v)=\frac{\rho^{\rm ini}(u)}{1-\rho^{\rm ini}(u)}. \label{eq:alpha-ini}\end{equation} Then, we have the following result.

\begin{lemma}[Initial zero-range density profile]
\label{lem:InitialLE} 
For any continuous test function $\varphi:\T\to\R$ and any $\delta >0$,
\begin{equation}\label{eq:initialLE}
		 \lim_{N\to\infty}\P_{\nu_N}\bigg( \bigg| \frac{1}{M} \sum_{y=1}^{M} \varphi\Big(\frac{y}{M}\Big) \omega_y(0) - \int_{\T}\varphi(v)\alpha^{\rm ini}(v)dv\bigg| > \delta \bigg)=0. 
\end{equation}
Note that we also have for any $t > 0$,
\begin{equation}\label{eq:theta}
\theta = \lim_{N\to\infty} M / N=\int_{\T}(1-\rho^{\rm ini}(u)) du = \int_{\T}(1-\rho_t(u)) du \qquad \text{$\nu_N$ -- almost surely}, 
\end{equation}
where $\rho_t (u)$ is the weak solution to \eqref{eq:PDEstrong}. More precisely, $\theta$ can be seen as the macroscopic deficit of mass in the exclusion process which does not depend on $t$ due to the conservation of the total number of particles.
\end{lemma}

\begin{remark}
Note that $\rho^{\rm ini} \leq \rho_\star < 1$ is bounded away from $1$ iff $\alpha^{\mathrm{ini}}$ is bounded uniformly in $v  \in \T$ with
\begin{equation}
\label{eq:assaslpha}
\alpha^{\mathrm{ini}}(v)\leq \alpha_\star := \frac{\rho_\star}{1 - \rho_\star}<\infty.
\end{equation}

\end{remark}

	\begin{proof}[Proof of Lemma \ref{lem:InitialLE}]
		We first prove \eqref{eq:theta}. Since 
		\[M^\eta = \sum_{x \in \bb{T}_N} (1 -\eta_x (0)),\]
		and the random variables $\{\eta_{x} (0),\,x \in \bb{T}_N\}$ are independent, there exists a finite constant $C$ independent of $N$ such that
		\begin{align*}
			\sum_{N \geq 1} \bb{P}_{\nu_N} \Big( \Big| \frac{M^\eta - \bb{E}_{\nu_N} [M^\eta] } N\Big| > N^{-1/8} \Big) 
			&\leq  \sum_{N \geq 1}  N^{-7/2} \bb{E}_{\nu_N} \big[ \big(M^\eta - \bb{E}_{\nu_N} [M^\eta]\big)^4 \big]   \\ 
			& \leq  C \sum_{N \geq 1} N^{-3/2} < \infty.
		\end{align*}
		From Borel-Cantelli's Lemma, this implies
		\[\lim_{N \rightarrow \infty} \frac1N\big(M^\eta - \bb{E}_{\nu_N} [M^\eta] \big) = 0 \qquad \text{$\nu_N$ -- a.s.} \]
		This proves the second identity in \eqref{eq:theta} since, after a straightforward computation using the definition \eqref{eq:nuN} of $\nu_N$, one gets
		\[\lim_{N \rightarrow \infty} \frac1N\bb{E}_{\nu_N} [M^\eta] = \int_{\T} (1- \rho^{\rm ini} (u)) du. \] 
		Besides, the last identity in \eqref{eq:theta} is due to the conservation of number of particles in the exclusion process.
		
		It remains to prove \eqref{eq:initialLE}. It is easy to check  that  the sequence of measure-valued random variables
		\[M^{-1} \sum_{x \in \bb{T}_M}  \omega_y (0) \delta_{y/M}, \quad \text{indexed by }N \geq 1,\]
		is tight \lj{since the distribution of $\omega_y (0)$ is stochastically bounded by the geometric distribution with parameter $\alpha_\star$}.  Moreover, any weak limit of the sequence is concentrated on trajectories which are absolutely continuous with respect to the Lebesgue measure. Now we pick a convergent subsequence denoted by $N^\prime$. Then there exists $\alpha^{\rm ini}: \T \rightarrow \R_+$ such that for every $\varphi \in C(\T)$,
		\[\lim_{N^\prime \rightarrow \infty} \frac{1}{M} \sum_{y \in \bb{T}_M} \varphi \big(\tfrac{y}{M} \big) \omega_0 (y) = \int_{\T} \varphi(v) \alpha^{\rm ini} (v)\,dv \quad \text{in $\P_{\nu_{N^\prime}}$ -- probability.}\]
		Observe that for every $N$,
		\begin{align*}
			\frac{1}{N} \sum_{x \in \bb{T}_{N}} \varphi \big(\tfrac{x}{N}\big) (1 - \eta_x (0)) &= \frac{1}{N} \sum_{y \in \bb{T}_M}  \varphi \Big(\frac{X_y(0)}{N}\Big)\\
			&= \frac{1}{N} \sum_{y \in \bb{T}_M}  \varphi \Big(\frac{1}{N} \sum_{y^\prime = 1}^{y-1}  [ \omega_{y'} (0) + 1] + \frac{X_1(0)}{N}\Big).
		\end{align*}
		Letting $N^\prime \rightarrow \infty$, we obtain (using \eqref{eq:theta})
		\[\int_{\T} \varphi (u) (1 - \rho^{\rm ini} (u)) du = \theta \int_{\T} \varphi \Big(\theta \int_0^v (1 + \alpha^{\rm ini} (v^\prime) )  d v^\prime\Big) dv
		= \int_{\T} \varphi(u) (1+\alpha^{\rm ini} (v))^{-1} du,\]
		where
		\[u(v) := \theta \int_0^v (1 + \alpha^{\rm ini} (v^\prime) )  d v^\prime.\]
		Therefore,
		\[\alpha^{\rm ini} (v) = \frac{\rho^{\rm ini} (u)}{1 - \rho^{\rm ini} (u)}, \qquad  \text{and}\qquad v(u) = \theta^{-1} \int_{0}^u (1 - \rho^{\rm ini} (u^\prime))\,du^\prime.\]
		This is enough to conclude the proof since $\alpha^{\rm ini}$ is uniquely determined by $\rho^{\rm ini}.$
	\end{proof}

This mapping intrinsically depends on the position of the tagged empty site in the exclusion configuration. However, its position \lj{can also be tracked} in the zero-range configuration, where it is given by the total particle current crossing the first edge of the system. For what follows, we need access to the macroscopic counterpart of this quantity in order to define the mapping at the macroscopic scale. 

\begin{lemma}[Law of large numbers for the first tagged empty site]
\label{lem:TP}  Let $\rho^{\rm ini}:\T\to[0,1]$  satisfy Assumption \ref{ass:L1} and $\alpha^{\rm ini}$ be defined as in \eqref{eq:alpha-ini}. 
Consider the weak solution  $\rho_t(u)$, resp.~$\alpha_t(v)$ of the \lj{parabolic} equation \eqref{eq:PDEstrong} with initial condition $\rho^{\rm ini}$ , resp.~\eqref{eq:PDEstrongZR} with initial condition $\alpha^{\rm ini}$ (see Definition \ref{Def:PDEweak}).

Then, for any $t\geq 0$, for any $\epsilon>0$,
\[\lim_{N\to\infty} \P_{\nu_N} \bigg(\bigg|\frac{X_1(t)}{N}-\chi_t\bigg|>\epsilon\bigg)=0\]
 where  $\chi_t$ is defined explicitly as a function of $\alpha$ by
\begin{equation}\label{eq:chi}\chi_t:=\theta \big\langle v,\alpha_{t  \theta^{-2} }-\alpha^{\mathrm{ini}}\big\rangle = \theta \int_\T v(\alpha_{t\theta^{-2}}-\alpha^{\rm ini})(v)dv,\end{equation}
and implicitly as a function of $\rho$ as the solution of
\begin{equation}\label{eq:chibis}\int_0^{\chi_t}(1-\rho_t(u))du=\big\langle u,\rho_{t}-\rho^{\mathrm{ini}}\big\rangle = \int_\T u(\rho_t-\rho^{\rm ini})(u)du.\end{equation}
In the identities above, the notation $u,\, v$ represents the identity functions $u\mapsto u$ and  $v\mapsto v$ on $\T$.
\end{lemma}

\begin{remark} Note that one could formally write by integration by parts
\begin{equation}
\label{eq:dtustar}
\frac{d}{dt}\chi_t=\theta^{-1}\partial_v \mathcal{G}(\alpha_{t  \theta^{-2}}) (0)=\frac{1}{1-\rho_t(\chi_t)}\partial_u \mathcal{H} (\rho_t )(\chi_t).
\end{equation}
However, some work is required to give sense to these quantities, given the weak notion of solutions of the \lj{parabolic} \ccl{Stefan problem}s \eqref{eq:PDEstrong} and \eqref{eq:PDEstrongZR}. For this reason, we settle for the definitions \eqref{eq:chi} and \eqref{eq:chibis}, which are well defined without having to prove any regularity property for the macroscopic profiles. See also Section \ref{sec:macromap} below.
\end{remark}

\begin{remark}In fact, Assumption \ref{ass:L1} is not \emph{a priori} required for Lemma \ref{lem:TP} to hold, except to guarantee that an empty site exists in the exclusion configuration. However, this ensures that the initial value for $\chi_t$ is $\chi_0=0$, and since Assumption \ref{ass:L1} will be required throughout anyway, we make it here as well.\end{remark}

\begin{proof}[Proof of Lemma \ref{lem:TP}]
We first assume that $\{\eta^N,\,N \geq 1\}$ is a \emph{deterministic} sequence of configurations such that 
\[\lim_{N\to\infty} \frac{1}{N} \sum_{x\in \T_N} \eta^N_x \delta_{x/N} (du) = \rho^{\rm ini} (u) du \]
under the weak topology (where $\delta_a(du)$ is the Dirac measure concentrated at  $a \in \T$), we are going to prove that, for any $\epsilon > 0$,
\begin{equation}\label{determLLN}
	\lim_{N\to\infty} \bb{P}_{\eta^N} \bigg(\bigg|   \frac{X_1 (t) - X_1 (0)}{N} - \chi_t  \bigg| > \epsilon\bigg) = 0.
\end{equation}
 Recall that we denote by $M=M_{\eta^N}$ the number of empty sites in $\eta^N$.
Let us couple as described in Section \ref{sec:ZR} the microscopic exclusion process starting from $\eta^N$ with a zero-range process $\omega$ with generator  $N^2 \genzr$.
Let $J_{y,y+1}^{\omega} (t)$ be the net number of particles crossing the bond $(y,y+1)$ up to time $t$ throughout the evolution of $\omega$, which means 
\begin{equation}
\label{eq:current}\omega_y (t) - \omega_y (0)=J_{y-1,y}^\omega(t)-J_{y,y+1}^\omega(t).
\end{equation}
Under the microscopic mapping between exclusion and zero-range process,  one can see that \[X_1 (t) - X_1 (0)=- J_{0,1}^{\omega} (t).\] Therefore, we only need to prove that for any $\epsilon > 0$,
 \[
	\lim_{N\to\infty} \bb{P}_{\eta^N} \bigg(\bigg|  \frac{J_{0,1}^{\omega} (t)}{M} + \frac{\chi_t}\theta \bigg| > \epsilon\bigg) = 0. 
 \]
Let $G (v) = v$,  for $v\in (0,1]$, from \eqref{eq:current} we have 
 \begin{equation}\label{eq:new}
 	-J_{0,1}^{\omega} (t) = -\frac{1}{M}\sum_{\ccl{y}=1}^M J_{y,y+1}^{\omega} (t) + \sum_{y=1}^M G (\tfrac{y}{M}) \big( \omega_y (t) - \omega_y (0)  \big).
 \end{equation}
\lj{Since 
\[J_{y,y+1}^{\omega} (t) -  \int_0^t \big[\ccl{{\mathbbm 1}_{\{\omega_y(s)\geq 2\}} - {\mathbbm 1}_{\{\omega_{y+1}(s)\geq 2\}}} \big] ds\]
is a martingale, whose quadratic variation is the number of jumps across the bond $(y,y+1)$ before time $t$, the first term in the right hand side of \eqref{eq:new} is also a martingale,}
whose quadratic variation is given by $M^{-2}N_t$, $N_t$ being the total number of jumps occurring in the system before time $t$, and whose expectation is therefore bounded by $CM$ for some positive constant $C$, because jumps occur at most at rate $2M^2$ at each site. 
In particular, we can write 
 \begin{equation}
 	\frac{1}{M}\sum_{y=1}^M G (\tfrac{y}{M}) \big( \omega_y (t) - \omega_y (0)  \big)+\frac{J_{0,1}^{\omega} (t)}{M} = \mathscr{M}^M_t,
 \end{equation}
where $\mathscr{M}^M_t$ is a martingale with quadratic variation bounded by $ 1/M$.

Since the function $G$ is not a smooth function on the torus, in order to use the hydrodynamic limit for the zero-range process, we approximate it by $G_\epsilon$ such that for any $\varepsilon>0$, $G_\epsilon \in C^2(\T)$  and the following holds: 
\[
G_\epsilon(v)\in [0,1], \quad \text{for any }v\in \T, \qquad \text{and}  \qquad G_\epsilon(v)=v, \quad \text{for any } v\in [\epsilon, 1-\epsilon].
\]
We can now write
\[
\Big| \frac{J_{0,1}^{\omega} (t)}{M} +\frac{1}{M} \sum_{y=1}^M G_{\epsilon} (\tfrac{y}{M}) \big( \omega_y (t) - \omega_y (0)  \big) \Big| \leq  |\mathscr{M}^M_t|+ \frac{2}{M} \sum_{y=[(1-\epsilon) M]}^{[\epsilon M]}   \big( \omega_y (t) + \omega_y (0)  \big).
\]
The zero-range process being stochastically dominated by an equilibrium zero-range process with constant density $1+\alpha_\star = (1-\rho_\star)^{-1}$ by attractiveness, the second  term in the right hand side is of order $C\varepsilon$. Thanks to the bound on the quadratic variation of the martingale $\mathscr{M}^M_t$, we can therefore write, \lj{for any $\delta > 0$,}
\[
	\limsup_{\epsilon \rightarrow 0}\, \limsup_{N \to\infty}\, \bb{P}_{\eta^N} \Big(  \Big|  \frac{J_{0,1}^{\omega} (t)}{M} + \frac{1}{M} \sum_{y=1}^M G_{\epsilon} (\tfrac{y}{M}) \big( \omega_y (t) - \omega_y (0)  \big) \Big| > \delta \Big) = 0.
\]
By maximum principle \ccl{(see e.g. \cite[Theorem 1.1, p.40]{Andreucci04})}, for any $0\leq \varepsilon<\delta (1+\alpha_\star)/6$, we have $\big \langle |G_\varepsilon-G|,\alpha_{t \theta^{-2}}-\alpha^{\mathrm{ini}}\big\rangle\leq \delta/3$. We can therefore write for any such $\varepsilon$,
\begin{align*}
\bb{P}_{\eta^N}& \Big(  \Big|  \frac{J_{0,1}^{\omega} (t)}{M} + \big\langle G,\alpha_{t \theta^{-2}}-\alpha^{\mathrm{ini}}\big\rangle \Big| > \delta \Big)\\&\leq \bb{P}_{\eta^N} \Big(  \Big|  \frac{J_{0,1}^{\omega} (t)}{M} + \frac{1}{M} \sum_{y=1}^M G_{\epsilon} (\tfrac{y}{M}) \big( \omega_y (t) - \omega_y (0)  \big) \Big| > \delta/3 \Big)\\
&+\bb{P}_{\eta^N} \Big(  \Big|  \frac{1}{M} \sum_{y=1}^M G_{\epsilon} (\tfrac{y}{M}) \big( \omega_y (t) - \omega_y (0)  \big) -  \big\langle G_\varepsilon,\alpha_{t \theta^{-2}}-\alpha^{\mathrm{ini}}\big\rangle \Big| > \delta/3 \Big).
\end{align*}
By Theorem \ref{thm:zr}, for any fixed $\varepsilon>0,$ the second term in the right hand side vanishes as $N \to\infty$. As $\varepsilon$ then goes to $0$, the other two terms vanish as well by the dominated convergence Theorem, which concludes the proof of \eqref{determLLN}.  The time factor $\theta^{-2}$ is a consequence of the fact that the zero-range process is accelerated by a factor $N^2=\theta^{-2}M^2+o(N^2)$ and not $M^2$.

We now extend \eqref{determLLN} to random initial configuration distributed under $\nu_N$, $N \geq 1$ (cf. \eqref{eq:nuN}) by using Skorohod's representation Theorem -- see \cite[Theorem 1.6.7]{billingsley2013convergence}.  Indeed, since 
\[\lim_{N\to\infty} \frac{1}{N} \sum_{x\in \T_N} \eta_x \delta_{x/N}(du) = \rho^{\rm ini} (u) du \quad \text{in $\nu_N$ -- probability}\]
under the weak topology, we can construct a sequence of random elements $\{\eta^N,\,N\geq 1\}$ on a common probability space $(\Sigma, \mc{F}, \nu^\star)$ such that $\eta^N (\sigma)$ has distribution $\nu_N$, and such that
\[\lim_{N\to\infty} \frac{1}{N} \sum_{x\in \T_N} \eta^N_x (\cdot) \delta_{x/N}(du) = \rho^{\rm ini} (u) du, \qquad \text{$\nu^\star$ -- a.s.}\]
under the weak topology. Then for any $\varepsilon > 0$,
\begin{align*}
	\lim_{N\to\infty} \bb{P}_{\nu^N} \Big(\Big|   \frac{X_1 (t) - X_1 (0)}{N} - \chi_t  \Big| > \varepsilon\Big) 
	&= \lim_{N\to\infty} \int \nu^N (d \eta) \bb{P}_{\eta} \Big(\Big|   \frac{X_1 (t) - X_1 (0)}{N} - \chi_t  \Big| > \varepsilon\Big)\\
	&=  \lim_{N\to\infty} \mathsf{E}_{\nu^\star} \bigg[ \bb{P}_{\eta^N (\cdot)} \Big(\Big|   \frac{X_1 (t) - X_1 (0)}{N} - \chi_t  \Big| > \varepsilon\Big) \bigg] \\&= 0,
\end{align*} where we denote by $\mathsf{E}_{\nu^\star}$ the expectation with respect to $\nu^\star$. 
Above, the last equality comes from \eqref{determLLN} and the dominated convergence Theorem. We conclude the proof by noting that $\lim_{N\to\infty} X_1 (0) / N = 0$ in $\nu_N$ -- probability.
\end{proof}

Before concluding this section we settle for a very crude bound on the zero-range density.
\begin{lemma}[Bound on the zero-range density]\label{lem:ZRdensity}
Under Assumption \ref{ass:L1}
\[\lim_{N \to\infty}\P_{\nu_N}\pa{\max_{\substack{y\in \T_M\\
t\leq TM^2}}\omega_y(t)\geq \log^2 M}=0.\]
\end{lemma} 
\begin{proof}
This result can be proved using the same steps as in \cite[Lemma 4.4]{blondel2021stefan}. We therefore omit the proof and simply sketch the main idea, and refer the interested reader to the latter for the full implementation. First, one can use attractiveness and the boundedness of $\alpha^{\rm ini}$ to bound the probability of the event above by the same under the equilibrium distribution (see \eqref{eq:mualphastar}) $\mu^\star_{\alpha_\star+1}$, where $\alpha_\star$ is defined in \eqref{eq:assaslpha} as an arbitrary uniform upper bound on $\alpha^{\rm ini}$. The continuous-time process at equilibrium can then be coupled with a discrete-time Markov chain at equilibrium which is the discrete skeleton of the continuous-time one. A union bound then concludes the proof. 
\end{proof}
%
%
%

\subsection{Macroscopic mapping}
\label{sec:macromap}
\subsubsection{Construction}
We now want to build a macroscopic mapping between weak solutions to \eqref{eq:PDEstrongZR} and \eqref{eq:PDEstrong}. We already did it for the initial profiles $\rho^{\rm ini}$, $\alpha^{\rm ini}$ in \eqref{eq:alpha-ini}. With some regularity on the initial profiles, one can proceed as follows: 
fix some measurable function ${\rho}_t:[0,T]\times \T \to [0,1]$,  satisfying 
\[0\leq \rho_t(u)\leq \rho_\star<1, \qquad \text{for any } u\in \T, \;t\leq T, \] 
as well as a measurable time-trajectory $\chi_t:[0,T]\to\T$ satisfying $\chi_0=0$.
Let us then define \begin{equation}\label{eq:defvt}\theta_t:=\int_\T(1-\rho_t(u)) du, \qquad \text{and}\qquad 
 v_t(u):=\theta_t^{-1}\int_{\chi_t}^u(1-\rho_t(u')) du', \quad \text{for any } u \in \T.
\end{equation} 
For any fixed $t>0$, the mapping $v_t:\T\to\T$ is strictly increasing and can be inverted, so that at a given time $t>0$, there is a one-to-one mapping between the two variables $u$ and $v:=v_t(u)$. Now, define 
\begin{equation}
\label{eq:mappingmacro}
\alpha_t(v)=\frac{\rho_t}{1-\rho_t}(v_t^{-1}(v)),
\end{equation}
then one can check that 
\lj{\[u_t(v):= \chi_t+\theta_t \int_0^v(1+\alpha_t)(v')dv',\]}
is the inverse mapping of $v_t$.  
In the same way, given $\alpha_t$ and $\chi_t$, one can build the corresponding $\rho_t$ by inverting all relationships above.

 The problem  is that the mapping $v_t$ defined in \eqref{eq:defvt} linking $\rho$ to $\alpha$ is not smooth, and neither are the functions $\mathcal{G}$ and $\mathcal{H}$. This issue can be overcome if the supercritical and subcritical phases are composed of a finite number of segments, and the critical phase is composed of a finite number of pointwise interfaces. In this case, weak solutions are in fact strong solutions, smooth in the supercritical set, and with smooth interface trajectories except at the initial time. In that case, the interfaces are the only points where $\mathcal{G}(\rho)$ and $\mathcal{H}(\alpha)$ are not smooth. However, this strategy restricts significantly the choices of initial profile, we therefore prove a regularization estimate that will allow us to use the mapping between weak solutions for general profiles $\rho^{\rm ini}$ and $\alpha^{\rm ini}$.

\subsubsection{Smoothing out solutions to the Stefan problems}
We first apply this mapping to smooth approximations of the weak solutions of the Stefan problem. To that aim, we consider smooth modifications of the functions $\mathcal{G}$ and $\mathcal{H}$ which are bounded away from $0$.
More precisely, consider an approximation $\mathcal{H}^\varepsilon$ of $\mathcal{H}$ on $[0,1]$, which satisfies  $\mathcal{H}^\varepsilon\in \mathcal{C}^\infty([0,1])$ for any $\epsilon >0$ and which is such that 
\begin{align}
\label{eq:H1epsilon}\tag{i}
& \lim_{\varepsilon \to 0}\norm{\mathcal{H}^\varepsilon-\mathcal{H}}_{\infty}=0,
 \\ \tag{ii}
&\text{for any } r\in [0,1], \quad \varepsilon\leq \mathcal{H}^\varepsilon(r)\leq1 , \quad \text{and} \quad 2\varepsilon\leq \frac{d}{dr}\mathcal{H}^\varepsilon(r)\leq 4. \label{eq:H3epsilon}
\end{align}
Further define on $[0,+\infty]$ the function $\mathcal{G}^\varepsilon(r)=\mathcal{H}^\varepsilon(\frac{r}{1+r})$ (recall \eqref{eq:relation}),  which satisfies analogous properties.

\medskip

Since all considered functions are then smooth for any positive $t>0$, tedious but straightforward computations yield the following result.

\begin{proposition}[Mapping between strong solutions]
\label{prop:classicalsol}
Fix $\varepsilon>0$. 

\begin{enumerate}
\item[$(\rho\to \alpha)$] Let  $\rho^{\rm ini}$ be a measurable function which satisfies Assumption \ref{ass:L1}  and define $\alpha^{\rm ini}$  through  \eqref{eq:alpha-ini}. Assume that $\rho_t^\varepsilon$ is the unique classical solution to 
\begin{equation}
\label{eq:classicalEXep}
\partial_t \rho_t^\varepsilon=\partial_u^2\mathcal{H}^\varepsilon(\rho_t^\varepsilon), \qquad  \rho_0^\varepsilon=\rho^{\rm ini},
\end{equation}
and define $\chi_t^\varepsilon$ implicitely by 
\begin{equation}
\label{eq:defnutep}
\int_0^{\chi_t^\varepsilon} (1-\rho_t^\varepsilon(u))\,du=\int_0^t \partial_u \mathcal{H}^\varepsilon (\rho_s^\varepsilon)(0)\,ds.
\end{equation}
Then, the function $\widehat{\alpha}_t^\varepsilon:=\alpha_{\theta^2 t}^\varepsilon$ built through the mapping \eqref{eq:mappingmacro} is the unique  classical solution of the parabolic equation 
\begin{equation}
\label{eq:classicalZRep}
\partial_t \widehat{\alpha}_t^\varepsilon=\partial_v^2\mathcal{G}^{\varepsilon}(\widehat{\alpha}_t^\varepsilon), \qquad  \widehat{\alpha}_0^\varepsilon=\alpha^{\rm ini}.
\end{equation} 
\item[$(\alpha\to\rho)$] Let $\alpha^{\rm ini}$ be a measurable function which satisfies \eqref{eq:assaslpha}, and define $\rho^{\rm ini}$ by inverting \eqref{eq:alpha-ini}. Assume that $\widehat{\alpha}_t^\varepsilon$  is the unique  classical solution of the parabolic equation 
\begin{equation}
\label{eq:classicalZRep2}
\partial_t \widehat{\alpha}_t^\varepsilon=\partial_v^2\mathcal{G}^{\varepsilon}(\widehat{\alpha}_t^\varepsilon), \qquad  \widehat{\alpha}_0^\varepsilon=\alpha^{\rm ini},
\end{equation}  and define $\chi_t^\varepsilon$ explicitely by \[\chi_t= \theta^{-1} \int_0^t \partial_v \mathcal{G}^\varepsilon (\widehat{\alpha}_{s  \theta^{-2}}) (0) ds=\theta \big \langle v, \widehat{\alpha}_{t \theta^{-2}}-\widehat{\alpha}_0 \big \rangle.\]  
Then, the function $\rho_t^\varepsilon$ built through the mapping \eqref{eq:mappingmacro} with $\alpha_t^\varepsilon:=\widehat{\alpha}_{t\theta^{-2}}$ is the unique classical solution to
\begin{equation}
\label{eq:classicalEXep2}
\partial_t \rho_t^\varepsilon=\partial_u^2\mathcal{H}^\varepsilon(\rho_t^\varepsilon), \qquad  \rho_0^\varepsilon=\rho^{\rm ini}.
\end{equation}
\end{enumerate} 

\end{proposition}
\mm{We note that the existence of classical solutions to equations \eqref{eq:classicalZRep} and \eqref{eq:classicalEXep2} is standard, thanks to the fact that the diffusion coefficient is smooth, bounded, and bounded away from $0$,  see for instance \cite[Section 3.1.1, p.~26]{Vasquez07}, or for a more detailed proof \cite[Theorem 6.1, p.~452]{Lady68}.} 
Following Proposition \ref{prop:classicalsol}, for  an initial profile $\rho^{\rm ini}$ bounded away from 1 as in Assumption \ref{ass:L1}, and for any $\varepsilon>0$, we denote by $\rho^\varepsilon$ and $\alpha^\varepsilon$ the classical solutions to \eqref{eq:classicalEXep} and \eqref{eq:classicalZRep} respectively, with initial conditions $\rho^{\rm ini}$ and the corresponding $\alpha^{\rm ini}$. We now state and prove the following convergence result.
\begin{proposition}
\label{prop:L2epsilon}
Fix $T>0$. Then,  uniformly in $t\in [0,T]$
\begin{equation}\label{eq:conv-rho}\rho_t^\varepsilon \xrightarrow[\epsilon \to 0]{L^2(\T)}\rho_t, \qquad \mathcal{H}^\varepsilon(\rho_t^\varepsilon) \xrightarrow[\epsilon \to 0]{L^2(\T)}\mathcal{H}(\rho_t) \end{equation}
as well as 
\begin{equation}\label{eq:conv-alpha}\alpha_t^\varepsilon \xrightarrow[\epsilon \to 0]{L^2(\T)}\alpha_t, \qquad \mathcal{G}^\varepsilon(\alpha_t^\varepsilon) \xrightarrow[\epsilon \to 0]{L^2(\T)}\mathcal{G}(\alpha_t), \end{equation}
where $ \rho_t$, $\alpha_t$ are the unique weak solutions to \eqref{eq:PDEstrong} and \eqref{eq:PDEstrongZR} in the sense of Definition \ref{Def:PDEweak}.
\end{proposition}
\begin{proof}
We will only prove the first two convergences in  \eqref{eq:conv-rho}, the other ones are proved analogously. Further note that the second one in \eqref{eq:conv-rho} follows straightforwardly from the first one, together with the boundedness of the derivative of $\mathcal{H}^{\varepsilon}$ (condition \eqref{eq:H3epsilon} above). Indeed, for any $t$, 
\begin{align*}
\int_0^1&\cro{\mathcal{H}^\varepsilon(\rho_t^\varepsilon)-\mathcal{H}(\rho_t)}^2(u) du \\ &\leq 2 \int_0^1\cro{\mathcal{H}^\varepsilon(\rho_t^\varepsilon)-\mathcal{H}^\varepsilon(\rho_t)}^2(u) du+2 \int_0^1\cro{\mathcal{H}^\varepsilon(\rho_t)-\mathcal{H}(\rho_t)}^2(u) du\\
&\leq 32 \int_0^1\cro{\rho_t^\varepsilon -\rho_t}^2(u) du+2 \norm{\mathcal{H}^\varepsilon-\mathcal{H}}_{\infty}^2.
\end{align*}
Both terms in the right hand side vanish, according to \eqref{eq:H1epsilon} and the first convergence result \eqref{eq:conv-rho}.

We now prove that $\rho_t^\varepsilon \to\rho_t$ in $L^2(\T)$ by adapting arguments given \emph{e.g.}~in  \cite[Theorem 4.4, Appendix 2]{klscaling}, to prove uniqueness of weak solutions of a nonlinear parabolic equation. For any $k\in \Z$, define on $\T$ the functions $\psi_k(u)=e^{2i\pi k u}$.
Define also
\[\delta^\varepsilon_t:=\rho_t^\varepsilon -\rho_t,\qquad \delta\mathcal{H}^\varepsilon_t:=\mathcal{H}^\varepsilon(\rho_t^\varepsilon)-\mathcal{H}(\rho_t).\] 
For any $L>0$, further introduce  the function 
 \[R_L(t)\equiv \lj{R_{L,\varepsilon} (t) := }\sum_{k\in\Z}\frac{L}{(k^2+1)(k^2+L)}\ang{\psi_k,\delta^\varepsilon_t}^2,\]
 which converges as $L\to\infty$ to 
\lj{ \[R(t)\equiv R_\varepsilon (t) := \sum_{k\in\Z}\frac{\ang{\psi_k,\delta^\varepsilon_t}^2}{k^2+1}\]}
\lj{by the dominated convergence theorem} because $\delta^\varepsilon_t$ is bounded \lj{by $2$, and in particular, $\ang{\psi_k,\delta^\varepsilon_t}$ is  bounded by some constant uniformly in $k$}. Furthermore, for any $L>0$, since $\rho^\varepsilon$ and $\rho$ are weak solutions to \eqref{eq:classicalEXep} and \eqref{eq:PDEstrong} respectively,
 \begin{align}
\frac{d}{dt}R_L(t)&=-8\pi^2\sum_{k\in\Z}\frac{k^2L}{(k^2+1)(k^2+L)}\ang{\psi_k,\delta^\varepsilon_t}\ang{\psi_k,\delta\mathcal{H}^\varepsilon_t}\notag\\ 
&\leq -8\pi^2 \ang{\delta^\varepsilon_t,\delta\mathcal{H}^\varepsilon_t}
+8\pi^2 \sum_{k\in\Z}\frac{L}{(k^2+1)(k^2+L)}\ang{\psi_k,\delta^\varepsilon_t}\ang{\psi_k,\delta\mathcal{H}^\varepsilon_t}\notag\\
&\quad +8\pi^2 \sum_{k\in\Z}\frac{k^2}{k^2+L}\ang{\psi_k,\delta^\varepsilon_t}\ang{\psi_k,\delta\mathcal{H}^\varepsilon_t}.\label{eq:RLT2}
 \end{align}
\ccl{We will ultimately let $L\to\infty$, so that the last term above will vanish.} We now bound the second term by the Cauchy-Schwarz inequality and we get
\begin{align}&8\pi^2 \sum_{k\in\Z}\frac{L}{(k^2+1)(k^2+L)}\ang{\psi_k,\delta^\varepsilon_t}\ang{\psi_k,\delta\mathcal{H}^\varepsilon_t}\notag\\
&\leq  16\pi^2  R_L(t)+\pi^2 \sum_{k\in\Z}\frac{L}{(k^2+1)(k^2+L)}\ang{\psi_k,\delta\mathcal{H}^\varepsilon_t}^2\notag\\
&\leq  16\pi^2  R_L(t)+2\pi^2 \sum_{k\in\Z}\ang{\psi_k,\mathcal{H}^\varepsilon(\rho_t^\varepsilon)-\mathcal{H}^\varepsilon(\rho_t)}^2+ 2\pi^2\sum_{k\in\Z}\frac{1}{k^2+1}\ang{\psi_k,\mathcal{H}^\varepsilon(\rho_t)-\mathcal{H}(\rho_t)}^2\notag\\
&\leq  16\pi^2 R_L(t)+2\pi^2 \int_0^1\big(\mathcal{H}^\varepsilon(\rho^\varepsilon_t)-\mathcal{H}^\varepsilon(\rho_t)\big)^2(u)du+ \ccl{C\norm{\mathcal{H}^\varepsilon-\mathcal{H}}_{\infty}^2}\notag,
\end{align} 
\ccl{where in the last inequality we used the Plancherel-Parseval identity and the fact that the $\psi_k$ are bounded by $1$.}
Regarding the second term, we can use \eqref{eq:H3epsilon} to obtain
\begin{align*}
2\pi^2\int_0^1(\mathcal{H}^\varepsilon\big(\rho^\varepsilon_t)-\mathcal{H}^\varepsilon(\rho_t)\big)^2(u)du&\leq8\pi^2\int_0^1\delta_t^\varepsilon\big(\mathcal{H}^\varepsilon(\rho^\varepsilon_t)-\mathcal{H}^\varepsilon(\rho_t)\big)(u)du\\
&\leq 16\pi^2 \ccl{\norm{\mathcal{H}^\varepsilon-\mathcal{H}}_{\infty}^2}+8\pi^2\int_0^1\delta_t^\varepsilon\delta\mathcal{H}^\varepsilon_t(u)du.
\end{align*}
\ccl{The second term cancels out with the first term in the right-hand side of \eqref{eq:RLT2}, so that putting together those bounds,} we obtain for a constant $C>0$ and $\varepsilon\leq 1$
\[R_L(t)\leq 16\pi^2 \int_0^t R_L(s) ds + Ct\ccl{\norm{\mathcal{H}^\varepsilon-\mathcal{H}}_{\infty}^2} + 8\pi^2\int_0^t \sum_{k\in\Z}\frac{k^2}{k^2+L}\<\psi_k,\delta^\varepsilon_t\>\<\psi_k,\delta\mathcal{H}^\varepsilon_t\>.\]
Since both $\delta^\varepsilon_t$ and $\delta\mathcal{H}^\varepsilon_t$ are in $L^2(\T\times [0,T])$, $\sum_k\<\psi_k,\delta^\varepsilon_t\>^2$ and $\sum_k\ang{\psi_k,\delta\mathcal{H}^\varepsilon_t}^2$ are both finite, the last term above vanishes as $L\to\infty$. Hence, by Gronwall's inequality, as $L\to\infty$, we obtain
\[R(t)\leq C \ccl{\norm{\mathcal{H}^\varepsilon-\mathcal{H}}_{\infty}^2} e^{16\pi^2t}\]
which vanishes as $\varepsilon \to 0$ uniformly in $t\in [0,T]$. This shows that for any $k\in \N$, $\<\psi_k,\delta^\varepsilon_t\>$ vanishes as $\varepsilon \to 0$, which in turn, by the dominated convergence Theorem, yields that 
\[ \lim_{\varepsilon\to 0}\<\delta_t^\varepsilon,\delta_t^\varepsilon\>=\lim_{\varepsilon\to 0}\sum_{k\in \Z}\<\psi_k,\delta^\varepsilon_t\>^2=0\]
as wanted.
\end{proof}

\subsubsection{Mapping between weak solutions to \eqref{eq:PDEstrong} and \eqref{eq:PDEstrongZR}}

We are now ready to state the final general result about the one-to-one correspondence between weak solutions:

\begin{proposition}[Mapping between weak solutions]\label{pro:mapping}
The result of Proposition \ref{prop:classicalsol} holds if one replaces $\mathcal{H}^\varepsilon$ by $\mathcal{H}$, $\mathcal{G}^\varepsilon$ by $\mathcal{G}$ and ``classical'' by ``weak'' in the sense of Definition \ref{Def:PDEweak}.

\end{proposition}

\begin{proof}
This result is a direct consequence of Propositions \ref{prop:classicalsol} and \ref{prop:L2epsilon}. Assume for example that  $\rho_t$  is the unique weak solution of \eqref{eq:PDEstrong} in the sense of Definition \ref{Def:PDEweak}. Then, the classical solution $\rho^\varepsilon$ to \eqref{eq:classicalEXep} converges in $L^2$ to $\rho$ according to Proposition \ref{prop:L2epsilon}. Then, define $\widehat{\alpha}_t^\varepsilon$ (resp. $\widehat{\alpha}_t$) the functions obtained applying the mapping \eqref{eq:mappingmacro}  to $\rho^\varepsilon_t$ (resp.\;$\rho_t$), according to Proposition \ref{prop:classicalsol}, $\widehat{\alpha}_t^\varepsilon$ is the unique classical solution to \eqref{eq:classicalZRep}, and according to Proposition \ref{prop:L2epsilon}, $ \rho_t^\varepsilon$ converges in $L^2(\T)$ to $\rho_t$ as $\varepsilon \to 0$. To conclude the proof, still according to  Proposition \ref{prop:L2epsilon}, it remains to show that $\alpha_t^\varepsilon$ converges in $L^2(\T)$ to $\alpha_t$ as $\varepsilon \to 0$, which will identify it as the weak solution of \eqref{eq:PDEstrongZR}. Fix $t>0$, we denote by $v_t^\varepsilon$, $\chi_t^\varepsilon$ the quantities relative to $\rho^\varepsilon$ defined by \eqref{eq:defvt} and \eqref{eq:defnutep}, which yields 
\[\int_0^1 (\alpha_t^\varepsilon- \alpha_t)^2(v)dv=\int_0^1 \pa{\frac{\rho_t^\varepsilon}{1-\rho_t^\varepsilon}\circ (v_t^\varepsilon)^{-1}- \frac{\rho_t}{1-\rho_t}\circ v_t^{-1}}^2(v)dv.\]
First note that since $\chi_t=\theta \big \langle v, \widehat{\alpha}_{t \theta^{-2} }-\alpha_0 \big \rangle$ and that the same identity holds for $\chi_t^\varepsilon$ and $\widehat{\alpha}^\varepsilon$, we can write for any $T>0$, as a consequence of Proposition  \ref{prop:L2epsilon}
\[\limsup_{\varepsilon \to 0}\sup_{t\in [0,T]}|\chi_t-\chi_t^\varepsilon|= 0.\]
In particular, still as a consequence of Proposition  \ref{prop:L2epsilon} and identity \eqref{eq:defvt}, 
\[\limsup_{\varepsilon \to 0}\sup_{u\in \T, \, t\in [0,T]}|v_t(u) -v_t^\varepsilon(u) |= 0.\]
so that since the mappings $v_t$, $v_t^\varepsilon$ are strictly increasing, 
\[\limsup_{\varepsilon \to 0}\sup_{v\in \T, \, t\in [0,T]}|v_t ^{-1}(v) -(v_t^\varepsilon)^{-1}(v) |= 0.\] 
This, together with the $L^2$--convergence of  $ \rho_t^\varepsilon$ to $\rho_t$ and the fact that $\rho_t$, $\rho_t^\varepsilon$ are bounded away from $1$ uniformly, proves the result.\end{proof}

\subsection{Proof of Theorem \ref{thm:ex} in the symmetric case}\label{sec:conclsym}

We can now provide an alternative proof of Theorem \ref{thm:ex} in the symmetric case, using the hydrodynamic limit for the symmetric facilitated zero-range process provided by Theorem \ref{thm:zr} (which will be proved independently in Section \ref{sec:hydroZR}). 

\begin{proof}[Proof of Theorem \ref{thm:ex} in the symmetric case] As in the proof of  Lemma \ref{lem:TP}, we construct a sequence of random elements $\{\eta^N,N \geq 1\}$ on a common probability space $(\Sigma, \mc{F}, \nu^\star)$ such that $\eta^N (\sigma)$ has distribution $\nu_N$, and  such that
		\[\lim_{N\to\infty} \frac{1}{N} \sum_{x\in \T_N} \eta^N_x (\cdot) \delta_{x/N}(du) = \rho^{\rm ini} (u) du, \qquad \text{$\nu^\star$ -- a.s.}\]
		under the weak topology.  We start from the observation that for every $\varphi \in C(\T)$,
		\begin{multline}\label{aa}
		\frac{1}{N}\sum_{x=1}^N\eta_x(t)\varphi (\tfrac x N)= \frac{1}{N} \sum_{y=1}^{M} \sum_{x=X_y (t) + 1}^{X_{y+1} (t) - 1} \varphi (\tfrac x N) 
		\\= \frac{1}{N}\sum_{y=1}^M \omega_y (t) \varphi \Big( \frac{X_y (t)} N\Big)+\mathcal{O}\Big(\frac1N\sup_{y} \omega_y (t) \Big).
		\end{multline}
	By Lemma \ref{lem:ZRdensity}, the error term above converges to zero in probability as $N \rightarrow \infty$.  Since \[X_y (t) = \sum_{y^\prime = 1}^{y-1} \omega_{y'} (t) + X_1 (t) + y-1,\]  by Lemma \ref{lem:TP} and Theorem \ref{thm:zr}, \lj{ the first term on the right hand side of \eqref{aa} }converges in $\P_{\eta^N}$--probability to 
	\[\theta \int_{\T} \alpha_{t \theta^{-2}} (v) \varphi \Big( \chi_t + \theta \int_0^v \big(1 + \alpha_{t \theta^{-2}} (v^\prime)\big)\,dv^\prime \Big)\,dv, \]
	where $\alpha_t (v)$ is the unique weak solution to \eqref{eq:PDEstrongZR} with initial condition $\alpha^{\rm ini}$ defined in \eqref{eq:alpha-ini}.  Let
	\[u_t (v)= \chi_t+\theta \int_0^v(1+\alpha_{t \theta^{-2}})(v')dv'.\] 
Since the function $u_t$ is strictly increasing, it has an inverse denoted by $v_t (u)$.  Let 
\[
\rho_t (u) = \frac{\alpha_{t \theta^{-2}}}{1+\alpha_{t \theta^{-2}}} (v_t (u)),
\]
then, by Proposition \ref{pro:mapping}, $\rho_t$ is the unique weak solution to \eqref{eq:PDEstrong} and
\[
\lim_{N\to\infty} \, \frac{1}{N}\sum_{x=1}^N\eta_x(t)\varphi (\tfrac x N)= \int_{\T}\, \rho_t (u) \varphi (u)\,du
\]
in $\P_{\eta^N}$--probability.    Moreover, the above convergence also holds in $\P_{\nu_N}$--probability by the dominated convergence Theorem. This concludes the proof.
\end{proof}

\section{The asymmetric case} \label{sec:asym}

\subsection{Microscopic mapping} The microscopic mapping in the asymmetric case is the same as in the symmetric case except that we do not have the extra factor $\theta$ since the system is defined on the infinite lattice. As in the symmetric case, let us label the empty sites from left to right in the exclusion process $\{\eta (t)\}$ on $\bL_M=\Z$, as follows:  
\lj{let $X_0(0)$ be the position of the first empty site to the right of (or at) the origin at the initial time. Then, for any $k>0$ (resp.~$k<0$), $X_k(0)$ is the position of the $(k+1)$-th empty site to the right (resp.~the $k$-th empty site to the left) of the origin at the initial time. We follow the trajectory of each of these empty sites throughout the configuration, and denote by $X_k (t)$ the position of the $k$-th empty site at time $t$.}
 Since the jumps are nearest neighbor, the orders of the empty sites are preserved along the evolution of the process, \emph{i.e.}~for any $t \geq 0$,
\[\ldots< X_{-1} (t) < X_0 (t) < X_1 (t) < \ldots\] 
For $y \in \Z$, let
\[\omega_y (t): = X_{y+1}(t) - X_y (t) - 1.\]
Then, as in Lemma \ref{lem:mapping},  $\{\omega (t)\}_{t\geqslant 0}$ is a Markov process on $\Gamma_M=\N^\Z$ which evolves according to the generator $N \genzr$ defined in \eqref{eq:DefLM}.   On the infinite line, both processes have the same scaling, \textit{i.e.}~throughout this section, we choose $M=N$.

As stated in Theorem \ref{thm:ex}, we need the following assumption on the initial profile: 

\begin{assumption}\label{ass:L2}
The initial density profile $\rho^{\rm ini}:\R\to[0,1]$ is integrable on $\R$, and bounded away from 1: 
\[ \rho^{\rm ini}(u)\leqslant \rho_\star <1, \qquad \text{for any }u\in\R.\]
\end{assumption}

Similar to Lemma \ref{lem:InitialLE}, the initial zero-range density profile in this case is given by 
\begin{equation}\label{ini_asy}
\alpha^{\rm ini} (v) = \frac{\rho^{\rm ini} (u)}{1 - \rho^{\rm ini} (u)}, \quad u \in \R,
\end{equation}
where
\[v = v(u) = \int_0^u (1-\rho^{\rm ini} (u^\prime)) du^\prime.\]
More precisely, for any continuous test function $\varphi:\R \to \R$ with compact support and for any $\delta >0$,
\[\lim_{N\to\infty}\P_{\nu_N}\bigg( \bigg| \frac{1}{N} \sum_{y \in \Z} \varphi\Big(\frac{y}{M}\Big) \omega_y(0) - \int_{\R}\varphi(v)\alpha^{\rm ini}(v)dv\bigg| > \delta \bigg)=0. \]
Since the proof is exactly the same as that of Lemma \ref{lem:InitialLE}, we do not repeat  it here.

\begin{lemma}[Law of large numbers]\label{lem:holes}
Let $\rho^{\rm ini}:\R\to[0,1]$ satisfy Assumption \ref{ass:L2} and $\alpha^{\rm ini}$ be defined as in \eqref{ini_asy}. 
Consider the entropy solution  $\alpha_t(v)$ of the hyperbolic equation \eqref{zrp:Hydro} with initial condition $\alpha^{\rm ini}$ (see Definition \ref{def:entropy}). 
	Then, for any $t\geqslant 0$, and any $\epsilon > 0$,
	\[
		\lim_{N \rightarrow \infty} \bb{P}_{\nu_N} \bigg( \bigg| \frac{ X_0 (t)}N - \sigma_t\bigg| > \epsilon  \bigg) = 0, \qquad \text{where } \sigma_t = \int_0^\infty (\alpha^{\rm ini}(v) - \alpha_t(v))\,dv.
	\]

\end{lemma}

\begin{proof}[Proof of Lemma \ref{lem:holes}]
As in the proof of Lemma \ref{lem:TP}, let $J^\omega_{y,y+1} (t)$ be the net current across the bond $(y,y+1)$ up to time $t$ throughout the evolution of $\omega$. Then 
\[X_0 (t) - X_0 (0) = - J^\omega_{0,1} (t).\]
Since $X_0 (0) / N \rightarrow 0$ in probability as $N \rightarrow \infty$, we only need to prove for every $\varepsilon > 0$,
\[\lim_{N \rightarrow \infty} \bb{P}_{\nu_N} \bigg( \bigg| \frac{ J_{0,1}^\omega (t)}{N} + \sigma_t\bigg| > \epsilon  \bigg) = 0.\]
By the conservation of the number of particles,
\[N^{-1} J_{0,1}^\omega (t) = N^{-1} \Big(\sum_{1 \leq y \leq K N} +  \sum_{y > KN} \Big) \{\omega_y (t) - \omega_y (0)\}.\]
The first sum above converges in probability, as $N \rightarrow \infty$ followed by $K \rightarrow \infty$, to $-\sigma_t$. For the second sum, it is bounded by 
\lj{\begin{equation}\label{eq:new2} N^{-1} \sum_{y > K N} [\omega_y (t) + \omega_y (0)] \leq  2 N^{-1}  \sum_{y \geq (K -2t) N}  \omega_y (0) + N^{-1}  \mathcal{N}_{t,K} ,\end{equation}
where $\mathcal{N}_{t,K}$ is the number of particles whose positions are smaller than $(K -2t) N$ at time $t$ but are greater than $KN$ initially.  Because the distribution of  $\omega_y(0)$ is stochastically bounded by a geometric distribution with parameter $\alpha_\star$ and the displacement of a particle to the left is bounded by a Poisson process with parameter $1$, by a standard large deviation argument, the expectation of  $\mathcal{N}_{t,K}$ is bounded uniformly in $N$. In particular, $N^{-1}  \mathcal{N}_{t,K}$ converges to zero in probability as $N \rightarrow \infty$. The first term on the right hand side of \eqref{eq:new2} also vanishes in the limit by the integrability of $\alpha^{\rm ini}$. This concludes the proof.}
\end{proof}

\subsection{Macroscopic mapping}

In the section, we mainly prove the following result.

\begin{lemma}[$\alpha \mapsto \rho$]\label{lem:macromapping}
	Let $\alpha_t (v)$ be the unique entropy solution to \eqref{zrp:Hydro} with initial data $\alpha^{\rm ini}$, and for any $t \geq 0,\, u \in \R$, let $v_t (u)$ be the unique point such that 
	\[u= \sigma_t  + \int_0^{v_t (u)} (1+\alpha_t (v^\prime)) d v^\prime,\] 
	with $\sigma_t$ defined as in Lemma \ref{lem:holes}. Then $\rho_t(u)$ defined as
	\[\rho_t(u) =  \frac{\alpha_t \circ v_t (u)}{1+\alpha_t \circ v_t (u)}  \]
	is the entropy solution to \eqref{ep:Hydro} with initial data $\rho^{\rm ini}$, where $\rho^{\rm ini}$ is defined by inverting \eqref{ini_asy}, \textit{i.e.}
	\[\rho^{\rm ini} (u) = \frac{\alpha^{\rm ini} (v_0 (u))}{1 + \alpha^{\rm ini} (v_0(u))}.\]
\end{lemma}

\medspace

As in the symmetric case, before proving Lemma \ref{lem:macromapping}, we first smooth out solutions to the hyperbolic equations.  Let $\phi \in C_c^\infty (\R)$ be a symmetric mollifier, \textit{i.e.}~which satisfies
\lj{\[\phi (u) \geq 0, \qquad \phi (u) = \phi(-u), \qquad \supp (\phi) \subset (-1,1),\qquad \int_{\R} \phi (u) du = 1. \]}
For any $\varepsilon > 0$, take $\phi_\varepsilon (u) = \varepsilon^{-1} \phi (u/\varepsilon).$   Let   $\alpha^{\rm ini}_\varepsilon = \alpha^{\rm ini} * \phi_\varepsilon$ and
\[ \mathfrak{H}^{\varepsilon} (r)=  (\mathfrak{H} * \phi_\varepsilon) (r):=\int_{\R} \mathfrak{H} (r - r^\prime) \phi_\varepsilon (r^\prime) d r^\prime,\qquad \mathcal{G}^\varepsilon (r)  = (1+r) \mathfrak{H}^\varepsilon \big(\tfrac{r}{1+r}\big).\]
Note that we define $ \mathcal{G}^\varepsilon$ in the above way (recall also \eqref{eq:relation}) instead of convoluting  the function $\mathcal{G}$ with the mollifier $\phi_\varepsilon$,  in order to prove Lemma \ref{lem:alphatorho} below.

\begin{lemma}\label{lemma:alphaeps} Let $\alpha^{\rm ini} \in L^1 (\R) \cap L^\infty (\R) \cap BV (\R)$. Then there exists a unique smooth solution $\alpha^\varepsilon$ to 
		\begin{equation}\label{pertur:zrp}
			\begin{cases}
				\partial_t \alpha^\varepsilon_t + (2p-1) \partial_v \mathcal{G}^\varepsilon (\alpha^\varepsilon_t) = \displaystyle \varepsilon \; \partial_v^2 \bigg(\frac{\alpha^\varepsilon_t }{1+\alpha^\varepsilon_t}\bigg), \quad t > 0, v \in \R,\\
				\alpha^\varepsilon_0 = \alpha^{\rm ini}_\varepsilon.
			\end{cases}
		\end{equation} 
Moreover, there exists a subsequence still denoted by $\alpha^\varepsilon$ such that 
\begin{equation}\label{conver_alpha}
	\lim_{\varepsilon \to 0} \alpha^\varepsilon = \alpha \quad \text{in }\; L^1_{\rm loc} (\R_+ \times \R) \quad \text{and} \quad \text{a.e. in }\; \R_+ \times \R,
\end{equation}
where $\alpha$ is the unique entropy solution to \eqref{zrp:Hydro} with initial distribution $\alpha^{\rm ini}$.
\end{lemma}

\begin{proof} We only sketch the proof and refer the readers to \cite[Section 2.4]{malek1996} for details. The existence and uniqueness of the smooth solution follow immediately from  \cite[Lemma 2.4.13]{malek1996}. Moreover, for all $t \geq 0$ we have
		\begin{align}
			||\alpha^\varepsilon_t||_{L^\infty (\R)} & \leq ||\alpha^{\rm ini}||_{L^\infty (\R)},\\
				||\partial_v \alpha^\varepsilon_t||_{L^1 (\R)}& \leq TV (\alpha^{\rm ini}),\label{alpha_eps_v}\\
					||\partial_t \alpha^\varepsilon_t||_{L^1 (\R)}& \leq C\, TV (\alpha^{\rm ini}).
		\end{align}
In particular, the sequence $\{\alpha^\varepsilon\; : \;  \varepsilon > 0\}$ is bounded in the space $L^\infty ((0,\infty) \times \R) \cap W^{1,1}_{\rm loc} ((0,\infty) \times \R)$, where $W^{1,1}_{\rm loc} ((0,\infty) \times \R)$ is the local version of the standard Sobolev space, \textit{i.e.}~the set of functions $f \in L^1((0,\infty) \times \R)$ such that the restriction of $f$ to any compact set is in $L^1$, as well as the restriction of $\nabla f$, \emph{cf.}~\cite{malek1996} for example. By a \lj{diagonal} argument, there exists a subsequence, still denoted by $\alpha^\varepsilon$, and a function $\alpha$ such that  \eqref{conver_alpha} holds. Moreover, the almost everywhere convergence in \eqref{conver_alpha} is uniform over compacts sets in $\R_+ \times \R$.

Next we show the limit $\alpha$  is the entropy solution to \eqref{zrp:Hydro} with initial condition $\alpha^{\rm ini}$. Let $(E,F^\varepsilon)$ be a convex entropy-entropy flux pair, \textit{i.e.}: take $E\in C^2(\R)$ any convex function, and define $F^\varepsilon$ through the relation \lj{$(\mathcal{G}^\varepsilon)^\prime E^\prime  = (F^\varepsilon)^\prime$}. Multiplying \eqref{pertur:zrp} by $E^\prime (\alpha^\varepsilon)$, a straightforward computation gives
\begin{align*}
\partial_t E(\alpha^{\rm \varepsilon}_t)  + (2p-1) \partial_v F^\varepsilon (\alpha^{\rm \varepsilon}_t)& = \varepsilon E^\prime (\alpha^{\rm \varepsilon}_t) \partial_v^2 \bigg(\frac{\alpha^\varepsilon_t }{1+\alpha^\varepsilon_t}\bigg)\\
&= \varepsilon \partial_v^2 \tilde{E} (\alpha^\varepsilon_t) - \varepsilon E^{\prime \prime} (\alpha^\varepsilon_t) (1+\alpha^\varepsilon_t)^{-2} (\partial_v \alpha^\varepsilon_t)^2 \leq  \varepsilon \partial_v^2 \tilde{E} (\alpha^\varepsilon_t),
\end{align*}
where $\tilde{E} (\alpha)$ is \lj{a primitive} function of $(1+\alpha)^{-2}E^\prime (\alpha)$, and the last inequality comes from the fact that $E$ is convex. Recall that we denote by $\langle \cdot,\cdot\rangle$ the standard $L^2(\R)$ scalar product, the weak formula of the above inequality for $0 \leq \varphi \in C_c^2 (\R \times \R)$ reads
\begin{equation}\label{entropy_ineqn}
	 \langle E(\alpha^\varepsilon_0 ), \varphi_0\rangle+ \int_0^\infty \langle E(\alpha^\varepsilon_t), \partial_t \varphi_t\rangle +(2p-1) \langle F^\varepsilon (\alpha^\varepsilon_t), \partial_v \varphi_t \rangle\,dt \geq - \varepsilon \int_0^\infty \langle \tilde{E} (\alpha^\varepsilon_t) \partial_v^2 \varphi_t\rangle\,dt.
\end{equation}
It can be easily checked that the right hand side of the above inequality converges to zero as $\varepsilon \rightarrow 0$ and that
\[\lim_{\varepsilon \rightarrow 0} E(\alpha^\varepsilon)  = E (\alpha) \quad \text{in $L^1_{\rm loc} (\R_+ \times \R)$}.\]
Moreover,
\begin{equation}\label{f_epsilon}
\lim_{\varepsilon \rightarrow 0} F^\varepsilon (\alpha^\varepsilon)  = F (\alpha) \quad \text{in $L^1_{\rm loc} (\R_+ \times \R)$},
\end{equation}
where $F$ satisfies $E^\prime \mathcal{G}^\prime = F^\prime$. To prove the last equation, first note that
\begin{multline*}
	|F^\varepsilon (\alpha^\varepsilon_t)  - F (\alpha_t)| = \Big| \int_0^{\alpha^\varepsilon_t} E^\prime (\beta) (\mathcal{G}^\varepsilon)^\prime (\beta) d \beta -  \int_0^{\alpha_t} E^\prime (\beta) \mathcal{G}^\prime (\beta) d \beta\Big|\\
	\leq \int_{-K}^K \big| E^\prime (\beta)  \big| \big|  (\mathcal{G}^\varepsilon)^\prime  (\beta) -   \mathcal{G}^\prime (\beta) \big| d \beta + \Big| \int_{\alpha_t}^{\alpha^\varepsilon_t} E^\prime (\beta) \mathcal{G}^\prime (\beta) d\beta \Big|,
\end{multline*}
where $K = ||\alpha^{\rm ini}||_{L^\infty (\R)}$. This yields
\[\lim_{\varepsilon \rightarrow 0} F^\varepsilon (\alpha^\varepsilon) = F (\alpha)\]
uniformly over compact sets  in $\R_+ \times \R$, which proves \eqref{f_epsilon} by dominated convergence theorem. Letting  $\varepsilon \rightarrow 0$ in \eqref{entropy_ineqn}, we have that $\alpha$ is the unique entropy solution to \eqref{zrp:Hydro} with initial distribution $\alpha^{\rm ini}$. This completes the proof.
\end{proof}

 Let us define
\[\sigma_t^\varepsilon = \int_0^\infty (\alpha^{\rm ini}_\varepsilon (v) - \alpha^\varepsilon_t(v))\,dv,\]
and for any $t \geq 0,\, u \in \R$, let $v^\varepsilon_t (u)$ be the unique point such that 
\[u= \sigma_t^\varepsilon  + \int_0^{v_t^\varepsilon (u)} (1+\alpha_t^\varepsilon (v^\prime)) d v^\prime.\] 
Finally, define
\begin{equation}
\label{eq:defrhoep}
\rho^\varepsilon_t(u) =  \frac{\alpha_t^\varepsilon \circ v_t^\varepsilon (u)}{1+\alpha_t^\varepsilon \circ v_t^\varepsilon (u)}.
\end{equation}
After a tedious but elementary calculation, we get the following
\begin{lemma}[$\alpha^\varepsilon \mapsto \rho^\varepsilon$]\label{lem:alphatorho}
The function $\rho^\varepsilon$ defined by \eqref{eq:defrhoep} is the unique weak solution to the following perturbation of the hyperbolic equation \eqref{ep:Hydro}
\begin{equation}\label{pertur:ep}
	\begin{cases}
		\partial_t \rho^\varepsilon_t + (2p-1) \partial_u \mathfrak{H}^\varepsilon (\rho^\varepsilon_t) = \varepsilon \partial_u^2 \rho^\varepsilon_t, \quad t > 0, u \in \R,\\
		\rho^\varepsilon_0= \rho^{\rm ini}_\varepsilon,
	\end{cases}
\end{equation} 
where
	\[\rho^{\rm ini}_\varepsilon (u) = \frac{\alpha^{\rm ini}_\varepsilon (v_0^\varepsilon (u))}{1 + \alpha^{\rm ini}_\varepsilon (v_0^\varepsilon(u))}.\]
\end{lemma}

\medspace

In the same spirit as in Lemma \ref{lemma:alphaeps},  there exists a subsequence still denoted by $\rho^\varepsilon$ such that 
\begin{equation}
	\lim_{\varepsilon \to 0} \rho^\varepsilon = \widetilde{\rho} \quad {\rm in}\; L^1_{\rm loc} (\R_+ \times \R) \quad \text{and} \quad \text{a.e.~in }\; \R_+ \times \R,
\end{equation}
where $\widetilde{\rho}$ is the unique entropy solution to \eqref{ep:Hydro} with initial condition $\rho^{\rm ini}$ as defined in Lemma \ref{lem:macromapping}.
We are now ready to prove Lemma \ref{lem:macromapping}.
\begin{proof}[Proof of Lemma \ref{lem:macromapping}]  We only need to prove
\[
\lim_{\varepsilon \to 0} \rho^\varepsilon = \rho \quad {\rm in}\; L^1_{\rm loc} (\R_+ \times \R)
\]
or equivalently,
\begin{equation}\label{convergence}
	\lim_{\varepsilon \to 0} \alpha_t^\varepsilon \circ v_t^\varepsilon = \alpha_t \circ v_t \quad {\rm in}\; L^1_{\rm loc} (\R_+ \times \R).
\end{equation}
Indeed, this implies $\rho = \widetilde{\rho}$ and whence concludes the proof. 

We claim that 
\[	\lim_{\varepsilon \to 0} \sigma^\varepsilon_t = \sigma_t \quad \text{for a.e.} \;t \in \R_+.\]Indeed, we have $\lim_{\varepsilon \rightarrow 0} ||\alpha^{\rm ini}_\varepsilon - \alpha^{\rm ini}||_{L^1 (\R)} = 0$ and for every $K > 0$,
\begin{multline*}
\limsup_{\varepsilon\to 0}	\Big| \int_0^\infty \alpha^{\varepsilon}_t (v) - \alpha_t(v) dv \Big| \leq \limsup_{\varepsilon\to 0} \Big| \int_0^K  \alpha^{\varepsilon}_t (v) - \alpha_t(v) dv  \Big| \\
+ \limsup_{\varepsilon\to 0} \int_K^\infty \alpha^{\varepsilon}_t (v) dv + \int_K^\infty \alpha_t (v) dv.
\end{multline*}
The first term on the right-hand side is zero by Lemma \ref{lemma:alphaeps}, and the last term converges to zero as $K \rightarrow \infty$ by the integrability of $\alpha_t$. For the second term, we have
\begin{multline*}
\int_K^\infty \alpha^{\varepsilon}_t (v) dv \leq \int_{|v| > K} \alpha^{\varepsilon}_t (v) dv = \int_{\R} \alpha_t (v) dv - \int_{|v| \leq K} \alpha^{\varepsilon}_t (v) dv \\
 \leq \int_{|v| > K} \alpha_t (v) dv + \int_{|v| \leq K}\big| \alpha^{\varepsilon}_t (v) - \alpha_t(v) \big|dv,
\end{multline*}
which vanishes as $\varepsilon \rightarrow 0,\,K \rightarrow \infty.$ This proves the claim.
Together with \eqref{conver_alpha}, we have
\[\lim_{\varepsilon \to 0} u^\varepsilon_t (v) = u_t (v) \]
uniformly over compact intervals in the space variable for a.e.~$t \in \R_+$. Since $u^\varepsilon_t$ and $u_t$ are strictly increasing,  we have
\begin{equation}\label{v_eps_to_v}
\lim_{\varepsilon \to 0} v^\varepsilon_t (u) = v_t (u)
\end{equation}
uniformly over compact intervals in the space variable for a.e.~$t \in \R_+$. We now write, for any $K > 0$,
\begin{multline}\label{maa1}
	\int_{-K}^K |\alpha_t^\varepsilon \circ v_t^\varepsilon (u) -  \alpha_t \circ v_t (u)| du \leq
	\int_{-K}^K |\alpha_t^\varepsilon \circ v_t^\varepsilon (u) -  \alpha_t^\varepsilon \circ v_t (u)| du\\
	+ \int_{-K}^K |\alpha_t^\varepsilon \circ v_t (u) -  \alpha_t \circ v_t (u)| du.
\end{multline}
By \eqref{alpha_eps_v} and \eqref{v_eps_to_v}, the first term on the right-hand side converges to zero as $\varepsilon \rightarrow 0$. The second term on the right-hand side of \eqref{maa1} also vanishes in the limit by \eqref{conver_alpha}. Therefore, we have proved
\[	\lim_{\varepsilon \to 0} \alpha_t^\varepsilon \circ v_t^\varepsilon = \alpha_t \circ v_t \quad {\rm in}\; L^1_{ \rm loc} (  \R)\]
for a.e.~$t \in \R$. Since $\|\alpha^\varepsilon\|_{L^\infty (\R_+ \times \R)} \leq \|\alpha^{\rm ini}\|_{L^\infty (\R_+ \times \R)}$, \eqref{convergence} follows from the dominated convergence Theorem. This concludes the proof.
\end{proof}

\subsection{Proof of Theorem \ref{thm:ex} in the asymmetric case}\label{sec:conclasym}

\begin{proof}[Proof of Theorem \ref{thm:ex} in the asymmetric case]
	Observe that for every continuous function $\varphi: \bb{R} \rightarrow \bb{R}$ with compact support,
	\begin{equation}
		\begin{aligned}
			\frac{1}{N} \sum_{x\in\bb{Z}} \varphi(\tfrac{x}{N}) \eta_x (t)  = \frac{1}{N} \sum_{x\in\bb{Z}} \varphi(\tfrac{x}{N}) - \frac{1}{N} \sum_{y \in\bb{Z}} \varphi (X_y (t)/N).
		\end{aligned}
	\end{equation} 
	Following the same argument in Subsection \ref{sec:conclsym},
	\begin{equation}\label{mapping1}
		\lim_{N \rightarrow \infty} \frac{1}{N} \sum_{x\in\bb{Z}} \varphi (\tfrac{x}{N}) \eta_x (t) = \int_{\bb{R}} \varphi (u)du - \int_{\bb{R}} \varphi \Big( \sigma_t  + \int_0^v (1+\alpha_t (v^\prime)) d v^\prime  \Big) \,dv 
	\end{equation} in $\P_{\nu_N}$ -- probability.
	Making the change of variables $v \mapsto v_t (u)$ in \eqref{mapping1}, where $v_t(u)$ is defined in Lemma \ref{lem:macromapping}, we have
	\begin{equation}\label{mapping2}
		\lim_{N \rightarrow \infty} \frac{1}{N} \sum_{x\in\bb{Z}} \varphi (\tfrac{x}{N}) \eta_x (t) = \int_{\bb{R}} \varphi (u) \frac{\alpha_t \circ v_t (u)}{1+\alpha_t \circ v_t (u)} du  \quad \text{in $\P_{\nu_N}$ -- probability.}
	\end{equation}
We conclude the proof by Lemma \ref{lem:macromapping}.
\end{proof}

\section{Proof of Theorem \ref{thm:zr}: the hydrodynamic limits for the zero-range process} \label{sec:hydroZR}

\subsection{Basic coupling and initial distribution} 
\label{sec:coupling}

The proof of the hydrodynamic limits in both the symmetric and asymmetric cases for the zero-range process strongly relies on its attractiveness property. The latter allows to couple two processes $(\omega (t),\zeta (t))$, each one evolving according to the generator $\genzr$, in a way that preserves their order.  The generator of the coupled process $(\omega (t),\zeta (t))$ is $\tgenzr$ which  is defined by making a particle jump from site $x$ to $x+1$ (resp. $x-1$) at rate $p$ (resp. $p'$) 
\begin{itemize}
\item in both processes $\omega$ and $\zeta$ if both $\omega_x\geq 2$ and $\zeta_x\geq 2$,
\item only in $\omega$ and not in $\zeta$ if $\omega_x\geq 2$ and $\zeta_x < 2$,
\item only in $\zeta$ and not in $\omega$ if $\zeta_x\geq 2$ and $\omega_x < 2$.
\end{itemize}
This dynamics is usually referred to as the \emph{basic coupling}, we will not give the explicit form of the coupled generator $\tgenzr$, since it would require burdensome notations, and the dynamics description above is fairly straightforward.

 Denote by $\mathscr{I}$ (resp.~$\tilde{\mathscr{I}}$\,) the set of invariant measures for the zero-range process $\omega (t)$ (resp.~the coupled process $(\omega (t),\zeta (t))$), and by $\mathscr{S}$ (resp.~$\tilde{\mathscr{S}}$\,) the set of translation-invariant measures on $\Gamma_M$ (resp.~on $\Gamma_{M} \times \Gamma_M$).

\begin{lemma}\label{lem:ergodic}
\begin{enumerate}
\item  If $\mu \in \mathscr{I} \cap \mathscr{S}$, then $\mu$ has the following decomposition 
\[\mu = \lambda \mu_1 + (1-\lambda) \mu_2,\]
for some $\lambda \in [0,1]$, where $\mu_1 \in \mathscr{S}$	 is concentrated on the set of frozen configurations
\[\mu_1 \{\omega: \omega_y \leq 1, \; \forall y\}=1\]
and $\mu_2$	 is concentrated on the set of ergodic configurations and satisfies
\[\mu_2 = \int_1^\infty \mu_\alpha^\star\; \beta(d\alpha)\qquad \text{for some probability measure } \beta \text{ on } [1,\infty),\]
where $\mu_\alpha^\star$ is the equilibrium measure for the zero-range process defined in \eqref{eq:mualphastar}.
\medskip

\item If $\tilde{\mu} \in \tilde{\mathscr{I}} \cup \tilde{\mathscr{S}}$,  then $\tilde{\mu}$ has the following decomposition 
\[\tilde{\mu} = \lambda \tilde{\mu}_1 + (1-\lambda) \tilde{\mu}_2,\]
for some $\lambda \in [0,1]$, where $\tilde{\mu}_1$ is concentrated on the set of configurations whose two components are both frozen, and $\tilde{\mu}_2$ satisfies
\begin{equation}\label{ordered}
\tilde{\mu}_2 \{(\omega,\zeta): \omega \leq \zeta \;\text{or }\; \zeta \leq  \omega 
\}= 1.
\end{equation}
\end{enumerate}
\end{lemma}

\begin{proof}[Proof of Lemma \ref{lem:ergodic}]
To sketch the proof of the first point, we follow the strategy of \cite{funaki1999free}. The main hurdle is to show that if  $\mu \in \mathscr{I} \cap \mathscr{S}$, then  $\mu$ is concentrated on the union of frozen and ergodic configurations, \textit{i.e.}
\begin{equation}\label{support}
\mu (\omega_y = 0, \omega_{y^\prime} \geq 2 \;\; \text{for some $y,\;y^\prime$}) = 0.
\end{equation}
To show the latter, since $\mu$ is invariant for the evolution of the process $\omega (t)$, first note that \lj{for any $y \in \bb{L}_M$,}
\[\mu \big(\genzr \mathbbm{1}_{\{\omega_y = 0\}}\big)   = 0.\]
Direct calculations show that if $\frac12 \leq p <1$, then
\lj{\[\mu (\omega_y=0,\omega_{y + 1} \geq 2) = \mu (\omega_y=0,\omega_{y - 1} \geq 2) = 0.\]}
If $p = 1$, we only have
\[\mu (\omega_y=0,\omega_{y - 1} \geq 2) = 0.\]
However, by further considering
\[\mu \big(\genzr \mathbbm{1}_{\{\omega_y = 0, \omega_{y+1} = k\}}\big)   = 0,\quad k \geq 1,\]
we get
\[\mu (\omega_y=0,\omega_{y + 1} \geq 2) = 0.\]
Using the same induction argument as in the proof of \cite[Lemma 4.1]{funaki1999free}, it is not hard to prove  \eqref{support}.  Now we decompose $\mu$ by respectively restricting it to the set of  frozen configurations and to that of ergodic configurations. \lj{It is well known from \cite[Theorem 1.9]{Andjel82} that}  the invariant and translation invariant measures of the zero-range process $\omega (t)$ restricted to the set of ergodic configurations are a linear combination of the measures $\{\mu_\alpha^\star,\,\alpha \geq  1\}$. This proves 1.

We now consider point 2.  By \eqref{support}, \lj{the first and the second marginals of $\tilde{\mu}$ are concentrated either on $\{0,1\}^{\bb{L}_M}$ or on $\{1,2,\ldots\}^{\bb{L}_M}$. Therefore,} $\tilde{\mu}$ is concentrated on the set of configurations
\begin{multline*}
\big\{(\omega,\zeta): \omega_y \leq 1\;\text{and}\; \zeta_y \leq 1\;\text{for all}\;y\big\}  \cup \big\{(\omega,\zeta): \omega_y \geq 1\;\text{and}\; \zeta_y \geq 1\;\text{for all}\;y\big\}\\
\cup \big\{(\omega,\zeta): \omega_y \leq 1\;\text{and}\; \zeta_y \geq 1\;\text{for all}\;y\;\text{or}\; \omega_y \geq 1\;\text{and}\; \zeta_y \leq 1\;\text{for all}\;y\big\}.
\end{multline*}
Decompose $\tilde{\mu}$ by restricting it to the above three sets respectively. Obviously,  $\tilde{\mu}$ restricted to  the first set is concentrated on the set of configurations whose two components are both frozen, and $\tilde{\mu}$ restricted to  the last set  satisfies \eqref{ordered}. By \cite[Proposition 5.1]{Andjel82}, \lj{any invariant and translation invariant measure of the classic coupled zero range process is concentrated on the sets of configuration pairs which are ordered. Therefore,} $\tilde{\mu}$ restricted to  the second set above satisfies \eqref{ordered}.  This completes the proof of Lemma \ref{lem:ergodic}.
\end{proof}

Because we used the hydrodynamic limit of the zero-range process through a mapping to prove Theorem \ref{thm:ex}, and because the image through the mapping of Bernoulli product measures fitting $\rho^{\rm ini}$ is \emph{not} a geometric product measure, (although it \emph{does} satisfy \eqref{eq:initmeasure}, see Lemma \ref{lem:InitialLE}) we need to prove Theorem \ref{thm:zr} for fairly general initial distributions. However, because of the degeneracy of the process, entropy tools cannot be used to prove Theorem \ref{thm:zr}, \ccl{be it} in the symmetric or asymmetric case. 
\lj{Because of this, assuming that the initial distribution is product is fundamental in order to apply the methods of \cite{rezakhanlou91}.}
Indeed, it is used to dominate by an equilibrium state and prove the two blocks-estimate in the symmetric case, and to prove the initial condition \eqref{eq:initZR} in the asymmetric case. We therefore first use attractiveness to prove the following result, which states that if two initial distributions fit the same initial profile $\alpha^{\rm ini}$, they have the same hydrodynamic limit.

\begin{lemma}
\label{lem:initprofileZR}
Consider two zero-range processes $\omega$ and $\zeta$ on $\mathbb{L}_M$ with respective initial distributions $\mu_M$, $\mu'_M$, we denote by $\widetilde{\mu}_M=\mu_M\otimes \mu'_M$ the product measure for the pair $(\omega, \zeta)$. Assume that for any compactly supported test function $\varphi\in C^2_c(\mathbb L)$, any $\varepsilon>0$
\begin{equation}
\label{eq:initprofileZRvanish}
\lim_{M\to\infty} \widetilde{\mu}_M \bigg(\bigg|\frac{1}{M}\sum_{y\in \bL_M}(\omega_y- \zeta_y)\varphi (\tfrac yM)\bigg|>\varepsilon\bigg)=0.
\end{equation}
Then, in both symmetric (I) and asymmetric (II) cases, if  \lj{\eqref{lln} holds for $\mu_M$, then it also holds for $\mu'_M$.}
\end{lemma}

\begin{proof}[Proof of Lemma \ref{lem:initprofileZR}]
To prove this result, we adapt ideas from  \cite{bahadoran2004blockage}. We start with the symmetric case, which takes place on the finite ring $\bb L_M=\T_M$ and will allow us to illustrate the main argument. We couple the two processes $\{\omega (t),\zeta (t)\}_{t \geq 0}$ by the basic coupling described in the beginning of the section, and we assume that $\omega$ admits a hydrodynamic limit, so that in particular we must have, for any 
\begin{equation}
\label{eq:defC0}
C_0>2\sup_v \alpha^{\rm ini}(v)
\end{equation} 
and using $\varphi\equiv 1$ in \eqref{eq:initmeasure} and \eqref{eq:initprofileZRvanish}, that
\begin{equation}
\label{eq:initprofileZRvanish2}
\lim_{M\to\infty} \widetilde{\mu}_M \bigg(\frac{1}{M}\sum_{y\in \T_M}(\omega_y+\zeta_y )>C_0\bigg)=0.
\end{equation} 
We claim that, denoting by $\Delta_{y,y'}(t)$ the particle difference between the two processes in the segment $[y,y']$ at time $t$,
\begin{equation}
\label{eq:Deltat0}
\max_{y, y'\in \T_M}\Delta_{y,y'}(t):=\max_{y, y'\in \T_M} \bigg| \sum_{z=y}^{y^\prime} \big(\omega_z (t) - \zeta_z (t)\big) \bigg| \leq 
\max_{y, y'\in \T_M}\Delta_{y,y'}(0).
\end{equation}
This \lj{inequality} is clear, because the left-hand side is decreasing in time under the basic coupling: consider a segment $[y, y']$ which realizes the maximum of $\Delta_{y,y'}(t)$, and assume for example that in this segment $\omega(t)$ contains more particles than $\zeta(t)$. Since jumps occur a.s.~one at a time, the only way for the maximum to increase in time is for an $ \omega$--particle to jump \lj{into} $[y, y']$ without a $\zeta$--particle \lj{or for a $\zeta$--particle to jump out of $[y, y']$ without an $\omega$--particle. Regarding the former case,} since particle jumps are nearest-neighbor, the only way this can happen is if $\Delta_{y,y'}$ is less than either $\Delta_{y-1,y'}$ or $\Delta_{y,y'+1}$, \lj{which is impossible. The latter case is similar and we leave it to the readers.}

Now, thanks to \eqref{eq:initprofileZRvanish}, we claim that $
\max_{y, y'\in \T_M}\Delta_{y,y'}(0)/M$ vanishes in probability \lj{as $M \rightarrow \infty$}. \lj{To prove it, fix $\varepsilon>0$ and  define $\varepsilon'=\varepsilon/4 C_0$, (where $C_0$ was given in \eqref{eq:defC0}).	Let $A_{\varepsilon'} = \{k\varepsilon': k =0,1,\ldots,(\varepsilon')^{-1}\}$.  Since
\begin{align*}
\max_{y, y'\in\T_M} \frac1M\Delta_{y,y'}(0)&=\max_{ v, v'\in\T} \frac1M \Delta_{vM,v^\prime M}(0)\\
&\leq \max_{ v, v'\in A_{\varepsilon'} } \frac1M \Delta_{vM,v^\prime M}(0)  + 2 \max_{k=0,\dots,(\varepsilon')^{-1}} \frac{1}{M} \sum_{y=kM \varepsilon^\prime}^{(k+1)M \varepsilon^\prime} [\omega_y+\zeta_y],
\end{align*}
we have
\begin{align*}
&\widetilde{\mu}_M  \bigg( \max_{y, y'\in\T_M} \frac1M\Delta_{y,y'}(0)> \varepsilon \bigg) \notag \\
&\leq \widetilde{\mu}_M  \bigg( \max_{ v, v'\in A_{\varepsilon'} } \frac1M \Delta_{vM,v^\prime M}(0) > \frac \varepsilon 2\bigg) + \widetilde{\mu}_M  \bigg( \max_{k=0,\dots,(\varepsilon')^{-1}} \frac{1}{M} \sum_{y=kM \varepsilon^\prime}^{(k+1)M \varepsilon^\prime} [\omega_y+\zeta_y] > \frac \varepsilon 4\bigg)\notag \\
&\leq \sum_{ v, v'\in A_{\varepsilon'} } \widetilde{\mu}_M  \bigg( \frac1M \Delta_{vM,v^\prime M}(0)> \frac \varepsilon 2\bigg) +\sum_{k=0}^{(\varepsilon')^{-1}}  \widetilde{\mu}_M  \bigg(  \frac{1}{M\varepsilon^\prime} \sum_{y=kM \varepsilon^\prime}^{(k+1)M \varepsilon^\prime} [\omega_y+\zeta_y] > C_0\bigg).
\end{align*}
By \eqref{eq:initprofileZRvanish}, approximating ${\mathbbm{1}}_{[v, v']}$ by smooth test functions, the first term in the right hand side vanishes as $M \rightarrow \infty$. The second term also vanishes by a straightforward adaptation of \eqref{eq:initprofileZRvanish2}.}
This, together with  \eqref{eq:Deltat0},  yields that $\max_{y, y'\in \T_M}\Delta_{y,y'}(t)/M$ also vanishes in probability for any $t>0$. Fix a smooth test function $\varphi\in C^2(\T)$, which we approximate, for any $n\in \N$, by $\varphi_n:=\sum_{k=1}^n\varphi(k/n)\mathbbm{1}_{((k-1)/n, k/n]}$ which is in  supremum norm $ \mathcal{O}(1/n)$ close to $\varphi$. We can now write 
\begin{multline}
\label{finalinitial}
\bigg|\frac{1}{M}\sum_{y\in \T_M}\zeta_y(t)\varphi (\tfrac yM)-\frac{1}{M}\sum_{y\in \T_M}\omega_y(t)\varphi (\tfrac yM)\bigg|\\
\leq \bigg|\frac{1}{M}\sum_{y\in \T_M}(\zeta_y-\omega_y)(t)\varphi_n (\tfrac yM)\bigg|+\frac{1}{M}\sum_{y\in \T_M}(\zeta_y+\omega_y)(t)\big|\varphi-\varphi_n\big| (\tfrac yM).
\end{multline}
For any $n$, the first term in the right hand side vanishes in probability  as $M\to\infty$ because $\varphi_n$ is bounded and piecewise constant, and $ \max_{y, y'\in \T_M}\Delta_{y,y'}(t)/M$ vanishes in probability. For any fixed $\varepsilon>0$, and for any $n$ large enough ($n>C_0\norm{\varphi'}_\infty/\varepsilon$), we can write according to \eqref{eq:initprofileZRvanish2}, and by conservation of the number of particles that 
\[{\bf P}_{\widetilde{\mu}_M}\bigg(\bigg|\frac{1}{M}\sum_{y\in \T_M}(\zeta_y+\omega_y)(t)\big|\varphi-\varphi_n\big| (\tfrac yM)\bigg)>\varepsilon\bigg)\]
vanishes  as $M\to\infty$, where ${\bf P}_{\widetilde{\mu}_M}$ denotes the distribution of the coupled zero-range process started from $\widetilde{\mu}_M$. In particular, since the left-hand side of \eqref{finalinitial} does not depend on $n$, it vanishes in probability as $M\to\infty$. This proves that if $\omega$ admits a hydrodynamic limit, $\zeta$ also does.

\medskip

We now turn to the asymmetric case on the full line. We claim that analogously to the symmetric case, regardless of $\mu_M$, $\mu'_M$,  for any $v < v^\prime$,
\begin{equation}
\label{eq:vM}
\sup_{vM < y < y^\prime < v^\prime M} \bigg| \sum_{z=y}^{y^\prime} \big(\omega_z (t) - \zeta_z (t)\big) \bigg| \leq \sup_{(v-2t)M < y < y^\prime < (v^\prime +2t) M} \bigg| \sum_{z=y}^{y^\prime} \big(\omega_z (0) - \zeta_z (0)\big) \bigg| + \varepsilon_M,
\end{equation}
where the remainder term $\varepsilon_M$ vanishes in probability as $M \rightarrow \infty$. To prove this claim,  let  $\{\xi(t)\}$ be a third process, coupled with the other two, but with initial configuration given  by
\[\xi_y(0)=\omega_y(0)\mathbbm{1}_{y\in[(v-2t)M,(v^\prime+2t)M]}+\zeta_y(0)\mathbbm{1}_{y\notin[(v-2t)M,(v^\prime+2t)M]}.\]
Then, the left hand side of \eqref{eq:vM} is less than 
\[
\sup_{vM < y < y^\prime < v^\prime M} \bigg| \sum_{z=y}^{y^\prime} \big(\omega_z (t) - \xi_z (t)\big) \bigg|+ \sup_{vM < y < y^\prime < v^\prime M} \bigg| \sum_{z=y}^{y^\prime} \big(\xi_z (t) - \zeta_z (t)\big) \bigg|.
\]
For the first term on the right-hand side of the above, since the displacement of a particle  is stochastically bounded by a Poisson process with parameter one, with very high probability,  $\omega (t) = \xi(t)$ in the interval $[(v-t)M,(v^\prime+t) M]$, \emph{cf.}~\cite{rezakhanlou91} for example.  In particular the first term vanishes in probability as $M \rightarrow \infty$. For the second term, note as in the symmetric case that thanks to the coupling, the quantity
\[\sup_{y < y^\prime} \bigg| \sum_{z=y}^{y^\prime} \big(\xi_z (t) - \zeta_z (t)\big) \bigg|\]
is non-increasing in time, and is a.s.~finite because $\xi$ and $\zeta$ only differ initially in the finite segment $[(v-2t)M,(v^\prime+2t)M]$.  This permits us to bound the second term by
\[\sup_{y < y^\prime} \bigg| \sum_{z=y}^{y^\prime} \big(\xi_z (0) - \zeta_z (0)\big) \bigg| =  \sup_{(v-2t)M < y < y^\prime < (v^\prime +2t) M} \bigg| \sum_{z=y}^{y^\prime} \big(\omega_z (0) - \zeta_z (0)\big) \bigg| .\]
This proves \eqref{eq:vM}. Repeating then the same arguments as in the symmetric case, it is straightforward to show thanks to \eqref{eq:initprofileZRvanish} that the right-hand side vanishes in probability as $M\to\infty$ for any fixed $v$, $v'$ and $t$. It then follows that for any compactly supported test function $\varphi$, the discrete integrals $\frac{1}{M}\sum_{y\in \Z}\omega_y(t)\varphi (\tfrac yM)$ and $\frac{1}{M}\sum_{y\in \Z}\zeta_y(t)\varphi (\tfrac yM)$ converge in probability to the same limit, which concludes the proof of Lemma \ref{lem:initprofileZR}.
\end{proof}

Thanks to the previous lemma, and throughout the rest of the section, we can now, without loss of generality, assume that the zero-range process is started from a product initial distribution. More precisely, we now assume that the initial distriution for the process is given by \eqref{eq:muM}, namely its marginals satisfy
\[
\mu_M(\omega_y=k)=\mathbbm{1}_{\{k\in \N\}}\frac{1}{1+\alpha^{\mathrm{ini}}(\frac{y}{M})}\pa{1-\frac{1}{1+\alpha^{\mathrm{ini}}(\frac{y}{M})}}^k, \quad \text{ for any } y \in \bL_M.
\]
\lj{Recall $\N = \{0,1,2,\ldots\}$.} Denote by  $\mu_\alpha$ the corresponding (fixed parameter) geometric product measure  with density $\alpha$, taking values in $\N$ at each site,
\begin{equation}
\label{eq:nualpha}
\mu_\alpha(\omega_y=k)={\bf 1}_{\{k\in \N\}}\frac{1}{1+\alpha}\pa{1-\frac{1}{1+\alpha}}^k, \quad \text{ for any } y \in \bL_M.
\end{equation}
As previously noted in Section  \ref{sec:ZR}, $\mu_\alpha$ is \emph{not} an equilibrium distribution for the zero-range process, since its equilibrium distributions are  given by geometric product measures $\mu_\alpha^\star$ taking values at each site in $\N^*:=\N\setminus\{0\}$, and only defined for densities $\alpha\geq1$,
\begin{equation}
\label{eq:nualphastar2}
\mu^\star_\alpha(\omega_0=k)=\mu_{\alpha-1}(\omega_y=k-1)={\bf 1}_{\{k\in \N^*\}}\frac{1}{\alpha}\pa{1-\frac{1}{\alpha}}^{k-1}, \quad \text{ for any } y \in \bL_M.
\end{equation}
For this reason, the process is not initially
in a state of local equilibrium, and cannot be locally coupled at time $0$ with equilibrium distributions. Further note that even with Lemma \ref{lem:initprofileZR}, we cannot start the process from a local equilibrium state, since the latter is only defined (see \eqref{eq:nualphastar2}) for densities $\alpha\geq 1$, and our initial profile can \textit{a priori} take any non-negative values.

However,  both distributions $\mu_\alpha$ and $\mu^\star_\alpha$ are parametrized by the particle density, $\E_{\mu^\star_\alpha}[\omega_y(0)]=\E_{\mu_\alpha}[\omega_y (0)]=\alpha$.
Note that because it is not the equilibrium distribution for the facilitated zero-range process, the measure $\mu_\alpha$ given by \eqref{eq:nualpha} is not a particularly natural choice of initial product measure. However, taking a geometric distribution is a natural choice through the mapping, because it translates in the exclusion process as a product Bernoulli measure. This is the reason behind the initial distributions \eqref{eq:muM} and \eqref{eq:nualpha}.

\medskip

We now give a technical lemma, which allows us to bound from above the zero-range process by an equilibrium state with slightly larger density. Recall that the initial profile $\alpha^{\rm ini}$ for the zero-range process is assumed to be bounded,  we denote by 
\begin{equation}
\overline{\alpha}=\sup_{v\in\bL}\alpha^{\rm ini}(v)+1.
\end{equation}

\begin{lemma}
\label{lem:upperbound}
There exists a coupling $\overline{\mu}_M$ between the initial configuration $\omega$ whose  distribution   is 
\[\mu_M(\omega=\cdot)=\bigotimes_{y\in \bL_M}\mu_{\alpha^{\rm ini}(y/M)}(\omega_y=\cdot),\] 
given by \eqref{eq:muM} and the equilibrium configuration $\zeta$ with  distribution $\mu^\star_{\overline{\alpha}}$, such that 
\begin{equation}
\label{eq:upperbound0}
\overline{\mu}_M(\omega\leq \zeta)=1.
\end{equation}
In particular, the zero-range process started from $\mu_M$ is at all times stochastically dominated by $\mu^\star_{\overline{\alpha}}$, in the sense that 
\begin{equation}
\label{eq:upperbound}
{\bf P}_{\overline{\mu}_M}(\omega(t)\leq \zeta(t))=1,
\end{equation}
where ${\bf P}_{\overline{\mu}_M}$ denotes the coupled distribution of the processes $\{\omega(t), \zeta(t)\}$ started from $\overline{\mu}_M$ and driven by the \lj{basic coupling}.
\end{lemma}

\begin{proof}[Proof of Lemma \ref{lem:upperbound}]
Thanks to the coupling, if \eqref{eq:upperbound} holds at the initial time, it holds for any time $t>0$. We therefore only need to build the initial coupling $\overline{\mu}_M.$ Denote by $\overline{\mu}_M$ the distribution of $M$ i.i.d.~variables $(U_y)_{y\in \bL_M}$, \lj{which are uniform on $[0,1]$,} and define for any $k\in \N$ 
\begin{align*}\omega_y=k\quad &\mbox{ iff } \quad  \mu_M(\omega_y< k) \leqslant U_y < \mu_M(\omega_y< k+1),\\ \zeta_y=k\; \quad &\mbox{ iff } \quad \mu_{\overline{\alpha}+1}(\omega_y< k) \leqslant U_y < \mu_{\overline{\alpha}+1}(\omega_y< k+1).\end{align*}
Since $\alpha^{\rm ini}+1\leq \overline{\alpha}$, we have for any $k$ and any $y\in \bL_M$
\[
\pa{\frac{\alpha^{\rm ini}(y/M)}{1+\alpha^{\rm ini}(y/M)}}^k=\mu_M(\omega_y \geq 
k)\leq \mu_{\overline{\alpha}}(\omega_y\geq k)=\pa{\frac{\overline{\alpha}-1}{\overline{\alpha}}}^{k-1},
\]
so that \eqref{eq:upperbound0} holds.
\end{proof}

\subsection{Symmetric case} 
We now prove Theorem \ref{thm:zr}. Even though in the symmetric case, one can prove that the transience time to reach a frozen/ergodic state is subdiffusive (see \emph{e.g}.~\cite[Theorem 2.6]{blondel2021stefan}), the two-phased nature of our zero-range process, unfortunately, precludes the use of any entropy tool. Indeed, both the entropy method \cite{guo1988nonlinear} and the relative entropy method \cite{yau1991relative} require the existence of a one-parameter family of stationary states, which for our model only exists in the supercritical case $\alpha \geq 1$ (see \eqref{eq:nualphastar2}). In the subcritical regions, any $\delta$--Dirac measure on a  frozen state is by definition stationary, and entropy arguments no longer apply.

However, thanks to the attractiveness of the process, we are able to provide a simple proof based
\begin{itemize}
\item on Funaki's one-block estimate (see \cite[Theorem 4.1]{funaki1999free}, \cite[Proposition 3.8]{blondel2021stefan}) in the context of Stefan problems, based on the decomposition of translation invariant stationary states for the infinite volume dynamics. This is the content of Lemma \ref{lem:OBE} below;
\item on Rezakhanlou's two-blocks estimate (see \cite[Lemma 6.2]{rezakhanlou91}), which he proved for attractive asymmetric particle systems, adapted in Lemma \ref{lem:TBE} below.
\end{itemize}

Unfortunately, Rezakhanlou's argument does not work \textit{verbatim} in our case, because as mentioned in the previous subsection, the process is not started from a state of local equilibrium, and cannot be locally coupled with equilibrium distributions.

\medskip

A strategy to overcome this issue is the following: instead of coupling with equilibrium distributions, we couple with pseudo-equilibrium distributions, that is with distributions ${\bf P}_{\mu_\alpha}$ of the zero-range process started from a constant density $\alpha>1$. According to \cite{blondel2020hydrodynamic}, for this process with $\alpha>1$, the hydrodynamic limit holds, and in particular one can show the one and two-blocks estimates by the classical entropy estimates, starting from the time where the ergodic component is reached.

\subsubsection{A one-block estimate}
We start by a local law of large numbers for the zero-range process, given by the following lemma.
\begin{lemma}
\label{lem:OBE}
Define $B_\ell(y)=\{y-\ell, \dots,y+\ell\}$, and \[\omega^\ell_y(t)=\frac{1}{2\ell+1}\sum_{y'\in B_\ell(y)} \omega_{y'}(t),\] then we have
\begin{equation}
\label{eq:OBE}
\limsup_{\ell\to\infty}\limsup_{M\to\infty}\mathbf{E}_{\mu_M}\bigg[\int_0^T\frac{1}{M}\sum_{y\in \T_M}\bigg| \frac{1}{2\ell+1}\sum_{y'\in B_\ell(y)} g(\omega_{y'}(t))-\mathcal{G}(\omega^\ell_y(t))\bigg|dt\bigg]=0.
\end{equation}
\end{lemma}
\begin{proof}[Proof of Lemma \ref{lem:OBE}]
The proof is analogous to Funaki's \cite[Theorem 4.1]{funaki1999free}, we simply sketch it. Denote by $\tau_x$ the translation by $x$ of a configuration, $(\tau_x \omega)_y=\omega_{x+y}$, and consider the space-time average of the process' distribution on $[0,T]\times \T_M$, 
\[\overline{\mu}^T_M(\cdot)=\frac{1}{TM}\int_0^T\sum_{x\in \T_M}{\color{red}\mathbf{P}_{\mu_M}}(\tau_x \omega(t)=\cdot)dt.\]
\lj{Since the sequence $(\mu_M)_{M \geq 1}$ is tight, the sequence $(\overline{\mu}^T_M)_{M\geq 1}$ is also tight}, and any of its limit point $\mm{\overline{\mu}^T}$ is translation-invariant and stationary \ccl{(because the generator is sped up by $M^2$, and therefore relaxes to stationarity very fast on a local scale, see \cite[Theorem 4.1]{funaki1999free} for the complete proof)} for the infinite volume zero-range generator $\genzr$ obtained from \eqref{eq:DefLM} with $p=\frac12$ and $ \bL_M = \Z$.

\medskip

In particular,  by Lemma \ref{lem:ergodic}, we have an explicit decomposition for the zero-range translation-invariant stationary measures, so that there must exist $\lambda\in [0,1]$ and a probability  measure $\beta$ on  $[1,+\infty)$ such that 
\begin{equation}
\label{eq:decompmubar}
\overline{\mu}(\cdot)=\lambda \overline{\mu}_{\mathcal{F}}(\cdot)+(1-\lambda)\int_{[1,+\infty)} \beta (d\alpha)\mu^\star_\alpha(\cdot).
\end{equation}
In this identity, $\overline{\mu}_{\mathcal{F}}$ is a measure supported on the frozen states $\mathcal{F}=\{0,1\}^\Z$. For any $\alpha\geq 1$, under the equilibrium measure,
\[\limsup_{\ell\to\infty}{\bf E}_{\mu^\star_\alpha}\bigg[\bigg| \frac{1}{2\ell+1}\sum_{y=-\ell}^\ell g(\omega_y)-\mathcal{G}\pa{\omega_y^\ell}\bigg|\bigg]=0\]
by the strong law of large numbers, whereas for any $\omega \in \mathcal{F}$ both terms inside the absolute value vanish. This, together with \eqref{eq:decompmubar}, yields
\[\limsup_{\ell\to\infty}{\bf E}_{\overline{\mu}}\bigg[\bigg| \frac{1}{2\ell+1}\sum_{y=-\ell}^\ell g(\omega_y)-\mathcal{G}\pa{\omega_y^\ell}\bigg|\bigg]=0,\]
which proves  Lemma \ref{lem:OBE}.
\end{proof}

\subsubsection{Two-blocks estimate}
We now prove a two-blocks estimate in the supercritical region, since the contribution of the subcritical one to the hydrodynamic limit vanishes.

\begin{lemma}
\label{lem:TBE}
Using the same notations as before, for any $T\geq 0$ and any positive $\delta$,
\begin{multline}
\label{eq:TBE}
\limsup_{\ell\to\infty}\limsup_{\varepsilon\to 0}\limsup_{M\to\infty}\int_0^T\frac{1}{M}\sum_{y\in \T_M}\mathbf{P}_{\mu_M}\Big(\big|\omega^\ell_y(t)-\omega^{\varepsilon M}_y(t)\big|>\delta, \\
\; \omega^\ell_y(t), \;\omega^{\varepsilon M}_y(t)>1+\delta\Big)dt=0,
\end{multline}
and for any $T\geq 0$
\begin{equation}
\label{eq:TBEG}
\limsup_{\ell\to\infty}\limsup_{\varepsilon\to 0}\limsup_{M\to\infty}\mathbf{E}_{\mu_M}\bigg[\int_0^T\frac{1}{M}\sum_{y\in \T_M}\Big|\mathcal{G}\big(\omega^\ell_y(t)\big)-\mathcal{G}\big(\omega^{\varepsilon M}_y(t)\big)\Big|dt\bigg]=0.
\end{equation}
\end{lemma}

To prove the two-blocks estimate, we adapt  Rezankhalou's coupling argument to our pseudo-equilibrium measures, together with the following result, that states that the two-blocks estimate holds starting from a uniform supercritical density.

\begin{lemma}
\label{lem:TBE2}
For any $\alpha> 1$, $ T>0$,
\begin{equation}
\label{eq:TBE2}
\limsup_{\ell\to\infty}\limsup_{\varepsilon\to 0}\limsup_{M\to\infty}\mathbf{E}_{\mu_\alpha}\bigg[\int_0^T\big|\omega^\ell_0(t)-\omega^{\varepsilon M}_0(t)\big|dt\bigg]=0.
\end{equation}
Furthermore, 
\begin{equation}
\label{eq:TBE3}
\limsup_{\varepsilon\to 0}\limsup_{M\to\infty}\mathbf{E}_{\mu_\alpha}\bigg[\int_0^T\big|\omega^{\varepsilon M}_0(t)-\alpha\big|dt\bigg]=0.
\end{equation}
Equation \eqref{eq:TBE3} also holds for the microscopic local density, \emph{i.e.}~with $\varepsilon M$ replaced by $\ell$ that goes to infinity after $M$.
\end{lemma}
Note that if $\mu_\alpha$ is replaced by the equilibrium measure $\mu_\alpha^\star$, this result would be a direct consequence of the strong law of large numbers.

\begin{proof}[Proof of Lemma \ref{lem:TBE2}]
The strategy to prove this result is the following: 
\begin{itemize}
\item First, \mm{we use a result from \cite[Proposition 4.1]{blondel2020hydrodynamic} which estimates the time scale over which the configuration becomes ergodic with high probability: more precisely, choosing $\beta>0$ large enough, and denoting $t_M=(\log M)^{4\beta}/M^2$, we have, for any fixed $\alpha>1$, 
\[\liminf_{M\to\infty}\mathbf{P}_{\mu_\alpha}\big(\omega_y\pa{t_M}\geq 1, \;\forall y\in\T_M\big)=1.\]
This result indeed holds since} for any $\delta>0$, the probability under $\mu_\alpha$ for the configuration to be $\delta$-regular (in the sense of \cite[equation (4.6)]{blondel2020hydrodynamic} is of order $1-o_N(1)$.
\item We then define for any positive time $t$ the measure 
\[\widetilde{\mu}_t(\cdot)=\mathbf{P}_{\mu_\alpha}\Big(\omega(t_M+t)=\cdot \;\big|\; \omega_y\pa{t_M}\geq 1, \;\forall y\in\T_M\Big),\]
of the process at time $t_M+t$ conditioned to have reached the ergodic component before time $t_M$, and define the density $f_t=d\widetilde{\mu}_t/ d\mu_\alpha^\star$. We can then define  the relative entropy 
$H(\widetilde{\mu}_t\mid \mu_\alpha^\star):=\int  f_t \log(f_t) d\mu_\alpha^\star.$
\smallskip

\item Starting from the distribution $\widetilde{\mu}_0$, the process is ergodic, and assuming that $H(\widetilde{\mu}_0\mid \mu_\alpha^\star)\leq CM$, Lemma \ref{lem:TBE2} follows from the standard two-blocks estimate for the zero-range process (see \emph{e.g.}~\cite[Lemma 3.2, p.83]{klscaling}). To estimate the initial entropy, we write by direct calculations
\[H(\widetilde{\mu}_0\mid \mu_\alpha^\star)\leq - \int  \log(\mu_\alpha^\star(\omega)) d\widetilde{\mu}_0(\omega)=\ccl{M\log \alpha}+ \log\pa{\frac{\alpha-1}{\alpha}}{\bf E}_{\widetilde{\mu}_0}\bigg[\sum_{y\in \T_M}\omega_y \bigg].\]
Since the number of particles is conserved by the zero-range dynamics, recalling that $\overline{\alpha} $ is an upper bound on $\alpha^{\rm ini}+1$, the second integral above is less than $M \overline{\alpha} /\mathbf{P}_{\mu_\alpha}(\omega_x\pa{t_M}\geq 1, \;\forall x\in\T_M),$
so that for $M$ large enough, 
\begin{equation}
\label{eq:H0estimate}
H(\widetilde{\mu}_0\mid \mu_\alpha^\star)\leq \overline{\alpha} C(\alpha) M.
\end{equation}
\end{itemize}
We do not detail the proof of the hydrodynamic limit for the ergodic zero-range process, since it is detailed in \cite[Chapter 5]{klscaling}, under condition \eqref{eq:H0estimate}. The only hurdle is that assumption (FEM) in \cite[Chapter 5]{klscaling} does not hold, however it is only used to cut off large densities, which can be done in our case using attractiveness (\emph{cf.}~\eqref{eq:upperbound}). This proves in particular the two-blocks estimate \eqref{eq:TBE2}.

\medskip

Since for fixed $\varepsilon>0$, $\omega_0^{\varepsilon M}(t)$ can be directly expressed (up to a small error term) as a function of the empirical measure of the process, the second identity \eqref{eq:TBE3} is a consequence of the hydrodynamic limit in the case of a constant supercritical initial density, together with the dominated convergence Theorem as $\varepsilon \to 0$.

The last statement of the lemma readily follows from the first two.
\end{proof}

We now prove the two-blocks estimate.

 \begin{proof}[Proof of Lemma \ref{lem:TBE}]
Thanks to the pseudo-equilibrium two-blocks estimate (Lemma \ref{lem:TBE2}), the proof of  \eqref{eq:TBE} is straightforwardly adapted from Rezankhalou's proof of \cite[Lemma 6.2]{rezakhanlou91}. Recall that $\overline{\alpha}$ is a uniform bound on $\alpha^{\rm ini}+1$. We first assume that $\alpha^{\rm ini}$ has only finite cross values, in the sense that the function $\alpha^{\rm ini}(\cdot)-c$ strictly  changes sign at most a finite number of times (depending on $c$) in $\T$. For any $c\geq 0$, we denote by $\mathcal{N}_c$ the number of times $\alpha^{\rm ini}-c$ strictly changes sign, and  denote by $\zeta^c$ a zero-range process started from the pseudo-equilibrium uniform profile with density $c$,  
 \[\zeta^c (0)\sim \mu_c,\]
 where $\mu_c$ is the non-stationary initial distribution given by \eqref{eq:nualpha}.
For any $c\geq 0$, we couple $\zeta^c$ with $\omega$  by the basic coupling  described in Section \ref{sec:coupling}. Given two zero-range configurations $\omega$, $\omega'$, denote by $\mathcal{N}(\omega, \omega')$ the number of times $\omega-\omega'$ strictly changes sign. Since we start both $\omega$ and the $\zeta^c$'s from product measures with marginals given by \eqref{eq:nualpha},  the initial configurations can be chosen in such a way that for any $c\geq 0$, and any $y\in \T_M$
\[\omega_y(0)-\zeta^c_y(0) \quad \mbox{ has same sign as } \quad  \alpha^{\rm ini}(y/N) -c,\]
so that $\mathcal{N}(\omega(0), \zeta^c(0))\leq \mathcal{N}_c.$ Furthermore, since $\omega$ and the $\zeta^c$'s evolution are coupled, the function $t\mapsto \mathcal{N}(\omega(t), \zeta^c(t))$ is non-increasing for any $c$, so that in particular, for any $c, \;t\geq 0$, 
\begin{equation}
\label{eq:Ntc}
\mathcal{N}(\omega(t), \zeta^c(t))\leq \mathcal{N}_c.
\end{equation}
We do not prove this last statement, it is \cite[Lemma 6.5]{rezakhanlou91}.

We now prove the two-blocks estimate in the case where $\alpha^{\rm ini}$ has only finite cross-values. Define for any $c, \;t\geq 0,$ and any integer $\ell\geq 0,$ the set 
\[\Gamma^c_{\ell}(t):=\left\{y\in \T_M, \; \omega(t)\mbox{ and $\zeta^{c}(t)$ are not ordered in }B_{\ell}(y)\right\}\subset \T_M,\]
where as before $B_\ell(y)=\{y-\ell,\dots, y+\ell\}$ is the box of size $\ell$ around $y$.
Thanks to \eqref{eq:Ntc}, we have for any $M \in \N$, and any $\varepsilon >0$,
\begin{equation}
\label{eq:Gammacbound}
\left|\Gamma^{c}_{\varepsilon M}(t)\right|\leq (2 \varepsilon M +3)\mathcal{N}_{c}.
\end{equation}
Further note that for any $y\notin \Gamma^{c}_{\varepsilon M}(t)$, and any $\ell \leq \varepsilon M$, $\omega(t)$ and $\zeta^{c}(t)$ are also ordered in $B_{\ell}(y)$.

Fix an integer $n>0$. For $1\leq k\leq n+1$, shorten  $c_k=1+(k+1)\overline{\alpha}/n$, and define 
\[\bar\Gamma_M(t)=\T_M\setminus \bigcup_{k=1}^{n+1} \Gamma^{c_k}_{\varepsilon M}(t),\]
which is the set of points around which  $ \omega(t)$ and $\zeta^{c_k}(t)$ are ordered for each $k$.
Note that for any $y\notin \Gamma^c_{\varepsilon M}(t)$, assuming that 
\[\quad \big|\zeta^{c,\ell}_y(t)-c\big|\leq \delta \quad  \mbox{ and }\quad \big|\zeta^{c,\varepsilon M}_y(t)-c\big|\leq \delta,\]
we must have 
\[\omega^\ell_y(t), \; \omega^{\varepsilon M}_y(t)\geq c-\delta \quad \mbox{ or }\quad \omega^\ell_y(t) , \; \omega^{\varepsilon M}_y(t)\leq c+\delta.\]
In particular, choosing $\delta=\overline{\alpha}/n$, for any $y\in \bar\Gamma_M(t)$, we must have:
\begin{enumerate}
\item either there exists $1\leq k\leq n+1$ such that 
\[\quad \big|\zeta^{c_k,\ell}_y(t)-c_k\big|> \delta \quad  \mbox{ or }\quad \big|\zeta^{c_k,\varepsilon M}_y(t)-c_k\big|> \delta,\]
\item or $\omega^\ell_y(t) , \; \omega^{\varepsilon M}_y(t)\geq c_{n+1}-\delta=1+\overline{\alpha}+\delta$
\medskip

\item or $\omega^\ell_y(t) , \; \omega^{\varepsilon M}_y(t)\leq c_{1}+\delta=1+3\delta$
\medskip

\item or for any $1\leq k\leq n+1$
\[\omega^\ell_y(t), \; \omega^{\varepsilon M}_y(t)\geq c_k-\delta \quad \mbox{ or }\quad \omega^\ell_y(t) , \; \omega^{\varepsilon M}_y(t)\leq c_k+\delta,\]
and $\omega^\ell_y(t), \; \omega^{\varepsilon M}_y(t) \in ( c_1+\delta, c_{n+1}-\delta)$. Since the $c_k$'s are distant of $\delta$, this last case implies in particular that $|\omega^\ell_y(t)- \; \omega^{\varepsilon M}_y(t)|\leq 3\delta$.
\end{enumerate}
Further note that according to \eqref{eq:Gammacbound}, for any $t\geq 0$,  
\[\mathrm{card}\big(\T_M\setminus\bar\Gamma_M(t)\big) \leq (2 \varepsilon M +3)\sum_{k=1}^{n+1} \mathcal{N}_{c_k}.\]
As before, we denote by $\mathbf{P}_{\overline{\mu}_M}$  and $\mathbf{E}_{\overline{\mu}_M}$ the distribution of the coupled processes started from $\mu_M$, $(\mu_c)_{c\geq 0}$ and the corresponding expectation.
We can now write 
\begin{align*}
\mathbf{E}_{\mu_M}\bigg[\int_0^T&\frac{1}{M}\sum_{y\in \T_M}\mathbbm{1}{\{|\omega^\ell_y(t)-\omega^{\varepsilon M}_y(t)|>3\delta, \; \omega^\ell_y(t), \;\omega^{\varepsilon M}_y(t)>1+3\delta\}}dt\bigg]\\
\leq & \; T\pa{2\varepsilon +\frac{3}{M}}\sum_{k=0}^{n+1} \mathcal{N}_{c_k}\\&+\mathbf{E}_{\overline{\mu}_M}\bigg[\int_0^T\frac{1}{M}\sum_{y\in \bar\Gamma_M(t)}\mathbbm{1}{\{|\omega^\ell_y(t)-\omega^{\varepsilon M}_y(t)|>3\delta, \; \omega^\ell_y(t), \;\omega^{\varepsilon M}_y(t)>1+3\delta\}}dt\bigg].
\end{align*}
Recall that $n$ is fixed, the first term on the right-hand side vanishes as $M\to\infty$ then $\varepsilon \to 0$. Because of the indicator function, and because the $y$'s are in $\overline{\Gamma}_M(t)$, \lj{only  cases 1 and 2 remain}. By union bound, the expectation in the right-hand side is therefore bounded from above by
\begin{multline*}
\int_0^T\frac{1}{M}\sum_{y\in \bar\Gamma_M}\sum_{k=0}^{n+1}\mathbf{P}_{\mu_{c_k}}\pa{\,|\zeta^{\ell}_y(t)-c_k|> \delta}+\mathbf{P}_{\mu_{c_k}}\pa{\,|\zeta^{\varepsilon M}_y(t)-c_k|> \delta}dt\\
+\int_0^T\frac{1}{M}\sum_{y\in \bar\Gamma_M}\mathbf{P}_{\mu_M}\pa{\omega^\ell_y(t) > 1+\overline{\alpha}+\delta}+\mathbf{P}_{\mu_M}\pa{\omega^{\varepsilon M}_y(t) >  1+\overline{\alpha}+\delta}dt.
\end{multline*}
Because the $c_k$'s are all strictly larger than $1$, according to Lemma \ref{lem:TBE2}, the first line vanishes in the limit $M\to\infty, \; \varepsilon\to 0, \; \ell\to\infty$. The second line also vanishes according to Lemma \ref{lem:upperbound} and the law of large number, because $\omega(t)$ is stochastically bounded by an equilibrium process with density $\overline{\alpha}$.
\lj{Since $n$ can be chosen arbitrarily large}, $\delta$ can be chosen arbitrarily small, which proves \eqref{eq:TBE}.

\medskip

We now prove \eqref{eq:TBEG}. Note that $\norm{\mathcal{G}'}_{\infty}=\norm{\mathcal{G}}_{\infty}=1$.
We distinguish five cases for the quantity 
\[Q_y(t):=\big|\mathcal{G}(\omega^\ell_y(t))-\mathcal{G}(\omega^{\varepsilon M}_y(t))\big|\] 
inside the absolute values.
\begin{enumerate}
\item If $y\in \Gamma^{c_1}_{\varepsilon M}(t)$, we have \emph{a priori} only a crude bound over $|Q_y(t)|\leq 2$, however the number of such $y$'s is less than $(2\varepsilon M+3)\mathcal{N}_{c_1}$, so that the contribution of the  $y\in\Gamma^{c_1}_{\varepsilon M}(t)$ vanishes in the limit $M\to\infty$ then $\varepsilon \to 0$.
\smallskip
\item If $|\zeta^{c_1,\ell}_y(t)-c_1|> \delta$  or $|\zeta^{c_1,\varepsilon M}_y(t)-c_1|> \delta$, we also only have $Q_y(t)\leq 2$, however the probability that this occurs is small, so that according to \eqref{eq:TBE3} this contribution also vanishes as $M\to\infty$ then $\varepsilon \to 0$ then $\ell\to\infty$. \smallskip
\item If $|\omega^\ell_y(t)-\omega^{\varepsilon M}_y(t)|\leq \delta$, then $Q_y(t)\leq \delta$. 
\smallskip
\item If $|\omega^\ell_y(t)-\omega^{\varepsilon M}_y(t)|> \delta$, and both $\omega^\ell_y(t)$, $\omega^{\varepsilon M}_y(t)$ are larger than $1+\delta$, the contribution vanishes in the triple limit according to \eqref{eq:TBE}.
\smallskip
\item Because we eliminated cases 1 and 2,  \lj{the only case remaining is where} both $\omega^\ell_y(t)$ and  $\omega^{\varepsilon M}_y(t)$ are smaller than $1+3\delta$, in which case both $\mathcal{G}(\omega^\ell_y(t))$ and $\mathcal{G}(\omega^{\varepsilon M}_y(t))$  are less than $3 \delta$.
\end{enumerate}
All three contributions 1, 2 and 4 vanish in the limit, therefore the left-hand side in \eqref{eq:TBEG} is less than $4 T \delta$, and since $\delta$ is arbitrarily small we obtain \eqref{eq:TBEG}.
\end{proof}

\subsubsection{Conclusion in the symmetric case}
With the one and two-blocks estimates stated in Lemmas \ref{lem:OBE} and \ref{lem:TBE} respectively, the proof of the hydrodynamic limit is straightforward, we simply sketch it.
We start by writing Dynkin's formula and performing summations by parts, and for any smooth test function $\varphi:\T\to \R$, since $\genzr \omega_y=g(\omega_{y+1})+g(\omega_{y+1})
-2g(\omega_y)$, we obtain
\[\frac{1}{M}\sum_{y\in \T_M} \varphi (\tfrac y M)\omega_y(t)=\frac{1}{M}\sum_{y\in \T_M} \varphi (\tfrac y M)\omega_y(0)+\int_0^t\frac{1}{M}\sum_{y\in \T_M}g(\omega_y(s)) M^2\Delta_M\varphi(\tfrac y M)ds+\mathscr{M}^\varphi_{t,M},\]
where $\Delta_M \varphi(\frac y M)=\varphi(\frac{y+1}M)+\varphi(\frac{y-1}M)-2\varphi(\frac y M)$ is the discrete Laplace operator. In the identity above, $\mathscr{M}^\varphi_{t,M}$ is a martingale whose quadratic variation is of order $\mathcal{O}(\frac 1M)$  and vanishes as $M\to\infty$ (see e.g. \cite[Appendix 1, Lemma 5.1, p. 330]{klscaling}). Since $\varphi$ is a smooth function, in the identity above, $g(\omega_y(s)) $ can be replaced by its average over a small microscopic box, 
\[\frac{1}{2\ell+1}\sum_{y'\in B_\ell(y)} g(\omega_{y'}(s)),\]
which in turn, thanks to the one-block estimate (Lemma \ref{lem:OBE}) can be replaced in the limit by $\mathcal{G}(\omega^\ell_y(t))$. According to \eqref{eq:TBEG}, $\mathcal{G}(\omega^\ell_y(t))$ can in turn be replaced by $\mathcal{G}(\omega^{\varepsilon M}_y(t))$. Letting $M\to\infty$ and $\varepsilon\to 0$, and because the microscopic configuration is stochastically dominated by an equilibrium configuration distributed as $\mu^\star_{\overline{\alpha}}$, we obtain that any limit point $ {\bf Q}^*$ of the distribution  ${\bf Q}_M$ of the empirical measure 
\[\pi_t^M:=\frac{1}{M}\sum_{y\in \T_M} \omega_y(t) \delta_{y/M}\]
is concentrated on trajectories $\pi_t=\alpha_t(u)du$ which are at any time $t$ absolutely continuous w.r.t.~the Lebesgue measure, and which satisfy
\[\int_{\T}\varphi(v)\alpha_t(v)dv=\int_{\T}\varphi(v)\alpha_0(v)dv+\int_0^t\int_{\T}\mathcal{G}(\alpha_s(v))\partial_v^2 \varphi(v)ds.\]
It is straightforward to show that \eqref{eq:weakF1ZR} holds if $\varphi$ also depends on the time variable, which proves the hydrodynamic limit in the symmetric case, since the solution is unique.

\subsection{Asymmetric case} In this subsection, we prove Theorem \ref{thm:zr} in the asymmetric case.  By Lemma \ref{lem:initprofileZR}, we only need to prove the result for product initial measures with a slowly varying density profile $\alpha^{\rm ini}$. We first prove a microscopic version of the entropy inequality, \emph{cf.\;}Lemma \ref{lem:MicroEntroIneqn}. The main ingredient for the proof is the characterization of invariant and translation invariant measures for the original and coupled processes, which is already proved in Lemma \ref{lem:ergodic}. Then we prove that the initial condition holds for product measures with a slowly varying density profile $\alpha^{\rm ini}$. The main problem here is that the initial measure is not in local equilibrium and it is impossible to couple it with invariant measures of the process. Instead, we couple with the process starting from pseudo-equilibrium distributions as in the symmetric case. At last, we conclude the proof  following the steps presented in \cite{rezakhanlou91}. 

\subsubsection{Microscopic entropy inequality} We first prove the following microscopic entropy inequality.

\begin{lemma}[Microscopic entropy inequality]\label{lem:MicroEntroIneqn}
For every non-negative smooth function  $\varphi : \bb{R}_+ \times \bb{R} \rightarrow \bb{R}$ with compact support in $(0,\infty) \times \bb{R}$,  for every $c \geq 0$ and for every $\epsilon > 0$,
\begin{multline}\label{entIne}
	\lim_{\ell \rightarrow \infty}\,\liminf_{M \rightarrow \infty} \, \mathbf{P}_{\mu_M}  \bigg(\int_0^\infty \frac{1}{M} \sum_{y \in\bb{Z}} \Big\{\big|\omega_y^{\ell} (t)- c\big| \partial_t\varphi_t(\tfrac{y}{M}) \\
	+ (2p-1) \big|\mc{G}(\omega_y^{\ell}(t)) - \mc{G} (c)\big| \partial_v \varphi_t (\tfrac{y}{M})\Big\}\,ds\geq - \epsilon\bigg)=1.
\end{multline}
\end{lemma}

\begin{proof}
We adapt the ideas from \cite[Theorem 3.1]{rezakhanlou91}.  For this reason we only sketch the proof. We consider the coupled process $(\omega(t),\zeta(t))$ with generator $M \tgenzr$ and initial distribution $\tilde{\mu}_M:=\mu_M \otimes \mu_c^\star$. In the subcritical region $c\leq 1$ where the equilibrium distribution $\mu_c^\star$ is not defined, we choose arbitrarily $\mu_c^\star$ to be a Bernoulli product measure on $ \Z$, which is concentrated on frozen configurations, and satisfies
\[\lim_{\ell \rightarrow \infty} \omega_y^\ell = c\quad \text{in $\mu_c^\star$ -- probability.} \]

\noindent {\sc Step 1.} We first prove that 
\begin{multline}\label{entIne1}
		\liminf_{M \rightarrow \infty} \, \mathbf{P}_{\tilde{\mu}_M} \bigg(\int_0^\infty \frac{1}{M} \sum_{y \in\bb{Z}} \Big\{\big|\omega_y (t)- \zeta_y (t)\big| \partial_t \varphi _t(\tfrac{y}{M}) \\
		+ (2p-1) \big|\mathbbm{1}_{\{\omega_y (t) \geq 2\}} - \mathbbm{1}_{\{\zeta_y (t) \geq 2\}}\big|  \partial_v \varphi _t(\tfrac{y}{M})\Big\}\,dt\geq - \epsilon\bigg)=1.
	\end{multline}
By Dynkin's formula,\lj{
\begin{multline}
\label{mart1}
	\mc{M}_t^\varphi  := \frac{1}{M} \sum_{y \in\bb{Z}} \big|\omega_y (t)- \zeta_y (t)\big| \varphi_t(\tfrac{y}{M})\\ 
	- \int_0^t  \frac{1}{M} \sum_{y \in\bb{Z}} \Big\{M \tgenzr \big|\omega_y (s)- \zeta_y (s)\big| \varphi (s,\tfrac{y}{M})+  \big|\omega_y (s)- \zeta_y (s)\big| \partial_s \varphi (s,\tfrac{y}{M})\Big\}\,ds
\end{multline}
}is a martingale, and a simple calculation (\emph{cf.}~\cite[Lemma 3.2]{rezakhanlou91} for details) yields that for every $T > 0$,
\[\lim_{M \rightarrow \infty} \mathbf{E}_{\tilde{\mu}_M} \Big[ \sup_{0 \leq t \leq T}  (\mc{M}_t^\varphi)^2 \Big] = 0.\]
Whence,  
\begin{multline}\label{entIne2}
	\liminf_{M \rightarrow \infty} \, \mathbf{P}_{\tilde{\mu}_M} \bigg( \int_0^{\infty}  \frac{1}{M} \sum_{y \in\bb{Z}} \Big\{M \tgenzr \big|\omega_y (t)- \zeta_y (t)\big| \varphi_t(\tfrac{y}{M})+\\
	  \big|\omega_y (t)- \zeta_y (t)\big| \partial_t \varphi_t(\tfrac{y}{M})\Big\}\,dt \geq - \epsilon\Big)=1.
\end{multline}
Since 
\begin{multline*}
 \sum_{y \in\bb{Z}} \varphi_t(\tfrac{y}{M})  \tgenzr |\omega_y (t)- \zeta_y (t)|
\leq \sum_{y \in\bb{Z}} \varphi_t(\tfrac{y}{M}) \Big(p \big|\mathbbm{1}_{\{\omega_{y-1} (t) \geq 2\}} -\mathbbm{1}_{\{ \zeta_{y-1} (t) \geq 2\}}\big| \\
 + (1-p) \big|\mathbbm{1}_{\{\omega_{y+1} (t) \geq 2\}} -\mathbbm{1}_{\{ \zeta_{y+1} (t) \geq 2\}}\big| - \big|\mathbbm{1}_{\{\omega_y (t) \geq 2\}} - \mathbbm{1}_{\{\zeta_y (t) \geq 2\}}\big|  \Big)\\
= \frac{2p-1}{M} \sum_{y \in\bb{Z}} \partial_v \varphi_t(\tfrac{y}{M}) \big|\mathbbm{1}_{\{\omega_y (t) \geq 2\}} - \mathbbm{1}_{\{\zeta_y (t) \geq 2\}}\big|  + o_M (1), \vphantom{\Bigg(}
\end{multline*}
together with \eqref{entIne2}, we have proved \eqref{entIne1}.

\medskip

\noindent {\sc Step 2.} Since $\varphi$ is smooth, we can introduce the block averages into \eqref{entIne1} and obtain that
\begin{multline}\label{entIne3}
		\lim_{\ell \rightarrow \infty} \liminf_{M \rightarrow \infty} \, \mathbf{P}_{\tilde{\mu}_M} \bigg(\int_0^\infty \frac{1}{M} \sum_{y \in\bb{Z}} \Big\{(2\ell + 1)^{-1} \sum_{|y^\prime - y| \leq \ell}\big|\omega_{y^\prime} (t)- \zeta_{y^\prime} (t)\big| \partial_t \varphi_t(\tfrac{y}{M}) \\
		+ (2p-1) (2\ell + 1)^{-1} \sum_{|y^\prime - y| \leq \ell} \big|\mathbbm{1}_{\{\omega_{y^\prime}(t) \geq 2\}} - \mathbbm{1}_{\{\zeta_{y^\prime} (t) \geq 2\}}\big|  \partial_v \varphi_t(\tfrac{y}{M})\Big\}\,dt\geq - \epsilon\bigg)=1.
	\end{multline}
To go from \eqref{entIne3} to \eqref{entIne}, we only need to prove that for every integer $k > 0$ and every $T > 0$, \lj{
\begin{equation}\label{entIne4}
	\lim_{\ell \rightarrow \infty} \limsup_{M \rightarrow \infty} \mathbf{E}_{\tilde{\mu}_M} \bigg[ \int_0^T\hspace{-0.5em} \frac{1}{2kM+1}\hspace{-0.3em} \sum_{|y| \leq kM}\hspace{-0.3em}  \Big| \frac{1}{2\ell + 1}\hspace{-0.3em} \sum_{|y^\prime - y| \leq \ell}\hspace{-0.3em}\big|\omega_{y^\prime}(t)- \zeta_{y^\prime}(t)\big|  - \big|\omega_y^\ell (t) - c\big|\; \Big| dt\bigg] = 0,
\end{equation}
and
\begin{multline}\label{entIne5}
	\lim_{\ell \rightarrow \infty} \limsup_{M \rightarrow \infty} \mathbf{E}_{\tilde{\mu}_M}  \bigg[ \int_0^T \frac{1}{2kM+1} \sum_{|y| \leq kM}  \Big| \frac{1}{2\ell + 1} \sum_{|y^\prime - y| \leq \ell} \big|\mathbbm{1}_{\{\omega_{y^\prime}(t) \geq 2\}} - \mathbbm{1}_{\{\zeta_{y^\prime}(t) \geq 2\}}\big| \\
- \big|\mc{G}(\omega_y^\ell (t)) - \mc{G}(c)\big|\; \Big| dt\bigg] = 0.
\end{multline}}
We only prove \eqref{entIne5}, \eqref{entIne4} can be handled in the same way. Let $\tilde{S}_t$ be the semigroup associated to $\tgenzr$ and let
\[\bar{\mu}^M_T = \frac{1}{TM} \int_0^{TM}  \frac{1}{2kM+1} \sum_{|y| \leq kM} \tau_y \tilde{\mu}_M \tilde{S}_t dt.\]
Then we can rewrite \eqref{entIne5} as
\begin{equation}\label{entIne6}
	\lim_{\ell \rightarrow \infty} \limsup_{M \rightarrow \infty} \int \Big| \frac{1}{2\ell+1}\sum_{|y^\prime - y| \leq \ell} \big|\mathbbm{1}_{\{\omega_{y^\prime} \geq 2\}} - \mathbbm{1}_{\{\zeta_{y^\prime} \geq 2\}}\big| 
	- \big|\mc{G}(\omega_y^\ell) - \mc{G}(c)\big|\Big| \bar{\mu}_T^M (d\omega,d\zeta)= 0.
\end{equation}
Recall $\bar{\alpha} = \sup \alpha^{\rm ini} (v)+1$. Observe that the first marginal of $\bar{\mu}^M_T$ is stochastically bounded by $\mu_{\bar \alpha}^\star$, and the second marginal is $\mu_c^\star$.  Therefore, the sequence of the measures $\{\bar{\mu}^M_T,\,M \geq 1\}$ is tight. Denote by $\mathscr{A}$ the set of all possible limit points.  Then we have that
\begin{enumerate}
	\item each marginal of any $\tilde{\mu} \in \mathscr{A}$ is stochastically bounded by $\mu_{\bar{\alpha}+c}^\star,$
	\item $\mathscr{A} \subset \tilde{\mathscr{I}} \cup \tilde{\mathscr{S}}$ (where both sets were defined just before Lemma \ref{lem:ergodic}),
	\item the second marginal of any $\tilde{\mu} \in \mathscr{A}$ is $\mu_c^\star$.
\end{enumerate}
It therefore suffices to prove
\begin{equation}\label{entIne7}
	\lim_{\ell \rightarrow \infty} \sup_{\tilde{\mu} \in \mathscr{A}} \int \Big| \frac{1}{2\ell+1}\sum_{|y^\prime - y| \leq \ell} \big|\mathbbm{1}_{\{\omega_{y^\prime} \geq 2\}} - \mathbbm{1}_{\{\zeta_{y^\prime} \geq 2\}}\big| 
	- \big|\mc{G}(\omega_y^\ell) - \mc{G}(c)\big|\Big| \tilde{\mu} (d\omega,d\zeta)= 0.
\end{equation}
By Lemma \ref{lem:ergodic}, we need to consider two cases: either $\tilde{\mu}$ is supported on $\{0,1\}^\Z \times \{0,1\}^\Z$ or  $\tilde{\mu}$ satisfies
\[
	\tilde{\mu} (\omega \leq \zeta \;\text{or }\; \zeta \leq  \omega ) = 1.
\]
The second case follows exactly from the strategy of \cite[Theorem 3.1]{rezakhanlou91}. The main idea is to replace the average of absolute value in \eqref{entIne7} by the absolute value of average, and then to use the law of large numbers. For the first case we must have $c \leq 1$. In particular, every term inside the integral is equal to zero. This proves \eqref{entIne7} and concludes the proof of the lemma.
\end{proof}

\subsubsection{Initial condition} In this subsection, we assume that the initial distribution is a product measure on $\Gamma_M$ with marginals given by \eqref{eq:muM}. We now prove a microscopic version of the initial condition  for the asymmetric  zero-range process. 

\begin{lemma}\label{lem:zrpinitial}
For every $A > 0$,
\begin{equation}\label{eqn:initial_value}
	\lim_{t \rightarrow 0} \limsup_{\ell \rightarrow \infty} \limsup_{M \rightarrow \infty} \mathbf{E}_{\mu_M}  \bigg[\frac{1}{M} \sum_{|y| \leq AM} \big|\omega_y^\ell (t) - \alpha^{\rm ini}(\tfrac{y}{M})\big| \bigg] = 0.
\end{equation}
\end{lemma} 

The proof of \cite[Lemma 5.6]{rezakhanlou91} cannot be  adapted to our case directly because the initial distribution is not a local equilibrium state, so it is impossible to couple the original process with the stationary process. To overcome this difficulty, we couple the process with the one with initial value $\mu_\alpha$ instead of $\mu_\alpha^\star$.  Following the proof of  \cite[Lemma 5.6]{rezakhanlou91} step by step, it is enough to prove for every $A > 0$ and for every $t > 0$,
\[\limsup_{\ell \rightarrow \infty}\,\limsup_{M \rightarrow \infty}\, \frac{1}{M} \sum_{|y| \leq AM} \mathbf{E}_{\mu_{\alpha^{\rm ini}(y/M)}} \Big[ \big|\omega^{\ell}_0 (t)- \alpha^{\rm ini}(\tfrac{y}{M})\big|\Big] = 0.\]
By attractiveness, the mapping  \[\alpha \mapsto \mathbf{E}_{\mu_{\alpha}} \Big[ \big|\omega^{\ell}_0 (t)- \alpha \big|\Big]\]
is uniformly continuous in $\ell$ and $M$. By the dominated convergence Theorem, we only need to prove  that for every $\alpha > 0$ and for every $t > 0$,
\begin{equation}\label{eqn:flat_profile}
\limsup_{\ell \rightarrow \infty}\,\limsup_{M \rightarrow \infty}\, \mathbf{E}_{\mu_\alpha} \big[|\omega^\ell_0 (t) - \alpha|\big] = 0.
\end{equation}
We first prove \eqref{eqn:flat_profile} for $\alpha > 1$.  Let $S(t)$ be the semigroup associated to the infinitesimal generator $\genzr$.  Recall that $\omega (t)$ is the rescaled zero-range process with generator $M \genzr$. In the supercritical case, it is easy to see that \eqref{eqn:flat_profile} is a direct consequence of the following result, which has its own interest. Indeed, we could rewrite the expectation in \eqref{eqn:flat_profile} as $E_{\mu_{\alpha} S(Mt)} [ |\omega_0^\ell - \alpha|]$. Then the result follows from law of large numbers.

\begin{proposition}\label{pro:superlimit} If $\alpha > 1$, then
	\[\lim_{t \rightarrow \infty} \mu_\alpha S (t) = \mu_\alpha^\star.\]
\end{proposition}

\begin{proof}[Proof of Proposition \ref{pro:superlimit}]
We follow the ergodicity argument of \cite[Theorem 7.1]{seppalainen2008translation}, where a similar result is proved for asymmetric $K$-exclusion process in one dimension.   Only throughout the proof of Proposition \ref{pro:superlimit} and without confusion, we use $\omega (t)$ to denote the original zero-range process with \emph{unaccelerated}  generator $\genzr$.

Let $(\omega (t),\zeta (t))$ be the coupled process evolving according to the generator $\tgenzr$(see section \ref{sec:coupling}) and with initial distribution $\tilde{\mu}_{\alpha,\varepsilon}:=\mu_\alpha \otimes \mu_{\alpha-\varepsilon}^\star$,  where $\varepsilon > 0$ is fixed and small enough such that $\alpha - \varepsilon > 1$.  Denote by $\tilde{\mu}_t$ the distribution of the coupled process at time $t$.  To finish the proof, it is enough to show that
\begin{equation}\label{neg_discre_vanish}
	\lim_{t \rightarrow \infty} \E_{\tilde{\mu}_t} \big[(\omega_y - \zeta_y)^-\big] = 0, \quad \text{for any } y \in \Z.
\end{equation}
Indeed, this implies that any limit point of $\tilde{\mu}_t$ is concentrated on the set of configurations $\{(\omega,\zeta): \omega \geq \zeta\}$. As a consequence, any limit point of $\mu_\alpha S(t)$ as $t \rightarrow \infty$ is bounded below by $\mu_{\alpha-\varepsilon}^\star$. Since $\varepsilon$ is arbitrarily small, any limit point of $\mu_\alpha S(t)$ is stochastically bounded below by $\mu_{\alpha}^\star$. By taking the initial distribution of the coupled process $(\omega(t),\zeta(t))$ as $\mu_\alpha \otimes \mu_{\alpha+\varepsilon}^\star$,  $\varepsilon > 0$ and using a similar argument as above, we prove that any limit point of $\mu_\alpha S(t)$ is stochastically bounded above by $\mu_{\alpha}^\star$. This is enough to prove Proposition  \ref{pro:superlimit}.

We now prove \eqref{neg_discre_vanish}.  Since the initial measure of the coupled process is spatially translation invariant and ergodic, and the process could be constructed by using a family of independent Poisson processes,  $\tilde{\mu}_t$ is also spatially translation invariant and ergodic for every $t \geq 0$, \emph{cf.\,}\cite[Lemma 7.4]{seppalainen2008translation} for details. Moreover, it is easy to see that $\mathbb{E}_{\tilde{\mu}_t} [\omega_0]$ and $\mathbb{E}_{\tilde{\mu}_t} [\zeta_0]$ are constant in time \lj{by translation invariance of the initial distribution and of the dynamics}, and that $\mathbb{E}_{\tilde{\mu}_t} [(\omega_0 - \zeta_0)^\pm]$ and $\mathbb{E}_{\tilde{\mu}_t} [|\omega_0 - \zeta_0|]$ are non-increasing in time.

Assume that \eqref{neg_discre_vanish} does not hold. We must then have for some $\delta>0$, and for any $t\geq 0$
\[{\bb E}_{\tilde{\mu}_t} [(\omega_0 - \zeta_0)^-] \geq \delta, \quad  \mbox{ and }\quad  {\bb E}_{\tilde{\mu}_t} [(\omega_0 - \zeta_0)^+]= \varepsilon +{\bb E}_{\tilde{\mu}_t} [(\omega_0 - \zeta_0)^-] \geq \delta.\]
We claim that $\tilde{\mu}_t (\omega_0 - \zeta_0>0)$ and $\tilde{\mu}_t (\omega_0 - \zeta_0<0)$ are both larger than some $\delta'>0$ for any time $t\geq 0$.  Indeed, since $(\omega_0 - \zeta_0)^+$ is nonnegative and integer-valued, by Cauchy-Schwarz inequality and attractiveness,
\[\tilde{\mu}_t (\omega_0 - \zeta_0>0) = \tilde{\mu}_t ((\omega_0 - \zeta_0)^+ \geq 1) \geq \frac{{\bb E}_{\tilde{\mu}_t} [(\omega_0 - \zeta_0)^+]}{\sqrt{{\bb E}_{\tilde{\mu}_t} \big[\big((\omega_0 - \zeta_0)^+\big)^2\big]}} \geq C(\alpha) \delta =: \delta^\prime.\]
The same argument works for $\tilde{\mu}_t (\omega_0 - \zeta_0<0)$ and this proves the claim. In particular, by ergodicity, any limit point of $\mu_\alpha S(t)$ is concentrated on configurations with densities of -both positive and negative- discrepancies which are larger than $\delta'$ (see \cite[Proposition 7.8]{seppalainen2008translation} for a more detailed argument).
Let $I$ be a finite interval, and $B(I)$ the set of configurations such that $I$ contains discrepancies of opposite sign, namely 
\[B(I)=\{(\omega, \zeta) \; ;\; \exists y, y'\in I \mbox{ such that }  \omega_y - \zeta_y>0\mbox{ and }\omega_{y'} - \zeta_{y'}<0 \}. \]
If a limit point of $\mu_\alpha S(t)$ is concentrated on configurations with both positive and negative discrepancy densities, we have 
\begin{equation}
\label{eq:BI}
\lim_{\ell\to\infty}\lim_{t \rightarrow \infty} \tilde{\mu}_t (B(\{-\ell,\dots,\ell\}))=1.
\end{equation}

We now claim that for any finite interval $I$,
\begin{equation}
\label{eqn:mu_tB}
	\lim_{t \rightarrow \infty} \tilde{\mu}_t (B(I)) = 0,
\end{equation}
which disproves \eqref{eq:BI} and proves \eqref{neg_discre_vanish}.

%
%
It remains to prove \eqref{eqn:mu_tB}, for which Sepp{\" a}l{\" a}inen's argument needs to be refined, because of the degeneracy of the dynamics. By translation invariance, we may take $I = \{0,1,\ldots,m-1\}$ for some positive integer $m$. Fix $T > 0$ and $K > 0$.  For any time $t > 0$, let $A_t:=A_t(I)$ be the following event: \ccl{during the time interval $[t,t+T)$, there are no jumps involving the sites in the interval $I$ except the following jumps: first, there are $q:=m(K-1)$ jumps from site $0$ to $1$, then $q$ jumps from $1$ to $2$, $\ldots$ up to $q$ jumps from $m-2$ to $m-1$. Note that
	\[p_{K,T} := \mathbf{P}_{\tilde{\mu}_{\alpha,\varepsilon}} (A_t(I)) > 0,\]
	and also that $A_t(I)$ is independent of the initial distribution $\tilde{\mu}_{\alpha,\varepsilon}$ of the two configurations since it  depends only on the Poisson jumps in the time interval $[t,t+T)$. Assume that the configuration at time $t$ has $1\leq \omega_y(t)\leq K$ particles on each site in $I$ (so that it has at most $q$ particles that can move in $I$), then on the event $A_t(I)$, all the particles that were in $I$ at time $t$ end up on site $m-1$ at time $T+t$, except one particle per site that remained stuck. In particular, as illustrated in Figure \ref{fig:At}, if two configurations are driven by the basic coupling and both satisfy  $1\leq \omega_y(t), \zeta_y(t)\leq K$  in $I$, then at time $t+T$, only remains at most one type of discrepancy, depending on which configurations initially had more particles in $I$.}
	
\begin{figure}
\includegraphics[width=10cm]{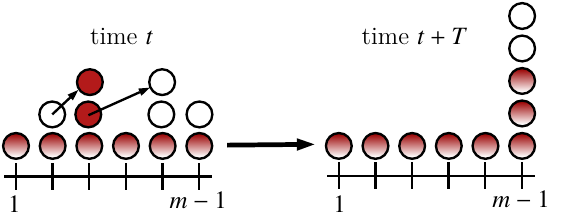}
\caption{\ccl{Representation of the event $A(t)$ for two configurations $\omega$ (white) and $\zeta$ (red), under the basic coupling. The particles with the red-white gradient are the coupled ones, the arrows represent couplings that occur before time $t+T$.}}
\label{fig:At}
\end{figure}

	Let $z_\ell \uparrow \infty$ be an increasing sequence of points such that there are no interactions of the system in the interval $\{-z_\ell,\dots, z_\ell\}$  with the outside $\Z\setminus \{-z_\ell,\dots, z_\ell\}$ during time interval $[t,t+T)$. 
	Denote by $B_t (I)$ the event that  there exist discrepancies of opposite sign at time $t$ in the interval $I$. For any $j\in \Z$, define $I_{j}=I+jm$. If  $1 \leq \omega_y (t),\,\zeta_y (t) \leq K$ for every $y \in I_j$, and the events $B_t (I_j)$ and $A_t (I_j)$ happen, then at least two discrepancies are annihilated in the interval $I_j$ during the time interval $[t,t+T)$, \ccl{as illustrated in Figure \ref{fig:At}}. In particular, \lj{defining $j_\ell:=\max\{j\in \N\; ;\; j m<z_\ell\}$,}
	\begin{multline*}
		\frac{1}{z_\ell} \sum_{y=-z_\ell}^{z_\ell} |\omega_y (t+T)- \zeta_y (t+T)| \leq  \frac{1}{z_\ell} \sum_{y=-z_\ell}^{z_\ell} |\omega_y (t)- \zeta_y (t)|   \\
		-  \frac{2}{z_\ell} \sum_{j=-j_\ell}^{j_\ell} \mathbbm{1} \{A_t (I_j)\}
		\times  \mathbbm{1} \big\{\{1 \leq \omega_y(t),\,\zeta_y (t) \leq K \;  \text{for every}\; y \in I_j\} \cap B_t (I_j) \big\}.
	\end{multline*}
	We remark that the two indicator functions above are independent, since the first one only depends on the Poisson clocks between times $t$ and $t+T$, whereas the second one only depends on the configuration at time $t$. Letting $z_\ell \uparrow \infty$, by ergodicity, we obtain 
	
	\vspace{-0.5cm}
	
	\begin{multline*}
	\mathbf{E}_{\tilde{\mu}_{\alpha,\varepsilon}} [ |\omega_0 (t+T)- \zeta_0 (t+T)|]\\
	 \leq \mathbf{E}_{\tilde{\mu}_{\alpha,\varepsilon}} [ |\omega_0 (t)- \zeta_0 (t)|] - 2 m^{-1} p_{K,T} \mathbf{P}_{\tilde{\mu}_{\alpha,\varepsilon}} \big(B_t (I) \cap \{1 \leq \omega_y (t),\zeta_y(t) \leq K, \;  \forall \; y \in I\}\big).
	 \end{multline*}
	Since $\mathbf{E}_{\tilde{\mu}_{\alpha,\varepsilon}}[ |\omega_0 (t)- \zeta_0 (t)|] $ only decreases in time and remains non-negative, we must have
	\[\limsup_{t \rightarrow \infty} \mathbf{P}_{\tilde{\mu}_{\alpha,\varepsilon}} \big(B_t (I) \cap \{1 \leq  \omega_y (t),\zeta_y(t) \leq K \;  \text{for every}\; y\in I\}\big) = 0.\]
	Therefore,
	\begin{align*}
		\limsup_{t \rightarrow \infty} \tilde{\mu}_t (B) &\leq \limsup_{t \rightarrow \infty}  \sum_{y \in I} \left\{ \mathbf{P}_{\tilde{\mu}_{\alpha,\varepsilon}} (\omega_y (t)= 0) + \mathbf{P}_{\tilde{\mu}_{\alpha,\varepsilon}} (\omega_y (t) > K) + \mathbf{P}_{\tilde{\mu}_{\alpha,\varepsilon}} (\zeta_y (t) > K) \right\}\\
&		\leq \limsup_{t \rightarrow \infty}  \sum_{y \in I}  \mathbf{P}_{\mu_\alpha} (\omega_y (t) = 0) + \frac{2 m \alpha}{K}.
	\end{align*}
	Since $K$ is arbitrary, and by translation invariance, it suffices to prove that 
	\begin{equation}
	\label{eqn:no_holes}
	\lim_{t \rightarrow \infty}  \mathbf{P}_{\mu_\alpha} (\omega_0 (t) = 0) = 0.
	\end{equation}
	Since $\mathbbm{1} \{\omega_y (t)=0\}$ is non-negative and decreasing  in time $t$, it has a limit denoted by $a_y$ as $t \rightarrow \infty$.  Furthermore, a.s. there cannot exist $y< y'\in \Z$ such that $a_y=a_y'=1$ and 
	\begin{equation}\sum_{z=y+1}^{y'-1}\omega_y(0)> y'-y-1.\end{equation}
Indeed, if the latter holds, then at some point a particle must have left the interval $\{y,\dots,y'\}$ and therefore we cannot have $a_y=a_y'=1$. Since the density $ \alpha>1$ is larger than one, by ergodicity a.s. there cannot exist an infinite sequence $(y_k)$ of sites such that $a_{y_k}=1$ $\forall k\in \N$. By ergodicity we must have in particular that $\mathbf{P}_{\mu_\alpha} (a_0 = 0)=1$.
This yields	
\[\mathbf{E}_{\mu_\alpha} [a_0] = \lim_{t \rightarrow \infty}  \mathbf{P}_{\mu_\alpha} (\omega_0 (t) = 0) = 0\]
which proves \eqref{eqn:no_holes} and concludes the proof.\end{proof}


We now prove \eqref{eqn:flat_profile} for $\alpha < 1$. By ergodicity, 
	\[\mathbf{P}_{\mu_\alpha} (\omega_0 (t) = 0) = \lim_{n \rightarrow \infty} \frac{1}{n} \sum_{y=1}^n \mathbbm{1} \{\omega_y (t)= 0\} \geq \lim_{n \rightarrow \infty} \frac{1}{n} \sum_{x=1}^n  (1-\omega_y (t)) = 1 - \alpha > 0, \]
	$ \mathbf{P}_{\mu_\alpha}- \text{a.s.}$ Recall that $a_y$ is the limit of $\mathbbm{1} \{\omega_y (t) = 0\}$ as $t \rightarrow \infty$, which satisfies
	\[\mathbf{E}_{\mu_\alpha} [a_0] \geq 1 - \alpha.\]
	Let $A_\ell$ be the event that there exist points $y \in \{-\ell - \sqrt{\ell},\dots, -\ell\}$ and $y^\prime \in  \{\ell ,\dots, \ell + \sqrt{\ell}\}$ such that $a_y= a_{y^\prime} = 1$. Note that on the event $A_\ell$, the total number of particles in the interval $\{y,\dots,y^\prime\}$ remains constant in time, therefore 
	\[\omega_0^\ell (t) = \omega_0^\ell (0) + \mathcal{O} (\ell^{-1/2}).\]
	Moreover,  following the same argument as in the case $\alpha>1$, since $\alpha<1$ we must have by ergodicity $\mathbf{P}_{\mu_\alpha} (A_\ell^c) \rightarrow 0$ as $\ell \rightarrow \infty$. Therefore, for any $\varepsilon > 0$, 
	\begin{multline*}
	\mathbf{P}_{\mu_\alpha} \big(|\omega_0^\ell (t) - \alpha| > \varepsilon\big)\\
	 \leq \mathbf{P}_{\mu_\alpha} \big(|\omega_0^\ell (t)- \omega_0^\ell (0) | >  \varepsilon/2, A_\ell\big) + \mathbf{P}_{\mu_\alpha} \big(|\omega_0^\ell (0) - \alpha| > \varepsilon/2\big) + \mathbf{P}_{\mu_\alpha} (A_\ell^c).
	 \end{multline*}
	This proves convergence in probability. By attractiveness, it is easy to see that $|\omega_0^\ell (t) - \rho| $ is uniformly integrable in $t$ and $\ell$, therefore $L^1$-convergence follows.
	
	For $\alpha = 1$, denote by $\omega^\alpha (t)$ the process with initial distribution $\mu_\alpha$. For any $\varepsilon > 0$, by attractiveness,
	\[\mathbf{P} (|\omega^{1,\ell}_0 (t) - 1| > \varepsilon) \leq \mathbf{P} (\omega^{1+\varepsilon/2,\ell}_0 (t)> 1+\varepsilon) + \mathbf{P} (\omega^{1-\varepsilon/2,\ell}_0 (t)< 1-\varepsilon).\]
	The above two probabilities on the right-hand side converge to zero as we have proved. By uniform integrability, we conclude the proof.

\subsubsection{Conclusion in the asymmetric case.} We now conclude the proof for the asymmetric case.  The steps are quite standard and we refer the readers to \cite{rezakhanlou91} or \cite[Chapter 8]{klscaling} for details of the proof. The main idea is to introduce the notions of measure-valued entropy solutions to \eqref{zrp:Hydro} and of Young measures, \emph{cf.\,}\cite{diperna1985measure} and \cite[Chapter 8]{klscaling} for such notions. Let $\mc{P} (\R_+)$ be the set of positive Radon measures on $\R_+$. A measurable map $\mu: (0,\infty) \times \R \rightarrow \mc{P} (\R_+)$ is said to be a measure-valued entropy solution to \eqref{zrp:Hydro} if 
\begin{enumerate}
	\item for any non-negative  test function  $\varphi\in C^{1,1}(\R_+\times\R)$ with compact support in $(0,\infty) \times \bb{R}$,  for any $c \geq 0$,
	\[	\int_0^\infty \,\int_{\bb{R}}\, \Big\{ \< \mu(t,v),\big|\alpha - c\big|\> \partial_t \varphi_t(v) + (2p-1)\<\mu(t,v), \mathfrak{q}\big(\alpha;c\big)\> \partial_v \varphi_t(v) \Big\}\,dv\,dt \geq 0,\]
	recall $\mathfrak{q}(\alpha;c)=\mathrm{sign}(\alpha-c)\big(\mathcal{G}(\alpha)-\mathcal{G}(c)\big)$;
	\item and for any $A > 0$,
	\[	\lim_{t \rightarrow 0}\, \int_{-A}^A\, \<\mu(t,v), |\alpha - \alpha^{\mathrm{ini}} (v)|\>\,dv = 0.\]
\end{enumerate}
Above, for a measurable function $f: \R_+ \rightarrow \R$,
\[\<\mu(t,v),f\> = \int_{\R_+} f(\alpha) \mu(t,v) (d\alpha).\]
For positive integers $M$ and $\ell$ and for each $t \geq 0$, define the {\it Young measure} $\pi_t^{M,\ell} (dv,d\alpha)$ as 
\[\pi_t^{M,\ell} (dv,d\alpha) = \frac{1}{M}\sum_{y \in \Z} \delta_{y/M} (dv) \delta_{\omega^\ell_y (t)} (d \alpha).\]
By Lemmas \ref{lem:MicroEntroIneqn} and \ref{lem:zrpinitial}, we could show that any weak limit of $\pi_t^{M,\ell}$ is concentrated on measure-valued entropy solutions to \eqref{zrp:Hydro} with initial value $\alpha^{\rm ini}$.  Moreover, since the process is attractive, it is easy to see that the limiting measure-valued entropy solutions are of Dirac type, \emph{cf.\,}\cite[Theorem 8.1.1]{klscaling} for example. For every smooth test function $\varphi$ with compact support in $\R$, note that
\[\frac{1}{M} \sum_{y \in \Z} \omega_y (t) \varphi(\tfrac{y}{M}) = \<\pi_t^{M,\ell},\varphi \alpha\> + \mathcal{O} (\ell / M).\]
This permits us to show that any weak limit of $\pi^M_t (dv)$ is concentrated on entropy  solutions to \eqref{zrp:Hydro} with initial value $\alpha^{\rm ini}$. We conclude the proof since the entropy solution is unique.

\bibliographystyle{imsart-number} 
\bibliography{bibliography}       


\end{document}